\documentclass[11pt,reqno]{amsart}
\usepackage{amssymb,amsmath,amsfonts,amsthm,enumerate,stmaryrd}
\usepackage{mathrsfs}
\usepackage{graphicx}

\usepackage[left=.75 in, right=.75 in,top=.75 in, bottom=.75 in]{geometry}
\usepackage{nameref,hyperref,cleveref}
\numberwithin{equation}{section}

\usepackage{color}

\providecommand{\abs}[1]{\left\vert#1\right\vert}
\providecommand{\norm}[1]{\left\Vert#1\right\Vert}
\providecommand{\snorm}[1]{\left[#1\right]}

\providecommand{\ip}[1]{\left(#1 \right)}

\def\nab{\nabla}
\def\dt{\partial_t}
\def\hal{\frac{1}{2}}
\def\ep{\varepsilon}
\def\vchi{\text{\large{$\chi$}}}
\def\ls{\lesssim}
\def\p{\partial}

\def\naba{\nab_{\mathcal{A}}}

\def\H1{{_0}H^1((-\ell,\ell))}

\providecommand{\abs}[1]{\left\vert#1\right\vert}
\providecommand{\norm}[1]{\left\Vert#1\right\Vert}

\providecommand{\br}[1]{\langle #1 \rangle}

\def\hal{\frac{1}{2}}
\def\ep{\varepsilon}
\def\vchi{\text{\large{$\chi$}}}
\def\ls{\lesssim}

\def\nab{\nabla}
\def\dt{\partial_t}
\def\p{\partial}

\def\naba{\nab_{\mathcal{A}}}

\def\SH0{\mathcal{H}^0((-\ell,\ell))}

\def\H{X}
\def\P{Y}

\def\A{\mathcal{A}}
\def\B{\mathcal{B}}
\def\C{\mathbb{C}}
\def\D{\mathcal{D}}

\def\H{\mathcal{H}}
\def\h{\mathcal{H}}

\def\L{\mathcal{L}}

\def\M{\mathcal{M}}
\def\N{\mathbb{N}}
\def\P{\mathcal{P}}

\def\R{\mathbb{R}}
\def\T{\mathbb{T}}

\def\Z{\mathbb{Z}}

\def\ef{\mathfrak{E}}

\def\jf{\mathfrak{J}}
\def\ff{\mathfrak{F}}

\def\kf{\mathfrak{K}}

\def\pf{\mathfrak{P}}
\def\sf{\mathfrak{S}}

\def\st{\;\vert\;}

\def\XXint#1#2#3{{\setbox0=\hbox{$#1{#2#3}{\int}$ }
\vcenter{\hbox{$#2#3$ }}\kern-.6\wd0}}

\DeclareMathOperator{\tr}{Tr}

\DeclareMathOperator{\diverge}{div}
\DeclareMathOperator{\di}{div}
\DeclareMathOperator{\supp}{supp}

\DeclareMathOperator{\sech}{sech}

\def\Zz{\mathbb{Z}}

\def\P{\mathbb{P}}

\def\cB{\mathcal{B}}

\def\cF{\mathcal{F}}
\def\cG{\mathcal{G}}

\def\cN{\mathcal{N}}

\def\cT{\mathcal{T}}
\def\cL{\mathcal{L}}

\def\cH{\mathcal{H}}

\def\ld{\lambda}

\def\wt{\widetilde}
\def\wL{\widetilde{L}}
\def\wh{\widehat}

\newtheorem{lem}{Lemma}[section]
\newtheorem{cor}[lem]{Corollary}
\newtheorem{prop}[lem]{Proposition}
\newtheorem{thm}[lem]{Theorem}
\newtheorem{remark}[lem]{Remark}

\newtheorem{dfn}[lem]{Definition}

\title{Traveling wave solutions to the one-phase Muskat problem: existence and stability}

\author{Huy Q. Nguyen}
\address{
Department of Mathematics\\
University of Maryland\\
College Park, MD 20742, USA
}
\email[H. Nguyen]{hnguye90@umd.edu}

\author{Ian Tice}
\address{
Department of Mathematical Sciences\\
Carnegie Mellon University\\
Pittsburgh, PA 15213, USA
}
\email[I. Tice]{iantice@andrew.cmu.edu}



\begin{document}

\begin{abstract}
We study the Muskat problem for one fluid in arbitrary dimension, bounded below by a flat bed and above by a free boundary given as a graph. In addition to a fixed uniform gravitational field, the fluid is acted upon by a generic force field in the bulk and an external pressure on the free boundary, both of which are posited to be in traveling wave form.  We prove that for sufficiently small force and pressure data in Sobolev spaces, there exists a locally unique traveling wave solution in Sobolev-type spaces. The free boundary of the traveling wave solutions is either periodic or asymptotically flat at spatial infinity. Moreover, we prove that small periodic traveling wave solutions induced by external pressure only are asymptotically stable. These results provide the first class of nontrivial stable solutions for the problem.  
\end{abstract}

\maketitle

\section{Introduction} 
In this paper we study traveling wave solutions to the one-phase Muskat problem, which concerns the dynamics of the free boundary of a viscous fluid in  homogeneously permeable porous media. The n-dimensional ($n\ge 2$) wet region $\Omega_{\zeta(\cdot,t)}$ lies above the flat bed of depth $b>0$ and below the free boundary that is the graph of an unknown time-dependent function $\zeta$, i.e.
\begin{equation}
\Omega_{\zeta(\cdot,t)} = \{x \in \Gamma \times \R \st -b < x_n < \zeta(x',t)\},
\end{equation}
where the cross-section $\Gamma$ is either $\R^{n-1}$ or $\T^{n-1} := \R^{n-1} / (2\pi \Z^{n-1})$.  Here, for any point $x\in \Gamma \times \R$ we split its horizontal and vertical coordinates as $x=(x', x_n)$. We denote the free-boundary and the flat bed respectively by $\Sigma_{\zeta(\cdot,t)}$ and $\Sigma_{-b}$; that is,
\begin{equation}
 \Sigma_{\zeta(\cdot,t)} = \{x \in \Gamma \times \R \st x_n = \zeta(x',t)\} \text{ and } \Sigma_{-b} = \{x \in \Gamma\times \R \st x_n = -b\}.
\end{equation}
We posit that when the cross-section is $\R^{n-1}$, the free boundary $\zeta(x', t)$ decays as $|x'|\to \infty$. 

The fluid is acted upon in the bulk by a uniform gravitational field $-e_n$ pointing downward, where $e_n$ is the upward pointing unit vector in the vertical direction, and a generic body force $\tilde{\mathfrak{f}}(\cdot, t): \Omega_{\zeta(\cdot, t)}\to \R^n$.  Then the fluid motion in the porous medium is modeled by the Darcy  law 
\begin{equation}
 w + \nab P = - e_n + \tilde{\mathfrak{f}} \quad  \text{and} \quad \diverge{w} = 0\quad \text{in } \Omega_{\zeta(\cdot,t)},
\end{equation}
where for the sake of simplicity we have normalized the dynamic viscosity, the fluid density, and the permeability of the medium  to unity. Here, $w$ and $P$ respectively denote the fluid velocity and pressure. On the surface, the fluid is acted upon by a constant pressure $P_0$ from the dry region above $\Omega_{\zeta(\cdot,t)}$ and an externally applied pressure $\phi(\cdot, t): \Sigma_{\zeta(\cdot, t)}\to \R$. This leads to the boundary condition 
\begin{equation}
P=P_0+\phi\quad\text{on } \Sigma_{\zeta(\cdot,t)}. 
\end{equation}
The no-penetration boundary condition is assumed on the flat bed:
\begin{equation}
w_n=0\quad\text{on }\Sigma_{-b}.
\end{equation}
Finally, the free boundary evolves according to the kinematic boundary condition 
\begin{equation}
 \dt \zeta = w \cdot (-\nab' \zeta,1) \quad\text{on } \Sigma_{\zeta(\cdot,t)}.
\end{equation}
We shall refer to the following system as the (one-phase) Muskat problem 
\begin{equation}\label{Muskat:wpzeta}
\begin{cases}
 w + \nab P = - e_n + \tilde{\mathfrak{f}}  & \text{in }\Omega_{\zeta(\cdot,t)} \\
 \diverge{w} = 0 &\text{in } \Omega_{\zeta(\cdot,t)} \\
 \dt \zeta = w \cdot (-\nab' \zeta,1) &\text{on } \Sigma_{\zeta(\cdot,t)} \\
 P = P_0 + \phi &\text{on } \Sigma_{\zeta(\cdot,t)} \\
 w_n =0 & \text{on } \Sigma_{-b}.
\end{cases}
\end{equation}
In the absence of the body force and the external pressure, i.e. $\tilde{\mathfrak{f}} =0$ and $\phi=0$, \eqref{Muskat:wpzeta} is called the {\it free} Muskat problem. 

The free Muskat problem can be recast as a nonlocal equation for the free boundary function $\eta$ (see \eqref{Muskat}). It was proved in \cite{NguyenPausader} that the problem is locally-in-time well-posed for large data $\eta_0\in H^s(\Gamma)$ for any $1+\frac{n-1}{2}<s\in \R$, which is the lowest Sobolev index guaranteeing that $\eta_0\in W^{1, \infty}(\Gamma)$. We also refer to \cite{CCG} for local well-posedness for the case of non-graph free boundary. The free Muskat problem admits the following trivial stationary solutions 
\begin{equation}\label{stationarysol}
(\overline{w}, \overline{P}, \overline{\zeta})=
\begin{cases}
(0, -x_n+P_0, 0)\quad \text{if } \Gamma=\R^{n-1},\\
(0, -x_n + c+P_0, c),~c\in \R\quad \text{if } \Gamma=\T^{n-1}.
\end{cases}
\end{equation}
In fact, under mild regularity and decay assumptions, \eqref{stationarysol} are the only stationary solutions. They have been proved to be stable in various norms \cite{CGBS, GGJPS, Nguyen2022,SieCafHow}. To the best of our knowledge, \eqref{stationarysol} are the only solutions that are known to be stable.

In this paper  we are interested in the construction of nontrivial special solutions and the stability of them. In view of the translation invariance of \eqref{Muskat:wpzeta} in the horizontal directions, it is natural to consider traveling wave solutions. These are solutions that propagate along a fixed direction, which without loss of generality we may assume is the  $x_1$-direction, with constant velocity $\gamma$.  To this end, we assume that
\begin{equation}
 \phi(x,t) = \varphi(x-\gamma t e_1) \text{ and } \tilde{\mathfrak{f}}(x,t) = \mathfrak{f}(x-\gamma t e_1)
\end{equation}
and make the traveling wave ansatz 
\begin{equation}
 \zeta(x',t) = \eta(x' - \gamma t e_1 ).
\end{equation}
This determines the unknown domain $\Omega_\eta = \{x \in \Gamma \times \R \st -b < x_n < \eta(x')\}$  as well as the free boundary $\Sigma_\eta = \{x \in \Gamma \times \R \st x_n = \eta(x')\}$ as before.  We then define the traveling wave unknowns $v : \Omega_\eta \to \R^n$ and $q: \Omega_\eta \to \R$ via 
\begin{equation}
 w(x,t) = v(x-\gamma t e_1 ) \text{ and } P(x,t) = P_0 - x_n + q(x - \gamma t e_1 ).
\end{equation}
In the latter we have subtracted off the hydrostatic pressure, as is often convenient.  The new equations for $(v,q, \eta)$ read
\begin{equation}\label{muskat_euler:intro}
\begin{cases}
 v + \nab q = \mathfrak{f} & \text{in }\Omega_{\eta} \\
 \diverge{v} = 0 &\text{in } \Omega_{\eta} \\
 -\gamma \p_1 \eta = v \cdot N&\text{on } \Sigma_{\eta} \\
 q - \eta =   \varphi &\text{on } \Sigma_{\eta} \\
 v_n =0 & \text{on } \Sigma_{-b},
\end{cases}
\end{equation}
where
\begin{equation}
N=(-\nab'\eta,1).
\end{equation}

We pause to remark that the only solutions to the free version of \eqref{muskat_euler:intro} ($ \mathfrak{f}=0$ and $\phi=0$) are the trivial solutions as given in \eqref{stationarysol}. Indeed, assuming that $(v, q, \eta)$ is a decaying solution, then using Green's theorem and the boundary conditions for $v$ and $q$, we obtain 
\begin{equation}
\int_{\Omega_\eta}|v|^2dx =-\int_{\Omega_\eta} v\cdot \nab q dx\\
= -\int_{\Sigma_\eta} q(v\cdot \frac{N}{|N|})dS-\int_{\Sigma_{-b}}qv_n dS+\int_{\Omega_\eta}  q\diverge{v}dx\\
=\gamma\int_{\Gamma}\eta\p_1\eta dx=0.
\end{equation}
It follows that $v=0$ and hence $q=c$, a constant. Consequently $\eta(x')=q(x', \eta(x'))=c$.  When $\Gamma = \R^{n-1}$ this implies that $\eta=q=0$ since  $\eta$ decays.  Thus $(v, q, \eta)=(0, 0, 0)$ is the trivial solution when $\Gamma = \R^{n-1}$. In the periodic case, $\Gamma = \T^{n-1}$, we obtain the trivial solutions $(v, q, \eta)=(0, c, c)$, for $c\in \R$ (one can uniquely determine $c$ by fixing a mass of the fluid).  This is not a surprise since for free Muskat problem, the energy dissipates, so it cannot sustain the permanent structure of traveling waves. It is therefore necessary to have some sort of external energy in order for traveling wave solutions to exist.  In the context of \eqref{muskat_euler:intro}, this is provided by the external bulk force $ \mathfrak{f} $ and the external pressure $\phi$.  

Our first main result states  that for suitably small $\mathfrak{f}$ and $\phi$, there exists a locally unique traveling wave solution to \eqref{muskat_euler:intro} in Sobolev-type spaces. Note that in the following statement we employ a reformulation of the problem \eqref{muskat_euler:intro} as well as some nonstandard function spaces; these will be explained after the theorem statement.

\begin{thm}(Proved in Section \ref{sec_tw_exist})\label{well_posedness_flat:intro}
Let $\frac{n}{2} - 1 < s \in \N$  and consider the open set
\begin{equation}
 U^s_\delta = \{(u,p,\eta) \in {_n}H^{s+1}(\Omega;\R^n) \times H^{s+2}(\Omega) \times \H^{s+3/2}(\Sigma) \st \norm{\eta}_{\H^{s+3/2}} < \delta \}
\end{equation}
with $\delta >0$ as  constructed in Theorem \ref{nlin_well_def}.  Define the open set $\mathfrak{C}\subseteq \R$ to be $\R$ if $\Gamma  = \T^{n-1}$ and $\R \backslash \{0\}$ if $\Gamma = \R^{n-1}$.  Then there exist open sets 
\begin{equation} 
\mathcal{D}^s \subseteq \mathfrak{C} \times H^{s+3/2}(\Sigma) \times H^{s+3}(\Gamma \times \R) \times H^{s+1}(\Omega;\R^n) \times H^{s+2}(\Gamma \times \R;\R^n)
\text{ and } \mathcal{S}^s \subseteq U^s_\delta
\end{equation}
such that the following hold.
\begin{enumerate}
 \item $\mathfrak{C} \times \{0\} \times  \{0\} \times  \{0\} \times  \{0\} \subseteq \mathcal{D}^s$ and  $(0,0,0) \in \mathcal{S}^s$.
 \item For each $(\gamma, \varphi_0, \varphi_1, \mathfrak{f}_0, \mathfrak{f}_1) \in \mathcal{D}^s$ there exists a locally unique $(u,p,\eta) \in \mathcal{S}^s$ classically solving 
\begin{equation}\label{well_posedness_flat_02:intro}
\begin{cases}
 u + \nab_{\A} p + \nab_{\A} \pf \eta = \jf \M^T[ \mathfrak{f}_0 +  \mathfrak{f}_1\circ \ff_\eta] & \text{in }\Omega:=\Gamma\times (-b, 0) \\
 \diverge{u} = 0 &\text{in } \Omega \\
 -\gamma \p_1 \eta = u _n &\text{on } \Sigma:=\Gamma\times \{0\} \\
 p  = \varphi_0 +  \varphi_1 \circ \ff_\eta &\text{on } \Sigma \\
 u_n =0 & \text{on } \Sigma_{-b}.
\end{cases}
\end{equation}
 \item The map  $\mathcal{D}^s \ni (\gamma, \varphi_0, \varphi_1, \mathfrak{f}_0, \mathfrak{f}_1) \mapsto (u,p,\eta) \in \mathcal{S}^s$ is $C^1$ and locally Lipschitz.
\end{enumerate}
\end{thm}

Some remarks are in order:

\begin{enumerate}
\item \eqref{well_posedness_flat_02:intro} is a reformulation of \eqref{muskat_euler:intro} in the fixed domain $\Omega$ and with $\mathfrak{f}(x)=\mathfrak{f}_0(x')+\mathfrak{f}_1(x)$ and $\varphi(x)=\varphi_0(x')+\varphi(x)$.  See Section \ref{subsection:flattening} for the derivation of \eqref{well_posedness_flat_02:intro} and for the precise meaning of  $\mathcal{A}$,  $\mathfrak{J}$, $\mathcal{M}$ and $\mathfrak{F}_\eta$. The preceding forms of $\mathfrak{f}$ and $\varphi$ are imposed since we assume less regularity for $\mathfrak{f}$ and $\varphi$ when they are independent of the vertical variable $x_n$. Note also that the integer constraint for the regularity parameter $s$ comes from the need to verify that the maps $(\mathfrak{f}_1,\eta) \mapsto \mathfrak{f}_1 \circ \mathfrak{F}_\eta$ and $(\varphi_1,\eta) \mapsto \varphi_1 \circ \mathfrak{F}_\eta$  are $C^1$.  If these forcing terms are ignored, then we may relax this requirement for $s$: see Theorems \ref{theo:existence:finite} and \ref{theo:existence:finite:p}.

\item The space ${_n}H^{s+1}(\Omega;\R^n)$ is defined in Definition \ref{nHs_def}.

\item  When the cross-section is $\T^{n-1}$, the boundary function $\eta$ is constructed in  $\mathring{H}^s(\T^{n-1})$, the usual Sobolev space of zero-mean functions. On the other hand, when the cross-section is $\R^{n-1}$,  then $\eta$ belongs to the anisotropic Sobolev space  $\cH^{s+\frac32}(\R^{n-1})$, as defined in Definition \ref{specialized_dfn}.  At high frequencies this space provides standard $H^{s+3/2}$ Sobolev control, but at low frequencies it only controls
\begin{equation}
\int_{B(0,1)} \frac{\xi_1^2+|\xi|^4}{|\xi|^2}|\widehat{\eta}(\xi)|^2d\xi.
\end{equation}
The modulus  $\frac{\xi_1^2+|\xi|^4}{|\xi|^2}$ naturally arises from the linearized operator $\p_1-|D|\tanh(b|D|)$ and the following  structure of the nonlinearity at low frequencies: $\mathcal{N}=|D|\widetilde{\mathcal{N}}$. The anisotropic Sobolev space $\cH^s(\R^d)$, which satisfies the inclusions $H^s(\R^d) \subset \cH^s(\R^d) \subseteq H^s(\R^d) + C^\infty_0(\R^d)$, was introduced  in \cite{LeoniTice} for the construction of traveling wave solutions to the free boundary Navier-Stokes equations and plays a key role in our construction here. We recall the definition and basic properties of $\cH^s$ in Appendix \ref{appendixA1}. 

\item Theorem \ref{well_posedness_flat:intro} asserts the uniqueness of traveling wave solutions in the small but does not exclude the possibility of nonuniqueness in the large. 
\end{enumerate}

Our proof of Theorem \ref{well_posedness_flat:intro} is based on the implicit function theorem, applied in a neighborhood of the trivial solutions obtained with $\gamma \in \mathfrak{C}$, $\mathfrak{f}_0 = \mathfrak{f}_1 = 0$, $\varphi_0 = \varphi_1 =0$, $u=0$, $p=0$, and $\eta =0$. In order for this strategy to work, we need a good understanding of the solvability of the linear problem 
\begin{equation}\label{linear_full}
\begin{cases}
 u + \nab p + \nab \pf \eta = F & \text{in }\Omega \\
 \diverge{u} = G &\text{in } \Omega \\
 u_n +\gamma \p_1 \eta = H &\text{on } \Sigma \\
 p  =   K &\text{on } \Sigma \\
 u_n = 0 & \text{on } \Sigma_{-b},
\end{cases}
\end{equation}
which is obtained by linearizing the flattened reformulation of \eqref{muskat_euler:intro} given in \eqref{well_posedness_flat_02:intro} around the trivial solutions.  More precisely, we need to identify appropriate function spaces $\mathbb{E}$ and $\mathbb{F}$ for which the linear map $\mathbb{E} \ni (u,p,\eta) \mapsto (F,G,H,K) \in \mathbb{F}$ induced by \eqref{linear_full} is an isomorphism.  When $\Gamma = \T^{n-1}$ the function spaces we employ are standard $L^2-$Sobolev spaces, but when $\Gamma = \R^{n-1}$ even identifying appropriate spaces turns out to be quite delicate for a couple reasons.  First, in $L^2-$Sobolev spaces on the infinite domain $\Omega = \R^{n-1} \times (-b,0)$ there are some subtle compatibility conditions that the data tuple $(F,G,H,K)$ need to satisfy, and these need to be encoded in $\mathbb{F}$.  Second, as mentioned in the above remarks, even when the data satisfy the appropriate compatibility conditions, the free surface function $\eta$ necessarily lives in the strange anisotropic Sobolev spaces given in Definition \ref{specialized_dfn}, which behave like standard $L^2-$based Sobolev spaces at large frequencies but have unusual anisotropic behavior at low frequencies (for instance, these spaces are not closed under composition with rotations).  Similar issues arose in the second author's recent work on the construction of traveling wave solutions to the incompressible Navier-Stokes system, \cite{koganemaru_tice,LeoniTice,stevenson_tice}, and fortunately, we were able to adapt some of the techniques used in those works to handle the Muskat construction of Theorem \ref{well_posedness_flat:intro}.  

In identifying the appropriate function spaces, we also uncover the method for showing that \eqref{linear_full} induces an isomorphism.  We first take the divergence of the first equation and eliminate $u$ to arrive at a problem for $p$ and $\eta$ only, \eqref{linear_pressure}.  To solve this problem we initially ignore the $\eta$ terms and view the resulting problem as an overdetermined problem for $p$, \eqref{ODP}.  This overdetermined problem is only solvable for data satisfying certain compatibility conditions, reminiscent of those from the closed range theorem, which we identify in Section \ref{sec_odp_cc}.  These turn out to be the key to solving \eqref{linear_pressure}, as they lead us to a pseudodifferential equation for $\eta$ that can be solved independently of $p$:
\begin{equation} 
\left[- i \gamma \xi_1 +  \abs{\xi} \tanh( \abs{\xi} b)\right] \hat{\eta}(\xi) = \psi(\xi),
\end{equation}
where $\psi$ is a specific function determined linearly by the data in \eqref{linear_pressure} (see \eqref{psi_def} for the precise definition).  It is this equation that forces $\eta$ into the anisotropic Sobolev spaces, but in turn the spaces allow us to construct $\eta$ and verify that it is a reasonably nice function.  Note that when $\Gamma = \R^{n-1}$ we require $\gamma \neq 0$ precisely because this term is responsible for ensuring that $\eta$ is a nice function; in the case $\gamma =0$ we lose the ability to verify this.  With $\eta$ in hand, we can then solve for $p$ and show that \eqref{linear_pressure} induces an isomorphism (see Theorem \ref{T2_iso}).  Then in Theorem \ref{T3_iso} we show that we can return to \eqref{linear_full} and uncover an isomorphism.  Finally, in Section \ref{sec_nonlinear_tw} we verify that our function spaces are nice enough to be used in an implicit function theorem argument and then employ the IFT to prove Theorem \ref{well_posedness_flat:intro}.

It is natural to investigate  the stability of the traveling wave solutions constructed in Theorem \ref{well_posedness_flat:intro}, and we next turn to this topic.  We expect that the stability analysis depends on the type and form of external forces. In our second main result, we prove that under the sole effect of the external pressure (i.e. $\mathfrak{f}_0=\mathfrak{f}_1=0$),  the small {\it periodic} traveling wave solutions constructed in Theorem \ref{well_posedness_flat:intro} are {\it asymptotically stable}. For simplicity we state the result for $\varphi(x)=\varphi_0(x')$. 

\begin{thm}(Proved in Section \ref{sec_stability} )\label{theo:stability:intro}
Let $\gamma \in \R$ and $1+\frac{n-1}{2}<s\in \R$. There exists a small positive constant $\ep_*=\ep_*(s, b, n)$ such that if $\| \nab \varphi_0\|_{H^{s-\frac12}(\T^{n-1})}<\ep_*$, then  the unique steady solution $\eta_*\in \mathring{H}^s(\T^{n-1})$ of \eqref{Muskat} with $\varphi(x)=\varphi_0(x')$ is asymptotically stable in $\mathring{H}^s(\T^{n-1})$. More precisely, there exist  positive constants  $\nu$ and $\delta$, both depending only on $(s, b, n)$, such that if $\eta_0 \in \mathring{H}^s(\T^{n-1})$ satisfies $\| \eta_0-\eta_*\|_{\mathring{H}^s}<\delta$, then the dynamic problem \eqref{Muskat} with initial data $\eta_0$ has a unique solution $\eta \in \eta_\ast +  B_{Y^s([0, \infty))}(0, \nu)$, where
\begin{equation}
 Y^s([0, \infty))=\wL^\infty([0, \infty); \mathring{H}^s(\T^{n-1}) )\cap L^2([0,\infty); \mathring{H}^{s+\hal}(\T^{n-1}));
\end{equation}
moreover, we have the estimates
\begin{equation}
\| \eta(t)-\eta_*\|_{H^s}\le \| \eta_0-\eta_*\|_{H^s}e^{-c_0t}\quad\forall t>0
\end{equation}
and
\begin{equation}
\int_0^\infty \| \eta(t)-\eta_*\|_{H^{s+\hal}}^2dt\le \frac{1}{2c_0}\| \eta_0-\eta_*\|^2_{H^s},
\end{equation}
where $c_0=c_0(s, b, d)$.
\end{thm}
To be best of our knowledge, Theorem \ref{theo:stability:intro} provides the first class of nontrivial stable solutions to the one-phase Muskat problem with graph free boundary. 

Inspired by the proof of stability of the trivial solution for the Muskat problem in \cite{Nguyen2022}, we obtain the stability of small periodic traveling wave solutions by linearizing the Dirichlet-Neumann operator about the flat surface, 
\begin{equation}
G(\eta)h=m(D)h+R(\eta)h,\quad m(D)=|D|\tanh(b|D|),
\end{equation}
 and establish good boundedness and contraction estimates for the remainder $R(\eta)$. More precisely, the results obtained in Section \ref{Section:DN} imply the estimates 
\begin{equation}\label{boundR:intro}
\| |D|^{-\hal}R(\eta)h\|_{H^{s}(\T^d)}\lesssim\| \eta\|_{H^{s}(\T^d)}\| \nab h\|_{H^{s-\frac12}(\T^d)}+ \| \eta\|_{H^{s+\frac12}(\T^d)}\| \nab h\|_{H^{s-1}(\T^d)}
\end{equation}
and
\begin{equation}\label{contractionR:intro}
\begin{aligned}
&\| |D|^{-\hal}\left\{R(\eta_1)h-R(\eta_2)h\right\}\|_{H^{s}(\T^d)}\\
&\quad\lesssim\| \eta_\delta\|_{H^{s}(\T^d)}\Big(\| \nab h\|_{H^{s-\frac12}(\T^d)}+\| \eta_1\|_{H^{s+\frac12}(\T^d)}\| \nab h\|_{H^{s-1}(\T^d)}\Big)+\|\eta_\delta\|_{H^{s+\frac12}(\T^d)} \| \nab  h\|_{H^{s-1}(\T^d)},
\end{aligned}
\end{equation}
where $\eta_\delta=\eta_1-\eta_2$ and $d=n-1$. The estimates in Section \ref{Section:DN} for the Dirichlet-Neumann operator are obtained for the free boundary belonging to the anisotropic Sobolev spaces $\cH^s(\Gamma)$, $\Gamma\in\{\R^d, \T^d\}$, and  are of independent interest. 

Now,  fix a traveling wave solution $\eta_*$ with data $\varphi_0$. The perturbation $f=\eta-\eta_*$ then satisfies 
\begin{equation}\label{eqf:intro}
\p_tf=\gamma\p_1 f-m(D)f+\left[R(\eta_*)(\eta_*+{\varphi_0})-R(\eta_*+f)(\eta_*+{\varphi_0})\right]-R(\eta_*+f)f.
\end{equation}
Assuming that $f_0$ has zero mean in $H^s(\T^d)$, then $f(t)$ has zero mean for $t>0$. When performing the $H^s$ energy estimate, the dissipation term $m(D)$ yields a gain of $\frac12$ derivative:
\begin{equation}
(m(D)f, f)_{H^s(\T^d)}\ge c_0(s, b, d)\| f\|_{H^{s+\frac12}(\T^d)}^2. 
\end{equation}
On the other hand, by virtue of \eqref{boundR:intro} and \eqref{contractionR:intro}, we can control the nonlinear terms in \eqref{eqf:intro}  in $H^{s-\frac12}(\T^d)$ by 
\begin{equation}
C\big(\alpha +\| f\|_{H^s(\T^d)}\big)\| f\|_{H^{s+\frac12}(\T^d)}^2,
\end{equation}
where the coefficient $\alpha$ is small when $\varphi_0$ and $\eta_*$ are small. Therefore, if $\|f(t)\|_{H^s}$ is small globally, then it decays exponentially. On the other hand, the global existence and smallness of $\|f(t)\|_{H^s}$ are proved by appealing to the estimates \eqref{boundR:intro} and \eqref{contractionR:intro} again for the mild-solution formulation of \eqref{eqf:intro}.
\section{Problem reformulations} 

In this section we present two reformulations we will use in proving our two main theorems. The first one is a reformulation for  the general traveling wave system \eqref{muskat_euler:intro} in a flattened domain.  When the generic body force $\mathfrak{f}$ is absent, we present a reformulation for the dynamic problem  \eqref{Muskat:wpzeta} using the Dirichlet-Neumann operator. 

\subsection{Flattening the traveling wave system}\label{subsection:flattening}

Consider the flat domain $\Omega := \Omega_0 = \Gamma \times (-b,0)$ and write $\Sigma = \Sigma_0 = \Gamma \times \{0\}$.  We define the Poisson extension operator $\pf$ as in Appendix \ref{appendix:poisson_ext}.  Assuming that $\eta \in \H^{s+3/2}(\Sigma)$ (see Definition \ref{specialized_dfn} for the precise definition of this anisotropic Sobolev space), we define the flattening map $\ff_\eta : \bar{\Omega} \to \bar{\Omega}_\eta$ via 
\begin{equation}\label{flattening_def}
 \ff_\eta(x) = (x', x_n + \pf \eta(x)(1+x_n/b)) = x + \pf\eta(x)\left(1+x_n/b\right) e_n.
\end{equation}
Note that $\ff_\eta\vert_{\Sigma_{-b}} = I$ and $\ff_\eta(\Sigma) = \Sigma_{\eta}$.  We compute 
\begin{equation}
 \nab \ff_\eta(x) = 
\begin{pmatrix}
I_{n-1 } & 0_{(n-1) \times 1} \\
(1+x_n/b) \nab'\pf \eta(x) & 1 + \pf\eta(x) /b  + \p_n \pf\eta(x) (1+x_n/b)
\end{pmatrix}.
\end{equation}
We define the functions $\jf,\kf : \Omega \to (0,\infty)$ via  
\begin{equation}\label{JK_def}
  \jf(x) = \det \nab \ff_\eta(x) = 1 + \pf\eta(x) /b  + \p_n \pf\eta(x) (1+x_n/b) \text{ and } \kf(x) = 1/\jf(x). 
\end{equation}
It will be useful to introduce the matrix field $\M : \Omega \to \R^{n \times n}$ via
\begin{equation}\label{M_def}
 \M(x) = (\nab \ff_\eta(x))^{-T} = 
\begin{pmatrix}
I_{n-1} & - \kf(x)(1+x_n/b) \nab' \pf \eta(x) \\
0_{1 \times (n-1)} & \kf(x)
\end{pmatrix}.
\end{equation}

Our interest in the field $\M$ comes from a trio of useful identities it satisfies.  The first is Piola identity, 
\begin{equation}
 \p_j[\jf \M_{ij} ] =0 \text{ for } 1\le i \le n.
\end{equation}
The second and third are a pair of identities on $\Sigma$ and $\Sigma_{-b}$: 
\begin{equation}
 \jf \M e_n \vert_{\Sigma} = (-\nab' \eta,1) \text{ and } \jf \M e_n \vert_{\Sigma_{-b}} = e_n.
\end{equation}
To see the utility of the Piola identity note that $v : \Omega_\eta \to \R^n$ satisfies $\diverge{v}=0$ if and only if $\hat{v} = v \circ \ff_\eta : \Omega \to \R^n$ satisfies $\jf \M_{ij} \p_j \hat{v}_i =0$ (the summation convention is used here), but 
\begin{equation}
 \jf \M_{ij} \p_j \hat{v}_i =  \p_j[\jf \M_{ij} \hat{v}_i  ] = \p_j[\jf \M^T_{ji} \hat{v}_i  ],
\end{equation}
so a further equivalent condition is that $u = \jf \M^T \hat{v}: \Omega \to \R^n$ satisfies $\diverge{u}=0$.

In light of the previous calculations, we use $\ff_\eta$ and $\M$ to rephrase the traveling wave Muskat system \eqref{muskat_euler:intro} in the fixed domain $\Omega$ by defining $u = \jf \M^T v \circ \ff_\eta$ and $p = -\pf \eta +  q \circ \ff_\eta$.  The new system reads
\begin{equation}\label{muskat_flattened}
\begin{cases}
 u + \nab_{\A} p + \nab_{\A} \pf \eta = \jf \M^T \mathfrak{f}\circ \ff_\eta & \text{in }\Omega \\
 \diverge{u} = 0 &\text{in } \Omega \\
 -\gamma \p_1 \eta = u_n &\text{on } \Sigma \\
 p  =   \varphi \circ \ff_\eta &\text{on } \Sigma \\
 u_n =0 & \text{on } \Sigma_{-b},
\end{cases}
\end{equation}
where $\A : \Omega \to \R^{n \times n}_{sym}$ is defined by  
\begin{equation}\label{A_def}
 \A(x) = \jf(x) \M^T(x) \M(x) = 
\begin{pmatrix}
\jf(x) I_{n-1} & -  (1+x_n/b) \nab' \pf \eta(x) \\
- (1+x_n/b) \nab' \pf \eta(x) & \kf(x) + \kf(x) (1+x_n/b)^2 \abs{ \nab' \pf \eta(x)}^2
\end{pmatrix},
\end{equation}
and we write
\begin{equation}
 \naba \psi = \A \nabla \psi.
\end{equation}

\subsection{Dirichlet-Neumann reformulation}\label{subsection:DN}

We consider the dynamic  Muskat problem \eqref{Muskat:wpzeta} with $\widetilde{\mathfrak{f}}=0$ and  $\phi(x,t) = \varphi(x-\gamma t e_1)$. In the moving frame $x\mapsto x-\gamma t e_1$, we make the change of variables
\begin{equation}
\zeta(x', t)=\eta(x'-\gamma t e_1, t),\quad w(x,t) = v(x-\gamma t e_1, t),\quad P(x,t) = P_0 - x_n - q(x - \gamma t e_1, t)
\end{equation}
to obtain the system 
\begin{equation}\label{Muskat:0}
\begin{cases}
v+\nab q=0\quad\text{in } \Omega_\eta ,\\
\di v=0\quad\text{in } \Omega_\eta,\\
\p_t \eta-\gamma\p_1 \eta=v\cdot (\nab' \eta, 1)\quad\text{on } \Sigma_\eta,\\
q=\eta+\varphi\quad\text{on } \Sigma_\eta,\\
v_n=0\quad\text{on } \Sigma_{-b},
\end{cases}
\end{equation}
This problem can be recast on the free boundary by means of the Dirichlet-Neumann operator \eqref{def:DN} defined as follows.  Let $\psi$ be the solution of
 \begin{equation}
\begin{cases}
\Delta \psi=0\quad\text{in } \Omega_\eta,\\
\psi=f\quad\text{on } \Sigma_\eta,\\
\p_n\psi=0\quad\text{on }  \Sigma_{-b}.
\end{cases}
\end{equation}
The Dirichlet-Neumann operator associated to $\Omega$ is denoted by $G(\eta)$ and 
\begin{equation}\label{def:DN}
[G(\eta)f](x'):=N(x')\cdot(\nab\psi)(x', \eta(x')),\quad N(x')=(-\nab' \eta(x'), 1).
\end{equation}
By  taking the divergence of the first equation in \eqref{Muskat:0}, we deduce that $q$ satisfies 
\begin{equation}
\begin{cases}
\Delta q=0\quad\text{in } \Omega_\eta,\\
q=\eta+\varphi(\cdot, \eta(\cdot))\quad\text{on } \Sigma_\eta,\\
\p_nq=0\quad\text{on } \Sigma_{-b},
\end{cases}
\end{equation}
It follows from the third equation in \eqref{Muskat:0} that
\begin{equation}
v\cdot (-\nab'\eta, 1)=-\nab q\cdot (-\nab'\eta, 1)=-G(\eta)\big(\eta+\varphi(\cdot, \eta(\cdot))\big),
\end{equation}
where $G(\eta)$ denotes the Dirichlet-Neumann operator for $\Omega_\eta$.  Therefore,  $\eta$ obeys the equation 
\begin{equation}\label{Muskat}
\p_t\eta=\gamma \p_1\eta-G(\eta)\big(\eta+\varphi(\cdot, \eta(\cdot))\big)\quad \text{on } \R^{n-1}\times \R_+.
\end{equation}

\section{Linear analysis for the traveling wave system} 

In this section we study the linearization of \eqref{muskat_flattened} around the trivial solution, which is the system \eqref{linear_full}, where $(F,G,H,K)$ are given data.  Note that for the purposes of studying the linearization of \eqref{muskat_flattened} we could reduce to the case $G=0$ and $H=0$; we have retained these terms here for the sake of generality.

We can eliminate $u$ to get an equivalent formulation of the problem.  Indeed, we take the divergence of the first equation and then use the first equation to remove $u$ from the boundary conditions.  This results in the  problem 
\begin{equation}\label{linear_pressure}
\begin{cases}
 -\Delta p = G-\diverge{F}   & \text{in }\Omega \\
 -\p_n p -\p_n \pf \eta  +\gamma \p_1 \eta = H - F_{n}(\cdot,0) &\text{on } \Sigma \\
 p  =   K &\text{on } \Sigma \\
 -\p_n p - \p_n \pf \eta  =  -F_{n}(\cdot,-b) & \text{on } \Sigma_{-b}.
\end{cases}
\end{equation}
We will study this form of the problem and eventually show that it is equivalent to \eqref{linear_full}.

\begin{remark}
Throughout what follows we will often abuse notation by identifying 
\begin{equation}
\Sigma \simeq \Sigma_{-b} \simeq \Gamma \in \{\R^{n-1}, \T^{n-1}\} 
\end{equation}
in order to allow us to handle linear combinations of functions defined on $\Sigma$, $\Sigma_{-b}$, and $\Gamma$ in a simple way.  In reality we actually identify these through the natural isometric isomorphism, but this is obvious and the corresponding notation is too cumbersome to introduce.
\end{remark}

\subsection{The upper-Dirichlet-lower-Neumann isomorphism}

Consider the problem 
\begin{equation}\label{udln}
\begin{cases}
-\Delta p = f & \text{in } \Omega \\
p = k & \text{on } \Sigma \\
-\p_n p = l &\text{on } \Sigma_{-b}
\end{cases}
\end{equation}
for given $(f,k,l) \in H^s(\Omega) \times H^{s+3/2}(\Sigma) \times H^{s+1/2}(\Sigma_{-b})$.  Associated to this PDE is the bounded linear map 
\begin{equation}
 T_0 : H^{s+2}(\Omega) \to H^s(\Omega) \times H^{s+3/2}(\Sigma) \times H^{s+1/2}(\Sigma_{-b})
\end{equation}
given by 
\begin{equation}
 T_0 p = (-\Delta p, p \vert_{\Sigma}, -\p_n p \vert_{\Sigma_{-b}}).
\end{equation}

\begin{thm}\label{udln_iso}
 The map $T_0$ is an isomorphism for every $0 \le s \in \R$.
\end{thm}
\begin{proof}
To see that $T_0$ is injective we suppose that $T_0 p =0$, multiply the resulting equation $-\Delta p =0$ by $p$ and integrate by parts.  Using the boundary conditions contained in the identity $T_0 p =0$, we deduce that $\int_{\Omega} \abs{\nab p}^2 = 0,$ and so $p=0$ in $\Omega$ since $p=0$ on $\Sigma$. Thus $p=0$, and injectivity is proved.

It remains to prove that $T_0$ is surjective, and this ultimately boils down to the weak solvability and elliptic regularity associated to the problem \eqref{udln}, which we will briefly sketch.  We initially define the space ${^0}H^1(\Omega) = \{f \in H^1(\Omega) \st f=0 \text{ on } \Sigma\},$ which we can equip with the inner-product $\ip{f,g}_{{^0}H^1} = \int_\Omega \nab f \cdot \nab g$. This in indeed an inner-product and generates the usual $H^1$ topology thanks to a Poincar\'{e}-type inequality provided by the vanishing on $\Sigma$.  Then by Riesz representation, for any $\mathcal{F} \in ( {^0}H^1(\Omega))^\ast$, there exists a unique $p \in  {^0}H^1(\Omega)$ such that 
\begin{equation}
 \int_{\Omega} \nab p \cdot \nab q = \br{\mathcal{F},q} \text{ for all } q \in  {^0}H^1(\Omega),
\text{ and } \norm{p}_{ {^0}H^1} = \norm{\mathcal{F}}_{( {^0}H^1(\Omega))^\ast}.
\end{equation}

Next we consider $s \in \N$ and data $f \in H^{s}(\Omega)$ and $l \in H^{s+1/2}(\Sigma_{-b})$.  According to standard trace theory and the above, we can then find a unique $p \in  {^0}H^1(\Omega)$ such that 
\begin{equation}
 \int_{\Omega} \nab p \cdot \nab q = \int_{\Omega} fq +  \int_{\Sigma_{-b}} l q \text{ for all } q \in  {^0}H^1(\Omega),
\text{ and }
\norm{p}_{ {^0}H^1} \ls \norm{f}_{H^s} +  \norm{l}_{H^{s+1/2}}.
\end{equation}
Standard interior elliptic regularity shows that $p \in H^{s+2}_{\text{loc}}(\Omega)$ and $-\Delta p =f$ in $\Omega$.  Using horizontal difference quotients, we may deduce in turn that 
\begin{equation}
 \sum_{\abs{\alpha} \le s+1, \alpha_n=0} \norm{\p^\alpha p}_{ {^0}H^1} \ls  \norm{f}_{H^s} +  \norm{l}_{H^{s+1/2}}.
\end{equation}
We then recover control of the vertical derivatives by using the identity $-\p_n^2 p = \Delta' p + f$ together with a simple iteration argument; this yields the inclusion $p \in H^{s+2}(\Omega)$ with the estimate 
\begin{equation}
 \norm{p}_{H^{s+2}} \ls \norm{f}_{H^s} +  \norm{l}_{H^{s+1/2}}.
\end{equation}
Returning to the weak formulation and integrating by parts, we find that 
\begin{equation}
 \int_{\Omega} (-\Delta p - f) q = \int_{\Sigma_{-b}} (l+\partial_n p) q \text{ for all } q \in {^0}H^1(\Omega),
\end{equation}
and hence that $-\p_n p = l$ on $\Sigma_{-b}$.  Thus, $p \in H^{s+2}(\Omega) \cap {^0}H^1(\Omega)$ satisfies $T_0 p = (f,0,l)$.

For each $s \in \N$ this analysis defines a bounded linear map $S_0 : H^{s}(\Omega) \times H^{s+1/2}(\Sigma_{-b})  \to H^{s+2}(\Omega) \cap {^0}H^1(\Omega)$ via $S_0(f,l) = p$.   Employing the usual Sobolev interpolation theory (see, for instance, \cite{bergh_lof,triebel}), we deduce that $S_0$ extends to a map between the same spaces but for all $0 \le s \in \R$.

Now suppose that  $f \in H^s(\Omega)$, $k \in H^{s+3/2}(\Sigma)$, and $l \in H^{s+1/2}(\Sigma_{-b})$ for some $0\le s \in \R$.  By trace theory, we can pick $K \in H^{s+2}(\Omega)$ such that $P = k$ on $\Sigma_b$.  Using the above, we then find $P = S_0(f+\Delta K,l+\p_n K)  \in H^{s+2}(\Omega)\cap {^0}H^1(\Omega)$, which satisfies $T_0 P = (f + \Delta K, 0, l + \p_n K)$.  Then $p:= P+K \in H^{s+2}(\Omega)$ satisfies $T_0 p = (f,k,l)$, and we conclude that $T_0$ is surjective.
\end{proof}

Later in our analysis we will need to consider the following bounded linear operator.

\begin{dfn}\label{Xi_def}
We define the bounded linear map $\Xi : H^{s+3/2}(\Sigma) \to H^{s+2}(\Omega)$ via $\Xi k = T_0^{-1}(0,k,0)$.  
\end{dfn}

The next result records a crucial property of $\Xi$.

\begin{thm}\label{xi_symbol}
The map $\Xi$ from Definition \ref{Xi_def} satisfies  $\widehat{\Xi k} (\xi,x_n) = \hat{k}(\xi) Q(\xi,x_n)$ for  $\xi \in \hat{\Gamma}$, where $Q : \R^{n-1} \times (-b,0) \to \R$ is defined by 
\begin{equation}
 Q(\xi,x_n) = \frac{\cosh( \abs{\xi} (x_n+b))}{\cosh( \abs{\xi} b)}.
\end{equation}
Note that in the case $\Gamma = \T^{n-1}$ we have that the dual group is $\hat{\Gamma} = \Z^{n-1} \subset \R^{n-1}$ and $Q$ is given by restriction to $\hat{\Gamma}$.
\end{thm}
\begin{proof}
Write $p = T_0^{-1}(0,k,0) \in H^{s+2}(\Omega)$, and let $\hat{p}$ denote its horizontal Fourier transform.  Then $\hat{p}$ satisfies the ordinary differential boundary value problem
\begin{equation}
\begin{cases}
 (- \abs{\xi}^2 + \p_n^2) \hat{p}(\xi,x_n) =0 &\text{for }x_n \in (-b,0) \\
 \hat{p}(\xi,0) = \hat{k}(\xi) \\
 \p_n\hat{p}(\xi,-b) = 0.
\end{cases}
\end{equation}
From this it's an elementary exercise to verify that $\hat{p}(\xi,x_n) = \hat{k}(\xi) Q(\xi,x_n)$, and the result follows.
\end{proof}

\subsection{The over-determined problem: compatibility conditions}\label{sec_odp_cc}

Consider the over-determined problem 
\begin{equation}\label{ODP}
\begin{cases}
-\Delta p = f & \text{in } \Omega \\
p = k & \text{on } \Sigma \\
-\p_n p = h_+ &\text{on } \Sigma \\
-\p_n p = h_- &\text{on } \Sigma_{-b}
\end{cases}
\end{equation}
for given $(f,h_+,h_-,k) \in H^s(\Omega) \times H^{s+1/2}(\Sigma) \times H^{s+1/2}(\Sigma_{-b}) \times H^{s+3/2}(\Sigma)$.

Associated to \eqref{ODP} are a pair of compatibility conditions.  The first actually is associated to a sub-system of \eqref{ODP}.

\begin{prop}\label{odp_CC1}
Suppose that $(f,h_+,h_-) \in H^s(\Omega) \times H^{s+1/2}(\Sigma) \times H^{s+1/2}(\Sigma_{-b})$ and  $p \in H^{s+2}(\Omega)$ satisfy
\begin{equation}
\begin{cases}
-\Delta p = f & \text{in } \Omega \\
-\p_n p = h_+ &\text{on } \Sigma \\
-\p_n p = h_- &\text{on } \Sigma_{-b}.
\end{cases}
\end{equation}
Then the following hold.
\begin{enumerate}
 \item If $\Gamma = \R^{n-1}$, then 
\begin{equation}
\int_{-b}^0 f(\cdot,x_n) dx_n - (h_+ - h_-) \in \dot{H}^{-2}(\Sigma) \cap \dot{H}^{-1}(\Sigma)
\end{equation}
and we have the bounds
\begin{equation}
 \snorm{\int_{-b}^0 f(\cdot,x_n) dx_n - (h_+ - h_-) }_{\dot{H}^{-2}} \ls \norm{p}_{L^2}
\text{ and }
 \snorm{\int_{-b}^0 f(\cdot,x_n) dx_n - (h_+ - h_-) }_{\dot{H}^{-1}} \ls \norm{\nab' p}_{L^2}. 
\end{equation}

 \item If $\Gamma = \T^{n-1}$, then 
\begin{equation}
\int_{-b}^0 \hat{f}(0,x_n) dx_n - (\hat{h}_+(0) - \hat{h}_-(0)) =0.
\end{equation}
 
\end{enumerate}
\end{prop}
\begin{proof}
We will only record the proof for $\Gamma = \R^{n-1}$, as the other case follows from similar but simpler analysis.  We have that 
\begin{equation}
 \int_{-b}^0 f(\cdot,x_n) dx_n =  \int_{-b}^0 -\Delta' p(\cdot,x_n) dx_n -  \int_{-b}^0 \p_n^2 p(\cdot,x_n) dx_n.
\end{equation}
We then compute 
\begin{equation}
\int_{-b}^0 -\Delta' p(\cdot,x_n) dx_n = -\Delta' \int_{-b}^0  p(\cdot,x_n) dx_n 
\end{equation}
and 
\begin{equation}
\int_{-b}^0 \p_n^2 p(\cdot,x_n) dx_n = \p_n p(\cdot,0) - \p_n p(\cdot,-b) =  -(h_+ - h_-).
\end{equation}
Combining these, we see that 
\begin{equation}
 \int_{-b}^0 f(\cdot,x_n) dx_n - (h_+ - h_-) = -\Delta' \int_{-b}^0 p(\cdot,x_n) dx_n,
\end{equation}
and the $\dot{H}^{-2}$ inclusion and estimate then follow from an application of Cauchy-Schwarz, Fubini-Tonelli, and Parseval: 
\begin{multline}
\snorm{-\Delta' \int_{-b}^0 p(\cdot,x_n) dx_n}_{\dot{H}^{-2}}^2 = \int_{\R^{n-1}} \frac{(  \abs{\xi}^2)^2}{\abs{\xi}^4} \abs{\int_{-b}^0 \hat{p}(\xi,x_n) dx_n}^2 d\xi \\
\le  b^2 \int_{\R^{n-1}} \int_{-b}^0 \abs{\hat{p}(\xi,x_n)}^2 dx_n d\xi 
\ls \int_{\Omega} \abs{p(x)}^2 dx.
\end{multline}
The $\dot{H}^{-1}$ inclusion and estimate follow similarly.
\end{proof}

Next we identify the formal adjoint of the over-determined problem as an under-determined problem, given here in homogeneous form:
\begin{equation}\label{UDP}
\begin{cases}
-\Delta q = 0 & \text{in } \Omega \\
-\p_n q = 0 &\text{on } \Sigma_{-b}.
\end{cases}
\end{equation}
We can augment this problem with an extra Dirichlet condition at the upper boundary in order to introduce the upper-Dirichlet-lower-Neumann problem \eqref{udln}.  Indeed, we can parameterize solutions to \eqref{UDP} by letting $q = \Xi \psi$ for some $k \in H^{s+3/2}(\Sigma)$, where $\Xi$ is as in Definition \ref{Xi_def}.  With this in mind we borrow an idea from the closed range theorem to deduce a second compatibility condition.

\begin{prop}\label{odp_CC2}
Suppose that  $(f,h_+,h_-,k) \in H^s(\Omega) \times H^{s+1/2}(\Sigma) \times H^{s+1/2}(\Sigma_{-b}) \times H^{s+3/2}(\Sigma)$ and $p \in H^{s+2}(\Omega)$ satisfy \eqref{ODP}.  Then the data $(f,h_+,h_-,k)$ satisfy both of the following equivalent conditions:
\begin{enumerate}
 \item For every $\psi \in H^{s+3/2}(\Sigma)$, if we let $q = \Xi \psi \in H^{s+2}(\Omega)$ for $\Xi$ as in  Definition \ref{Xi_def}, then 
\begin{equation}\label{odp_CC2_00}
 \int_{\Omega} f q - \int_{\Sigma} k \p_n q + h_+ \psi + \int_{\Sigma_{-b}} h_- q = 0.
\end{equation} 
 
 \item For a.e. $\xi \in \hat{\Gamma}$ we have that
\begin{equation}\label{odp_CC2_01}
0= \int_{-b}^0 \hat{f}(\xi, x_n) \frac{\cosh( \abs{\xi} (x_n+b))}{\cosh( \abs{\xi} b)} dx_n  - \hat{k}(\xi)  \abs{\xi} \tanh( \abs{\xi} b) - \hat{h}_+(\xi)  + \hat{h}_-(\xi) \sech( \abs{\xi} b). 
\end{equation}
\end{enumerate}
\end{prop}
\begin{proof}
Let $\psi \in H^{s+3/2}(\Sigma)$ and write $q = \Xi \psi \in H^{s+2}(\Omega)$.  Multiplying the first equation in \eqref{ODP} by $q$ and integrating by parts, we find that 
\begin{multline}\label{odp_CC2_1}
 \int_{\Omega} f q = \int_{\Omega} -\Delta p q = \int_{\Omega} -\Delta q p + \int_{\p \Omega} p \p_\nu q - \p_\nu p q 
= \int_{\Sigma} p \p_n q - \p_n p q - \int_{\Sigma_{-b}} p \p_n q - \p_n p q  \\
= \int_{\Sigma} k \p_n q  + h_+  \psi - \int_{\Sigma_{-b}} h_-  q.
\end{multline}
Rearranging yields \eqref{odp_CC2_00}.  It remains to prove that \eqref{odp_CC2_01} is equivalent to this.

Viewing $k$, $h_+$, and $h_-$ as functions on $\hat{\Gamma}$ in the natural way, we may rearrange \eqref{odp_CC2_1} and apply Fubini-Tonelli to see that
\begin{equation}
 \int_{\hat{\Gamma}} \left[\int_{-b}^0 f(\cdot, x_n) q(\cdot,x_n) dx_n  - k \p_n q(\cdot,0) - h_+ \psi + h_- q(\cdot,-b)  \right] =0.
\end{equation}
From this, Parseval's theorem, and Theorem \ref{xi_symbol}, we then find that 
\begin{equation}
 \int_{\hat{\Gamma}} \left[\int_{-b}^0 \hat{f}(\xi, x_n) \overline{Q(\xi,x_n)} dx_n  - \hat{k}(\xi) \overline{\p_n Q(\xi,0)} - \hat{h}_+(\xi)  + \hat{h}_-(\xi) \overline{Q(\xi,-b)}  \right] \overline{\hat{\psi}}(\xi) d\xi =0
\end{equation}
for all $\psi \in H^{s+3/2}(\Sigma)$.  This implies the  identity 
\begin{equation}
0= \int_{-b}^0 \hat{f}(\xi, x_n) \overline{Q(\xi,x_n)} dx_n  - \hat{k}(\xi) \overline{\p_n Q(\xi,0)} - \hat{h}_+(\xi)  + \hat{h}_-(\xi) \overline{Q(\xi,-b)} 
\end{equation}
for a.e. $\xi \in \hat{\Gamma}$, and \eqref{odp_CC2_01} then follows by employing the formula for $Q(\xi,x_n)$ from Theorem \ref{xi_symbol}.  The fact that \eqref{odp_CC2_01} implies \eqref{odp_CC2_00} is readily seen by multiplying \eqref{odp_CC2_01} by $\overline{\hat{\psi}}$ and then working backward through the above argument with Parseval and Fubini-Tonelli.
\end{proof}

Next we show that data obeying the conditions identified in this result must also obey an estimate in $\dot{H}^{-2}$ as in Proposition \ref{odp_CC1}.

\begin{prop}\label{odp_CC_implication}
If $(f,h_+,h_-,k) \in H^s(\Omega) \times H^{s+1/2}(\Sigma) \times H^{s+1/2}(\Sigma_{-b}) \times H^{s+3/2}(\Sigma)$ satisfy either (and thus both) of the conditions in Proposition \ref{odp_CC2}.  Then the following hold.
\begin{enumerate}
 \item If $\Gamma = \R^{n-1}$, then 
\begin{equation}\label{odp_CC_implication_00}
\int_{-b}^0 f(\cdot,x_n) dx_n - (h_+ - h_-) \in \dot{H}^{-2}(\R^{n-1})
\end{equation}
and 
\begin{equation}\label{odp_CC_implication_01}
 \snorm{\int_{-b}^0 f(\cdot,x_n) dx_n - (h_+ - h_-) }_{\dot{H}^{-2}} \ls \norm{f}_{L^2} + \norm{h_+}_{L^2} + \norm{h_-}_{L^2} + \norm{k}_{L^2}.
\end{equation}
 \item If $\Gamma = \T^{n-1}$, then 
\begin{equation}\label{odp_CC_implication_02}
\int_{-b}^0 \hat{f}(0,x_n) dx_n - (\hat{h}_+(0) - \hat{h}_-(0)) =0.
\end{equation}
\end{enumerate}
\end{prop}
\begin{proof}
We will only record the proof when $\Gamma = \R^{n-1}$ as the other case is simpler.  The condition \eqref{odp_CC2_01} implies that 
\begin{multline}
\int_{-b}^0 \hat{f}(\xi,x_n) dx_n - (\hat{h}_+(\xi) - \hat{h}_-(\xi)) = 
\int_{-b}^0 \hat{f}(\xi, x_n) \left[1-\frac{\cosh( \abs{\xi} (x_n+b))}{\cosh(\abs{\xi} b)} \right]dx_n \\
+ \hat{k}(\xi)  \abs{\xi} \tanh(\abs{\xi} b) +    \hat{h}_-(\xi) \left[1- \sech( \abs{\xi} b)\right].
\end{multline}
Upon making routine Taylor expansions and applying Cauchy-Schwarz and Parseval, we see that 
\begin{multline}\label{odp_CC_implication_1}
\int_{B(0,1)} \frac{1}{\abs{\xi}^4}  \abs{\int_{-b}^0 \hat{f}(\xi,x_n) dx_n - (\hat{h}_+(\xi) - \hat{h}_-(\xi)) }^2 d\xi \\
\ls \int_{B(0,1)} \left[ \abs{\hat{k}(\xi)}^2 + \abs{\hat{h}_-(\xi)}^2 + \int_{-b}^0 x_n^2 \abs{\hat{f}(\xi,x_n)}^2 dx_n \right] d\xi \ls \norm{k}_{L^2}^2 + \norm{h_-}_{L^2}^2 + \norm{f}_{L^2}^2.
\end{multline}
This yields the low frequency control of the $\dot{H}^{-2}$ seminorm, but the high frequency control comes directly from Cauchy-Schwarz and Parseval:
\begin{equation}\label{odp_CC_implication_2}
\int_{B(0,1)^c} \frac{1}{\abs{\xi}^4}  \abs{\int_{-b}^0 \hat{f}(\xi,x_n) dx_n - (\hat{h}_+(\xi) - \hat{h}_-(\xi)) }^2 d\xi \ls \norm{f}_{L^2}^2 + \norm{h_+}_{L^2}^2 + \norm{h_-}_{L^2}^2. 
\end{equation}
Thus, the inclusion \eqref{odp_CC_implication_00} holds, and upon summing \eqref{odp_CC_implication_1} and \eqref{odp_CC_implication_2} we deduce the estimate \eqref{odp_CC_implication_01}.

\end{proof}

\subsection{A pair of useful function spaces}

We now introduce a couple function spaces that will be useful in our study of the over-determined problem \eqref{ODP}.

\begin{dfn}\label{Ys_Zs_defs}
For $0 \le s \in \R$ we define the following spaces.
\begin{enumerate}
 \item For $\Gamma = \R^{n-1}$ and $0 < t \in \R$ we define the space
\begin{multline}
 Y^s_{t} = \{(f,h_+,h_-,k)  \in H^s(\Omega) \times H^{s+1/2}(\Sigma) \times H^{s+1/2}(\Sigma) \times H^{s+3/2}(\Sigma_{-b}) \st \\
 \int_{-b}^0 f(\cdot,x_n) dx_n - (h_+ - h_-) \in \dot{H}^{-t}(\Sigma)\}
\end{multline}
and endow this space with the square-norm
\begin{equation}
 \norm{(f,h_+,h_-,k)}_{Y^s_t}^2 = \norm{f}_{H^s}^2 + \norm{h_+}_{H^{s+1/2}}^2 + \norm{h_-}_{H^{s+1/2}}^2 + \norm{k}_{H^{s+3/2}}^2 + \snorm{\int_{-b}^0 f(\cdot,x_n) dx_n - (h_+ - h_-)}_{\dot{H}^{-t}}^2
\end{equation}
and its associated inner-product.

 \item For $\Gamma = \T^{n-1}$ and $0 < t \in \R$ we define the space
\begin{multline}
 Y^s_{t} = \{  (f,h_+,h_-,k) \in H^s(\Omega) \times H^{s+1/2}(\Sigma) \times H^{s+1/2}(\Sigma) \times H^{s+3/2}(\Sigma_{-b}) \st \\
 \int_{-b}^0 \hat{f}(0,x_n) dx_n - (\hat{h}_+(0) - \hat{h}_-(0)) =0\}
\end{multline}
and endow this space with the square-norm
\begin{equation}
 \norm{(f,h_+,h_-,k)}_{Y^s_t}^2 = \norm{f}_{H^s}^2 + \norm{h_+}_{H^{s+1/2}}^2 + \norm{h_-}_{H^{s+1/2}}^2 + \norm{k}_{H^{s+3/2}}^2 
\end{equation}
and its associated inner-product.

\item We define the space
\begin{multline}
 Z^s = \{(f,h_+,h_-,k) \in (f,h_+,h_-,k) \in H^s(\Omega) \times H^{s+1/2}(\Sigma) \times H^{s+1/2}(\Sigma_{-b}) \times H^{s+3/2}(\Sigma) \st \\
 \int_{\Omega} f q - \int_{\Sigma} k \p_n q + h_+ \psi + \int_{\Sigma_{-b}} h_- q = 0 \text{ for every } \psi \in H^{s+3/2}(\Sigma), \text{ where } q = \Xi \psi \}.
\end{multline}
Here we recall that $\Xi : H^{s+3/2}(\Sigma) \to H^{s+2}(\Omega)$ is defined in Definition \ref{Xi_def}.  We endow $Z^s$ with the square norm 
\begin{equation}
  \norm{(f,h_+,h_-,k)}_{Z^s}^2 = \norm{f}_{H^s}^2 + \norm{h_+}_{H^{s+1/2}}^2 + \norm{h_-}_{H^{s+1/2}}^2 + \norm{k}_{H^{s+3/2}}^2
\end{equation}
and its associated inner-product.
\end{enumerate}
\end{dfn}

The next result establishes some key properties of these spaces.

\begin{prop}\label{Ys_Zs_spaces}
Let $0 \le s \in \R$, $0 < t \in \R$, and let $Y^s_t$ and $Z^s$ be as in Definition \ref{Ys_Zs_defs}.  Then the following hold.
\begin{enumerate}
 \item $Y^s_t$ and $Z^s$ are Hilbert spaces.
 \item If $t < r \in \R$ then we have the continuous inclusion $Y^s_r \hookrightarrow Y^s_t$.
 \item We have the continuous inclusion $Z^s \hookrightarrow Y^s_2$.
\end{enumerate}
\end{prop}
\begin{proof}
The completeness of $Y^s_t$ is routine to verify, and since $\Xi$ is a bounded linear map it is easy to see that $Z^s$ is a closed subspace of $H^s(\Omega) \times H^{s+1/2}(\Sigma) \times H^{s+1/2}(\Sigma_{-b}) \times H^{s+3/2}(\Sigma)$ and thus complete.  This proves the first item.  The second item is trivial when $\Gamma = \T^{n-1}$, and when $\Gamma = \R^{n-1}$ it follows from the fact that 
\begin{equation}
 \int_{B(0,1)} \frac{1}{\abs{\xi}^{2t}} \abs{\psi(\xi)}^2 d\xi \le  \int_{B(0,1)} \frac{1}{\abs{\xi}^{2r}} \abs{\psi(\xi)}^2 d\xi
\end{equation}
when $t < r$ and $\psi$ is measurable.  The continuous inclusion $Z^s \hookrightarrow Y^s_2$ follows from Proposition \ref{odp_CC_implication}.
\end{proof}

\subsection{The over-determined problem: isomorphism}

We now aim to establish an isomorphism associated to the over-determined problem \eqref{ODP}.

\begin{thm}\label{ODP_iso}
The bounded linear map $T_1 : H^{s+2}(\Omega) \to Z^s$ associated to \eqref{ODP}, which is given by
\begin{equation}
T_1 p = (-\Delta p,  -\p_n p \vert_{\Sigma}, -\p_n p\vert_{\Sigma_{-b}}, p\vert_{\Sigma}),
\end{equation}
is well-defined and is an isomorphism for every $0 \le s \in \R$.
\end{thm}
\begin{proof}
The map $T_1$ is obviously a bounded linear map into  $H^s(\Omega) \times H^{s+1/2}(\Sigma) \times H^{s+1/2}(\Sigma_{-b}) \times H^{s+3/2}(\Sigma)$, but the range of $T_1$ lies in $Z^s$ by virtue of Proposition \ref{odp_CC2}.  Thus, $T_1 :  H^{s+2}(\Omega)\to Z^s$ is a well-defined and bounded linear map.  If $T_1 p =0$, then in particular $T_0 p =0$, where $T_0$ is the isomorphism from Theorem \ref{udln_iso}, and so $p=0$.  This means that $T_1$ is injective.  

Now let $(f,h_+,h_-,k) \in Z^s$.  Then Theorem \ref{udln_iso} allows us to set $p = T_0^{-1}(f,k,h_-) \in H^{s+2}(\Omega)$.  Set $H_+ = -\p_n p \vert_{\Sigma} \in H^{s+1/2}(\Sigma)$.  Then $p$ solves the over-determined problem
\begin{equation}
\begin{cases}
-\Delta p = f & \text{in } \Omega \\
p = k & \text{on } \Sigma \\
-\p_n p = H_+ &\text{on } \Sigma \\
-\p_n p = h_- &\text{on } \Sigma_{-b},
\end{cases}
\end{equation}
and so Proposition \ref{odp_CC2} tells us that 
\begin{equation}
 \int_{\Omega} f q - \int_{\Sigma} k \p_n q   + \int_{\Sigma_{-b}} h_- q = \int_{\Sigma} H_+ \psi
\end{equation} 
for every $\psi \in H^{s+3/2}(\Sigma)$, where $q = \Xi \psi \in H^{s+2}(\Omega)$ for $\Xi$ defined by Definition \ref{Xi_def}.  On the other hand, the compatibility condition on $(f,h_+,h_-,k)$ built into the definition of $Z^s$ requires that 
\begin{equation}
 \int_{\Omega} f q - \int_{\Sigma} k \p_n q   + \int_{\Sigma_{-b}} h_- q = \int_{\Sigma} h_+ \psi
\end{equation} 
for all such $\psi$ and $q$.  Equating these then shows that 
\begin{equation}
 \int_{\Sigma} (h_+ - H_+) \psi =0 \text{ for all } \psi \in H^{s+3/2}(\Sigma),
\end{equation}
from which we conclude that $h_+ = H_+$.  Hence $p$ solves \eqref{ODP}, or equivalently $T_1 p = (f,h_+,h_-,k)$.  Thus $T_1$ is surjective and so defines an isomorphism.
\end{proof}

\subsection{The  isomorphism for the pressure-free surface system}

Next we aim to show that the PDE system 
\begin{equation}\label{pfs}
\begin{cases}
 -\Delta p = f   & \text{in }\Omega \\
 -\p_n p -\p_n \pf \eta  +\gamma \p_1 \eta = h_+  &\text{on } \Sigma \\
 p  =   k &\text{on } \Sigma \\
 -\p_n p -\p_n \pf \eta   = h_- & \text{on } \Sigma_{-b}
\end{cases}
\end{equation}
induces an isomorphism between appropriate Banach spaces.   As a first step, in the next lemma we establish that the linear mapping associated to our PDE system actually takes values in $Y^s_1$ and is bounded.

\begin{lem}\label{T2_range}
Let $0 \le s \in \R$.  If $(p,\eta) \in H^{s+2}(\Omega) \times \H^{s+3/2}(\Sigma)$ then we have the inclusion
\begin{equation}\label{T2_range_00}
 (-\Delta p, -\p_n p \vert_{\Sigma}   -\p_n \pf \eta \vert_{\Sigma} + \gamma \p_1 \eta, -\p_n p \vert_{\Sigma_{-b}}  -\p_n \pf \eta \vert_{\Sigma_{-b}}, p\vert_{\Sigma}) \in Y^s_1,
\end{equation}
and 
\begin{equation}\label{T2_range_01}
 \norm{(-\Delta p, -\p_n p \vert_{\Sigma}  -\p_n \pf \eta \vert_{\Sigma}  + \gamma \p_1 \eta, -\p_n p \vert_{\Sigma_{-b}} -\p_n \pf \eta \vert_{\Sigma_{-b}} , p\vert_{\Sigma})}_{Y^s_1} \ls \norm{p}_{H^{s+2}} + \norm{\eta}_{\H^{s+3/2}}.
\end{equation}
\end{lem}
\begin{proof}
Write the tuple in \eqref{T2_range_00} as $(f,h_+,h_-,k)$.  From Theorems \ref{specialized_properties}, \ref{Ps_properties}, \ref{poisson_aniso_map} and standard trace theory we see that  
\begin{equation}
\norm{f}_{H^s} + \norm{h_+}_{H^{s+1/2}} + \norm{h_-}_{H^{s+1/2}} + \norm{k}_{H^{s+3/2}} \ls \norm{p}_{H^{s+2}} + \norm{\eta}_{\H^{s+3/2}}   ,
\end{equation}
and so in particular $(f,h_+,h_-,k)  \in H^s(\Omega) \times H^{s+1/2}(\Sigma)  \times H^{s+1/2}(\Sigma_{-b}) \times H^{s+3/2}(\Sigma)$.  

Suppose now that $\Gamma = \R^{n-1}$. Proposition \ref{odp_CC1} implies that 
\begin{equation}
 \snorm{\int_{-b}^0 f(\cdot,x_n) dx_n - (h_+ +  \p_n \pf \eta \vert_{\Sigma} - \gamma \p_1 \eta - h_- -  \p_n \pf \eta \vert_{\Sigma_{-b}}) }_{\dot{H}^{-1}} \ls \norm{p}_{H^{s+2}}.
\end{equation}
We know from Theorem \ref{specialized_properties} that  $\snorm{\p_1 \eta}_{\dot{H}^{-1}} \ls \norm{\eta}_{\H^{s+3/2}}$, and we know from Proposition \ref{normal_poisson_neg_est} that 
\begin{equation}
  \snorm{\p_n \pf\eta \vert_{\Sigma}  - \p_n \pf \eta\vert_{\Sigma_{-b}}  }_{\dot{H}^{-1}}   \ls  \norm{\eta}_{\H^{s+3/2}},
\end{equation}
so we deduce that 
\begin{equation}
 \snorm{\int_{-b}^0 f(\cdot,x_n) dx_n - (h_+  - h_-) }_{\dot{H}^{-1}} 
\ls \norm{p}_{H^{s+2}} + \norm{\eta}_{\H^{s+3/2}}.
\end{equation}
Thus $(f,h_+,h_-,k) \in Y^s_1$ and the estimate \eqref{T2_range_01} holds when $\Gamma = \R^{n-1}$.

Now consider the case $\Gamma =\T^{n-1}$.  In this case Proposition \ref{odp_CC1} shows that 
\begin{equation}
\int_{-b}^0 \hat{f}(0,x_n) dx_n - (\hat{h}_+(0) +  \widehat{\p_n \pf \eta \vert_{\Sigma}}(0) - \gamma \widehat{\p_1 \eta}(0) - \hat{h}_-(0) -  \widehat{\p_n \pf \eta \vert_{\Sigma_{-b}}}(0) ) =0,  
\end{equation} 
but Proposition \ref{normal_poisson_neg_est} shows $\widehat{\p_n \pf \eta \vert_{\Sigma}}(0) = \widehat{\p_1 \eta}(0) = \widehat{\p_n \pf \eta \vert_{\Sigma_{-b}}}(0)=0,$
so 
\begin{equation}
 \int_{-b}^0 \hat{f}(0,x_n) dx_n - (\hat{h}_+(0)  - \hat{h}_-(0) ) =0.
\end{equation}
Thus, $(f,h_+,h_-,k) \in Y^s_1$ and the estimate \eqref{T2_range_01} holds when $\Gamma = \T^{n-1}$.
\end{proof}

We can now state our isomorphism theorem associated to \eqref{pfs}.

\begin{thm}\label{T2_iso}
If $\Gamma = \R^{n-1}$, then assume that $\gamma \neq 0$.  Then the bounded linear map $T_2 : H^{s+2}(\Omega) \times \H^{s+3/2}(\Sigma) \to Y^s_1$ associated to \eqref{pfs}, which is defined by 
\begin{equation}
T_2(p,\eta) =  (-\Delta p, -\p_n p \vert_{\Sigma} -\p_n \pf \eta \vert_{\Sigma} + \gamma \p_1 \eta, -\p_n p \vert_{\Sigma_{-b}} -\p_n \pf \eta \vert_{\Sigma_{-b}}, p\vert_{\Sigma}), 
\end{equation}
is an isomorphism for every $0 \le s \in \R$.
\end{thm}
\begin{proof}
First note that Lemma \ref{T2_range} tells us that $T_2$ is a well-defined bounded linear map.  If $T_2(p,\eta) =0$, then 
\begin{multline}
0 = \int_{\Omega} \diverge(-\nab p - \nab \pf \eta) (p+\pf \eta) = \int_{\Omega} \abs{\nab p +\nab \pf \eta}^2 - \int_{\p \Omega} \p_\nu (p+\pf \eta) (p+ \pf \eta) \\
=  \int_{\Omega} \abs{\nab p + \nab \pf \eta}^2 + \int_{\Sigma} -\p_n (p+\pf \eta) (p+ \pf \eta) \\
= \int_{\Omega} \abs{\nab p + \nab \pf \eta}^2 + \int_{\Sigma}  -\gamma \p_1 \eta \eta 
= \int_{\Omega} \abs{\nab p + \nab \pf \eta}^2, 
\end{multline}
and so $p + \pf \eta =C$ for some constant $C \in \R$.  However, on $\Sigma$ we have that $p=0$ and $\pf \eta =\eta$, so $\eta =C$.  In turn this requires that $\eta =0$ (since $\eta\in \mathcal{H}^{s+3/2}(\Sigma)$) and $p =0$, and so $T_2$ is injective.

Now let $(f,h_+,h_-,k) \in Y^s_1$.  Define the function $\psi : \hat{\Gamma} \to \C$ via 
\begin{equation}\label{psi_def}
\psi(\xi) = \int_{-b}^0 \hat{f}(\xi, x_n) \frac{\cosh( \abs{\xi} (x_n+b))}{\cosh( \abs{\xi} b)} dx_n  - \hat{k}(\xi)  \abs{\xi} \tanh( \abs{\xi} b) 
- \hat{h}_+(\xi)  + \hat{h}_-(\xi) \sech( \abs{\xi} b). 
\end{equation}
Note that we may rewrite 
\begin{multline}
 \psi(\xi) = \int_{-b}^0 \hat{f}(\xi,x_n) dx_n - (\hat{h}_+(\xi) - \hat{h}_-(\xi)) 
+ \int_{-b}^0 \hat{f}(\xi, x_n) \left[\frac{\cosh( \abs{\xi} (x_n+b))}{\cosh(\abs{\xi} b)}-1 \right] dx_n \\
- \hat{k}(\xi)  \abs{\xi} \tanh( \abs{\xi} b) 
 + \hat{h}_-(\xi) \left[\sech( \abs{\xi} b) -1 \right].
\end{multline}
When $\Gamma = \R^{n-1}$, we readily deduce from this and standard Taylor expansion that 
\begin{multline}
 \int_{B(0,1)} \abs{\xi}^{-2} \abs{\psi(\xi)}^2 d\xi   \ls   \snorm{\int_{-b}^0 f(\cdot,x_n) dx_n + (h_+ - h_-) }_{\dot{H}^{-1}}^2 + \norm{f}_{L^2}^2 + \norm{h_-}_{L^2}^2 + \norm{k}_{L^2}^2 \\
\ls 
 \norm{(f,h_+,h_-,k)}_{Y^s_1}^2.
\end{multline}
Similarly, when $\Gamma = \T^{n-1}$ we must have that $\psi(0) =0$.  On the other hand, in both cases we can  bound 
\begin{equation}
 \int_{B(0,1)^c} (1+\abs{\xi}^2)^{s+1/2} \abs{\psi(\xi)}^2 d\xi 
\ls \norm{f}_{H^s}^2 + \norm{h_+}_{H^{s+1/2}}^2 + \norm{h_-}_{H^{s+1/2}}^2 + \norm{k}_{H^{s+3/2}}^2 
 \ls \norm{(f,h_+,h_-,k)}_{Y^s_1}^2.
\end{equation}
Combining these bounds shows that 
\begin{equation}\label{T2_iso_1}
 \int_{B(0,1)} \abs{\xi}^{-2} \abs{\psi(\xi)}^2 d\xi + \int_{B(0,1)^c} (1+\abs{\xi}^2)^{s+1/2} \abs{\psi(\xi)}^2 d\xi \ls \norm{(f,h_+,h_-,k)}_{Y^s_1}^2
\end{equation}
with the understanding that the first integral is replaced with $0$ when $\Gamma = \T^{n-1}$.

Next note that for $\xi \in \hat{\Gamma}$ we have that 
\begin{equation}\label{T2_iso_2}
 \abs{- i \gamma \xi_1 +  \abs{\xi} \tanh(\abs{\xi} b)}^2 =  \gamma^2 \xi_1^2 +  \abs{\xi}^2 \tanh^2( \abs{\xi}b) 
\asymp
\begin{cases}
\gamma^2 \xi_1^2 +  \abs{\xi}^4 b^2 & \text{for } \abs{\xi} \asymp 0 \\
(1+\gamma^2) \abs{\xi}^2 &\text{for } \abs{\xi} \asymp \infty,
\end{cases}
\end{equation}
and in particular the quantity on the left side vanishes if and only if $\xi =0$.  Consequently, we can define the measurable function $\hat{\eta} : \hat{\Gamma} \to \C$ via the identity
\begin{equation}\label{T2_iso_3}
\left[- i \gamma \xi_1 +  \abs{\xi} \tanh( \abs{\xi} b)\right] \hat{\eta}(\xi) = \psi(\xi)
\end{equation}
for $\xi \neq 0$ and $\hat{\eta}(0)=0$.  It may be easily checked that since the data are real-valued we have that $\overline{\psi(\xi)} = \psi(-\xi)$. The multiplier on the left side of \eqref{T2_iso_3} satisfies the same identity, and so we conclude that $\overline{\hat{\eta}(\xi)} = \hat{\eta}(-\xi)$, which means that $\eta$ is also real-valued.  Synthesizing \eqref{T2_iso_1} and \eqref{T2_iso_2}, we see from \eqref{T2_iso_3} that $\eta \in \H^{s+3/2}(\Sigma)$ and  
\begin{multline}\label{T2_iso_4}
\norm{\eta}_{\H^{s+3/2}}^2 \asymp \int_{B(0,1)}  \frac{\gamma^2 \xi_1^2 + \abs{\xi}^4 }{\abs{\xi}^2}  \abs{\hat{\eta}(\xi)}^2 d\xi + \int_{B(0,1)^c} (1+\abs{\xi}^2)^{s+3/2} \abs{\hat{\eta}(\xi)}^2 d\xi \\
\asymp 
 \int_{B(0,1)} \abs{\xi}^{-2} \abs{\psi(\xi)}^2 d\xi + \int_{B(0,1)^c} (1+\abs{\xi}^2)^{s+1/2} \abs{\psi(\xi)}^2 d\xi 
 \ls \norm{(f,h_+,h_-,k)}_{Y^s_1}^2,
\end{multline}
again with the understanding that the integrals over $B(0,1)$ are replaced by $0$ when $\Gamma = \T^{n-1}$, and recalling that $\mathcal{H}^{s+3/2}(\T^{n-1}) = H^{s+3/2}(\T^{n-1})$.

We now know that $\eta \in \H^{s+3/2}(\Sigma)$, so we can use Theorem \ref{poisson_aniso_map} to see that $\pf \eta \in \P^{s+2}(\Omega)$, as defined in Definition \ref{P_aniso_def}.  In particular, this,  Theorem \ref{Ps_properties}, and standard trace theory show that $\p_n \pf \eta \vert_{\Sigma} \in H^{s+3/2}(\Sigma)$ and $\p_n \pf \eta \vert_{\Sigma_{-b}} \in H^{s+3/2}(\Sigma_{-b})$.  Moreover, a simple computation shows that 
\begin{equation}\label{T2_iso_5}
 \widehat{ \p_n \pf \eta} \vert_{\Sigma}(\xi) =  \abs{\xi} \hat{\eta}(\xi) \text{ and }  \widehat{ \p_n \pf \eta} \vert_{\Sigma_{-b}}(\xi) =  \abs{\xi} e^{- \abs{\xi} b} \hat{\eta}(\xi)
\end{equation}
for $\xi \in \hat{\Gamma}$.  From these and the properties of $\H^{s+3/2}(\Sigma)$ given in Theorem \ref{specialized_properties}, we readily deduce that we have the inclusion
\begin{equation}
(f,h_+ -\gamma \p_1 \eta + \p_n \pf \eta \vert_\Sigma, h_-   + \p_n \pf \eta \vert_{\Sigma_{-b}}, k) \in  H^s(\Omega) \times H^{s+1/2}(\Sigma) \times H^{s+1/2}(\Sigma_{-b}) \times H^{s+3/2}(\Sigma).
\end{equation}
We claim that, in fact, this modified tuple belongs to the space $Z^s$.  To show this it suffices to check that the modified tuple satisfies the compatibility condition of Proposition \ref{odp_CC2}.  Using the identities \eqref{T2_iso_5}, we compute 
\begin{multline}
 - (\hat{h}_+(\xi) -\gamma \widehat{\p_1 \eta}(\xi) + \widehat{\p_n \pf \eta \vert_\Sigma}(\xi)) +  (\hat{h}_-(\xi)  + \widehat{\p_n \pf \eta \vert_{\Sigma_{-b}}}(\xi) ) \sech(\abs{\xi} b) \\
= - \hat{h}_+(\xi)  + \hat{h}_-(\xi) \sech( \abs{\xi} b)  - \left[- i \gamma \xi_1 +   \abs{\xi} \tanh( \abs{\xi} b) \right] \hat{\eta}(\xi).
\end{multline}
Thus, the identity 
\begin{multline}
0 = \int_{-b}^0 \hat{f}(\xi, x_n) \frac{\cosh( \abs{\xi} (x_n+b))}{\cosh(\abs{\xi} b)} dx_n  - \hat{k}(\xi)  \abs{\xi} \tanh( \abs{\xi} b) \\  - (\hat{h}_+(\xi) -\gamma \widehat{\p_1 \eta}(\xi) + \widehat{\p_n \pf \eta \vert_\Sigma}(\xi)) +  (\hat{h}_-(\xi)  + \widehat{\p_n \pf \eta \vert_{\Sigma_{-b}}}(\xi) ) \sech( \abs{\xi} b) 
\end{multline}
is equivalent to the identity \eqref{T2_iso_3}, which is satisfied by the construction of $\eta$.  Thus, for the modified tuple we have the inclusion 
\begin{equation}\label{T2_iso_6}
(f,h_+ -\gamma \p_1 \eta + \p_n \pf \eta \vert_\Sigma, h_-   + \p_n \pf \eta \vert_{\Sigma_{-b}}, k) \in Z^s
\end{equation}
as claimed.

In light of \eqref{T2_iso_6} and Theorem \ref{ODP_iso} we may then define 
\begin{equation}
p = T_1^{-1}(f,h_+ -\gamma \p_1 \eta + \p_n \pf \eta \vert_\Sigma, h_-   + \p_n \pf \eta \vert_{\Sigma_{-b}}, k) \in H^{s+2}(\Omega),
\end{equation}
which satisfies 
\begin{equation}
\begin{cases}
-\Delta p = f & \text{in } \Omega \\
p = k & \text{on } \Sigma \\
-\p_n p = h_+ -\gamma \p_1 \eta + \p_n \pf \eta \vert_\Sigma &\text{on } \Sigma \\
-\p_n p = h_-   + \p_n \pf \eta \vert_{\Sigma_{-b}} &\text{on } \Sigma_{-b}.
\end{cases}
\end{equation}
Thus, $T_2(p,\eta) = (f,h_+,h_-,k)$, and we conclude that $T_2$ is surjective and hence an isomorphism.
\end{proof}

\subsection{The  isomorphism for the velocity-pressure-free surface system}

Finally, we aim to show that the PDE system \eqref{linear_full} induces an isomorphism between appropriate Hilbert spaces.  First we must identify the domain and codomain by introducing two definitions.  The first defines a closed subspace of $H^s(\Omega;\R^n)$.

\begin{dfn}\label{nHs_def}
For $1/2 < s \in \R$ we define the space
\begin{equation}
 {_n}H^s(\Omega;\R^n) = \{u \in H^s(\Omega;\R^n) \st u_n \vert_{\Sigma_{-b}} =0 \}.
\end{equation}
Standard trace theory shows that this is a closed subspace of $H^s(\Omega;\R^n)$ and thus a Hilbert space.
\end{dfn}

The second definition introduces a container space for the data in the problem \eqref{linear_full}.

\begin{dfn}\label{Vs_def}
Let $0 \le s \in \R$.   For $\Gamma = \R^{n-1}$ we define the space
\begin{multline}
 V^s = \{ (F,G,H,K) \in H^{s+1}(\Omega;\R^n) \times H^{s}(\Omega) \times H^{s+1/2}(\Sigma)  \times H^{s+3/2}(\Sigma)  \st
\\ 
\int_{-b}^0 (G-\diverge{F})(\cdot,x_n) dx_n - (H - F_n\vert_{\Sigma}  + F_n\vert_{\Sigma_{-b}}) \in \dot{H}^{-1}(\Sigma)\}
\end{multline}
and endow it with the square norm
\begin{multline}
\norm{(F,G,H,K)}_{V^s}^s = \norm{F}_{H^{s+1}}^2 + \norm{G}_{H^s}^2 + \norm{H_+}_{H^{s+1/2}}^2  + \norm{K}_{H^{s+3/2}}^2 \\
+ \snorm{\int_{-b}^0 (G-\diverge{F})(\cdot,x_n) dx_n - (H - F_n\vert_{\Sigma}  + F_n\vert_{\Sigma_{-b}} )}_{\dot{H}^{-1}}^2
\end{multline}
and the associated inner-product.   On the other hand, for $\Gamma = \T^{n-1}$ we 
define the space
\begin{multline}
 V^s = \{ (F,G,H,K) \in H^{s+1}(\Omega;\R^n) \times H^{s}(\Omega) \times H^{s+1/2}(\Sigma)  \times H^{s+3/2}(\Sigma)  \st
\\ 
\int_{-b}^0 (\hat{G}-\widehat{\diverge{F}})(0,x_n) dx_n - (\hat{H}(0) - \hat{F}_n\vert_{\Sigma}(0)  + \hat{F}_n\vert_{\Sigma_{-b}}(0)) =0   \}
\end{multline}
and endow it with the square norm
\begin{equation}
\norm{(F,G,H,K)}_{V^s}^s = \norm{F}_{H^{s+1}}^2 + \norm{G}_{H^s}^2 + \norm{H_+}_{H^{s+1/2}}^2  + \norm{K}_{H^{s+3/2}}^2 
\end{equation}
and the associated inner-product.  It's easy to see that in both cases $V^s$ is a Hilbert space. 
\end{dfn}

\begin{remark}\label{Vs_remark}
Note that 
\begin{equation}
 \int_{-b}^0 -\diverge{F}(\cdot,x_n) dx_n = -\diverge'\int_{-b}^0 F'(\cdot,x_n) dx_n - F_n\vert_{\Sigma}  + F_n \vert_{\Sigma_{-b}}
\end{equation}
and so 
\begin{equation}
  \int_{-b}^0 -\diverge{F}(\cdot,x_n) dx_n + F_n \vert_{\Sigma} - F_n \vert{\Sigma_{-b}} = -\diverge'\int_{-b}^0 F'(\cdot,x_n) dx_n. 
\end{equation}
When $\Gamma = \R^{n-1}$ this provides the estimate 
\begin{equation}
 \snorm{  \int_{-b}^0 -\diverge{F}(\cdot,x_n) dx_n + F_n \vert_{\Sigma} - F_n \vert{\Sigma_{-b}}}_{\dot{H}^{-1}} \le \norm{F'}_{L^2},
\end{equation}
and this means that the term appearing in the $\dot{H}^{-1}$ seminorm in the definition of the $V^s$ norm can be replaced with 
\begin{equation}
\snorm{ \int_{-b}^0 G(\cdot,x_n) dx_n - H  }_{\dot{H}^{-1}}
\end{equation}
to produce an equivalent norm.  Similarly, when $\Gamma = \T^{n-1}$ these calculations show that a data tuple $(F,G,H,K) \in   H^{s+1}(\Omega;\R^n) \times H^{s}(\Omega) \times H^{s+1/2}(\Sigma)  \times H^{s+3/2}(\Sigma)$ belongs to $V^s$ if and only if 
\begin{equation}
 \int_{-b}^0 \hat{G}(0,x_n) dx_n - \hat{H}(0) =0.
\end{equation}
\end{remark}

Our next lemma shows that the linear map associated to \eqref{linear_full} takes values in $V^s$ and is bounded.

\begin{lem}\label{T3_range}
Let $0 \le s \in \R$.  Suppose $(u,p,\eta) \in {_n}H^{s+1}(\Omega;\R^n) \times H^{s+2}(\Omega) \times \H^{s+3/2}(\Sigma)$ and set 
\begin{equation}
F = u + \nab p + \nab \pf \eta, G = \diverge{u}, H  = u_n + \gamma \p_1 \eta,   \text{ and } K = p.
\end{equation}
Then $(F,G,H,K) \in V^s$, $(G - \diverge{F}, H - F_n \vert_{\Sigma} ,  - F_n\vert_{\Sigma_{-b}}, K) \in Y^s_1$, and we have the bound
\begin{equation}
 \norm{(F,G,H,K)}_{V^s}  + 
 \norm{(G - \diverge{F}, H - F_n \vert_{\Sigma} ,  - F_n\vert_{\Sigma_{-b}}, K)}_{Y^s_1} \ls \norm{u}_{H^{s+1}} + \norm{p}_{H^{s+2}} + \norm{\eta}_{\H^{s+3/2}}.
\end{equation}
\end{lem}
\begin{proof}
We may readily bound 
\begin{equation}
  \norm{F}_{H^{s+1}} + \norm{G}_{H^s} + \norm{H}_{H^{s+1/2}}  + \norm{K}_{H^{s+3/2}} \ls \norm{u}_{H^{s+1}} + \norm{p}_{H^{s+2}}  + \norm{\eta}_{\H^{s+3/2}} .
\end{equation}
On the other hand, if we define $f = G- \diverge{F} \in H^{s}(\Omega)$, $h_+ = H - F_n \vert_{\Sigma} \in H^{s+1/2}(\Sigma)$, $h_- = - F_n \vert_{\Sigma_{-b}} \in H^{s+1/2}(\Sigma_{-b})$, and $k=K \in H^{s+3/2}(\Sigma)$, then we see that 
\begin{equation} 
\begin{cases}
 -\Delta p = G-\diverge{F} = f  & \text{in }\Omega \\
 -\p_n p  -\p_n \pf \eta  +\gamma \p_1 \eta = H - F_{n}(\cdot,0) = h_+&\text{on } \Sigma \\
 p  =   K =k &\text{on } \Sigma \\
 -\p_n p -\p_n \pf \eta  =  -F_{n}(\cdot,-b) = h_-& \text{on } \Sigma_{-b},
\end{cases}
\end{equation}
and so Lemma \ref{T2_range} implies that $\norm{ (f,h_+,h_-,k) }_{Y^s_1} \ls  \norm{p}_{H^{s+2}}  + \norm{\eta}_{\H^{s+3/2}}$.  When $\Gamma = \R^{n-1}$,  the $\dot{H}^{-1}$ control provided by the $Y^s_1$ norm is exactly the $\dot{H}^{-1}$ control in the $V^s$ norm, and the stated estimate follows by summing our two bounds.  Similarly, when $\Gamma = \T^{n-1}$, the vanishing zero mode condition required for inclusion in $Y^s_1$ corresponds with the vanishing condition needed for inclusion in $V^s$. 
\end{proof}

Finally, we can state the isomorphism theorem for the map associated to \eqref{linear_full}.

\begin{thm}\label{T3_iso}
Assume that $0 \le s \in \R$, and if  $\Gamma = \R^{n-1}$, then assume that $\gamma \neq 0$.  Then the bounded linear map $T_3 : {_n}H^{s+1}(\Omega;\R^n) \times H^{s+2}(\Omega) \times \H^{s+3/2}(\Sigma) \to V^s$ associated to \eqref{linear_full}, which is defined by 
\begin{equation}
T_3(u,p,\eta) =  (u+\nab p +\nab \pf \eta, \diverge{u}, u_n \vert_{\Sigma} + \gamma \p_1 \eta, p\vert_{\Sigma}), 
\end{equation}
is an isomorphism. 
\end{thm}
\begin{proof}
Lemma \ref{T3_range} tells us that $T_3$ is well-defined and bounded.  If $T_3(u,p,\eta) =0$, then in particular $u+ \nab p + \nab \pf \eta =0$, and in turn this means that $T_2(p,\eta) = 0$.  Theorem \ref{T2_iso} then implies that $p=0$ and $\eta=0$, and so $u=0$ as well.  Thus, $T_3$ is injective.  

Now let $(F,G,H,K) \in V^s$.  Lemma \ref{T3_range} shows that 
\begin{equation}
(f,h_+,h_-,k) :=(G - \diverge{F}, H - F_n \vert_{\Sigma} , - F_n\vert_{\Sigma_{-b}}, K) \in Y^s_1,
\end{equation}
and so we may use Theorem \ref{T2_iso} to define $(p,\eta) = T_2^{-1}(f,h_+,h_-,k) \in H^{s+2}(\Omega) \times \H^{s+3/2}(\Sigma)$.  In other words, $(p,\eta)$ satisfy 
\begin{equation} 
\begin{cases}
 -\Delta p = G-\diverge{F}   & \text{in }\Omega \\
 -\p_n p -\p_n \pf \eta  +\gamma \p_1 \eta = H - F_{n}(\cdot,0) &\text{on } \Sigma \\
 p  =   K  &\text{on } \Sigma \\
 -\p_n p -\p_n \pf \eta = -F_{n}(\cdot,-b) & \text{on } \Sigma_{-b},
\end{cases}
\end{equation}
and upon setting $u = F - \nab p -\nab \pf \eta \in {_n}H^{s+1}(\Omega;\R^n)$  (where we have used Theorems \ref{specialized_properties}, \ref{Ps_properties}, \ref{poisson_aniso_map} to handle the $\nab \pf \eta$ term) we deduce that $T_3(u,p,\eta) = (F,G,H,K)$.  Thus, $T_3$ is surjective and so is an isomorphism.

\end{proof}

\section{Nonlinear analysis for the traveling wave system} \label{sec_nonlinear_tw}
 
Now we aim to invoke the implicit function theorem to solve \eqref{muskat_flattened}.

\subsection{The nonlinear mapping} 
 
To employ the implicit function theorem we first check that a number of basic nonlinear maps are well-defined.

\begin{prop}\label{AK_smooth}
Let $s > n/2 -1$.  Then there exists $\delta_0 >0$ such that the following hold.
\begin{enumerate}
 \item If $\eta \in \H^{s+3/2}(\Sigma)$ and $\norm{\eta}_{\H^{s+3/2}} < \delta_0$, then 
\begin{equation}\label{AK_smooth_01}
 \norm{b^{-1} \pf \eta + \tilde{b} \p_n \pf \eta}_{C^0_b} \le \hal,
\end{equation}
where $\tilde{b}(x) = 1+x_n/b$.  In particular, for such $\eta$ the functions $\kf$ and $\A$, defined in terms of $\eta$ via \eqref{JK_def} and \eqref{A_def}, are well-defined.  

 \item If $\eta \in \H^{s+3/2}(\Sigma)$ and $\norm{\eta}_{\H^{s+3/2}} < \delta_0$, then the flattening map $\ff_\eta : \Omega \to \Omega_\eta$ given by \eqref{flattening_def} is a $C^{1 + \lfloor s-n/2 +1 \rfloor}$  orientation-preserving diffeomorphism.

 \item For $\eta \in \H^{s+3/2}(\Sigma)$ such that $\norm{\eta}_{\H^{s+3/2}} < \delta_0$, the functions $\jf$ and $\kf$, given in terms of $\eta$ as in \eqref{JK_def}, define  $H^{s+1}(\Omega)$ multipliers.  Moreover, the maps 
\begin{equation}\label{AK_smooth_02}
\begin{split}
  B_{\H^{s+3/2}(\Sigma)}(0,\delta_0)   & \ni \eta \mapsto \jf \in \L(H^{s+1}(\Omega)) \text{ and } \\
  B_{\H^{s+3/2}(\Sigma)}(0,\delta_0)   & \ni \eta \mapsto \kf \in \L(H^{s+1}(\Omega))
\end{split}
\end{equation}
are smooth.

 \item For $\eta \in \H^{s+3/2}(\Sigma)$ such that $\norm{\eta}_{\H^{s+3/2}} < \delta_0$, the functions $\M$ and $\A$, given in terms of $\eta$ as in \eqref{M_def} and \eqref{A_def}, define  $H^{s+1}(\Omega;\R^n)$ multipliers.  Moreover, the maps 
\begin{equation}\label{AK_smooth_03}
\begin{split}
 B_{\H^{s+3/2}(\Sigma)}(0,\delta_0) &\ni \eta \mapsto \M \in \L(H^{s+1}(\Omega;\R^n)) \text{ and }\\
 B_{\H^{s+3/2}(\Sigma)}(0,\delta_0) &\ni \eta \mapsto \A \in \L(H^{s+1}(\Omega;\R^n)) 
\end{split}
\end{equation}
are smooth.
\end{enumerate}
\end{prop}
\begin{proof}
We will only provide the proof for the case $\Gamma = \R^{n-1}$, as the case $\Gamma = \T^{n-1}$ is similar but simpler.  Since $s+1 > n/2$, the existence of a $\delta_1 >0$ such that $\norm{\eta}_{\H^{s+3/2}} < \delta_1$ implies the
bound \eqref{AK_smooth_01} follows readily from the results in Theorems \ref{Ps_properties} and \ref{poisson_aniso_map}, which show that $\pf \eta \in \P^{s+2}(\Omega) \hookrightarrow C^1_b(\Omega)$.  With this bound in hand, we may appeal to \eqref{JK_def} to write 
\begin{equation}\label{AK_smooth_1}
 \kf = \jf^{-1} =  \left(1 + b^{-1} \pf \eta + \tilde{b}\p_n \pf \eta \right)^{-1} = \sum_{m=0}^\infty (-1)^m (b^{-1} \pf \eta + \tilde{b}\p_n \pf \eta)^m.
\end{equation}
In turn, \eqref{M_def} allows us to write 
\begin{equation}\label{AK_smooth_2}
 \mathcal{B} :=  \M -I = (-\kf \tilde{b} \nab' \pf \eta, \kf-1) \otimes e_n,
\end{equation}
and we conclude that $\kf$ and $\B$, and hence $\jf= \kf^{-1}$, $\M =I + \mathcal{B}$, and $\A = \jf(I+\mathcal{B})^T (I+ \mathcal{B})$ are well-defined when $\norm{\eta}_{\H^{s+3/2}} < \delta_1$.  

Next we use Theorem \ref{Ps_product_supercrit} (again noting that $s+1> n/2$) to see that the power series 
\begin{equation}
 \P^{s+1}(\Omega) \ni z \mapsto \sum_{m=0}^\infty (-1)^m z^m \in \L(H^{s+1}(\Omega))
\end{equation}
converges and defines an analytic function for $\norm{z}_{\P^{s+1}} < \delta_2$, for some $\delta_2 >0$.  Again employing Theorems  \ref{Ps_properties} and \ref{poisson_aniso_map}, we may choose $\delta_3 >0$ such that $\norm{\eta}_{\H^{s+3/2}} < \delta_3$ implies that $\norm{b^{-1} \pf \eta + \tilde{b}\p_n \pf \eta}_{\P^{s+1}} < \delta_2$. 

Set $\delta_0 = \min\{\delta_1,\delta_2,\delta_3\}$.  Then for $\norm{\eta}_{\H^{s+3/2}} < \delta_0$ we have that $\kf$, $\jf$, $\M$, and $\A$ are well-defined pointwise, and the formulas \eqref{AK_smooth_1} and \eqref{AK_smooth_2} then show that the maps given in \eqref{AK_smooth_02} and \eqref{AK_smooth_03} are smooth.

Finally, suppose $\norm{\eta}_{\H^{s+3/2}} < \delta_0$.  Then according to Theorems \ref{poisson_aniso_map} and \ref{Ps_properties}, the map $\ff_\eta$ is $C^{1 + \lfloor s-n/2+1 \rfloor }$.  Moreover, the bound \eqref{AK_smooth_01} implies that for each $x' \in \R^{n-1}$, the map $(-b,0) \ni x_n \mapsto e_n \cdot \ff_\eta(x',x_n) \in (-b, \eta(x'))$ is increasing since its derivative is greater than $1/2$ everywhere.  From this and the fact that $\ff_\eta(x)' = x'$ we conclude that $\ff_\eta$ is a bijection from $\Omega$ to $\Omega_\eta$. On the other hand, the bound \eqref{AK_smooth_01} also shows that $\det \nab \ff_\eta(x) \ge 1/2$ for $x \in \Omega$, so $\ff_\eta$ is a $C^{1 + \lfloor s-n/2 +1\rfloor}$ diffeomorphism by virtue of the inverse function theorem.
\end{proof}

We next introduce some useful notation.

\begin{dfn}
Let $0 \le s \in \R$ and $V$ be a finite dimensional real inner-product space.  We define the  bounded linear map $L_\Omega : H^s(\Gamma;V) \to H^s(\Omega;V)$ via $L_\Omega f(x) = f(x')$.
\end{dfn}

The next theorem verifies that the nonlinear maps associated to the problem \ref{muskat_flattened} are well-defined and $C^1$, which is essential for our subsequent use of the implicit function theorem.

\begin{thm}\label{nlin_well_def}
Let $n/2-1 < s \in \N$ and for $\delta >0$ define the set 
\begin{equation}
 U^s_\delta = \{(u,p,\eta) \in {_n}H^{s+1}(\Omega;\R^n) \times H^{s+2}(\Omega) \times \H^{s+3/2}(\Sigma) \st \norm{\eta}_{\H^{s+3/2}} < \delta \}.
\end{equation}
There exists a constant $\delta >0$ such that if $\gamma \in \R$, $\varphi_0 \in H^{s+3/2}(\Sigma)$, $\varphi_1 \in H^{s+3}(\Gamma \times \R)$, $\mathfrak{f}_0 \in H^{s+1}(\Omega;\R^n)$, $\mathfrak{f}_1 \in H^{s+2}(\Gamma \times \R;\R^n)$, and $(u,p,\eta) \in U^s_\delta$, and we define $F: \Omega \to \R^n$, $G: \Omega \to \R$, and $H,K : \Sigma \to \R$ via
\begin{equation}
\begin{aligned}
F &= u + \naba p + \naba \pf \eta  - \jf \M^T \left[  L_\Omega \mathfrak{f}_0 + \mathfrak{f}_1 \circ \ff_\eta \right],  &&&
G &=  \diverge{u}, \\
H &= u_n  + \gamma \p_1 \eta, &&&
K &= p -\varphi_0 - \varphi_1 \circ \ff_\eta \vert_{\Sigma}, 
\end{aligned}
\end{equation}
where $\ff_\eta$, $\jf$, $\M$, and $\A$ are determined by $\eta$ via \eqref{flattening_def}, \eqref{JK_def}, \eqref{M_def}, and \eqref{A_def}, then $(F,G,H,K) \in V^s$, where $V^s$ is as in Definition \ref{Vs_def}.  Moreover, the map 
\begin{equation}
\Psi:  \R \times H^{s+3/2}(\Sigma) \times H^{s+3}(\Gamma \times \R) \times H^{s+1}(\Omega;\R^n) \times H^{s+2}(\Gamma \times \R;\R^n) \times U^s_\delta  \to V^s
\end{equation}
defined by $\Psi(\gamma, \varphi_0, \varphi_1, \mathfrak{f}_0, \mathfrak{f}_1, u,p,\eta) = (F,G,H,K)$ is $C^1$.
\end{thm}
\begin{proof}
Again, we will only write the proof for the case $\Gamma = \R^{n-1}$, as the case $\Gamma = \T^{n-1}$ is similar but simpler.

Let $\delta >0$ be the smaller of $\delta_0>0$ from Proposition \ref{AK_smooth} and $\delta_\ast >0$ from Corollary \ref{Feta_comp}.  Proposition \ref{AK_smooth}, Theorems \ref{Ps_properties} and \ref{poisson_aniso_map}, and the first item of Corollary \ref{Feta_comp}, applied with $r = \sigma = s+1$, show that the map 
\begin{equation}
 (\mathfrak{f}_0, \mathfrak{f}_1, u,p,\eta) \mapsto  u + \naba p + \naba \pf \eta  - \jf \M^T \left[  L_\Omega \mathfrak{f}_0 + \mathfrak{f}_1 \circ \ff_\eta \right] \in H^{s+1}(\Omega;\R^n)
\end{equation}
is well-defined and $C^1$.  Next, we note that the maps 
\begin{equation}
 u \mapsto \diverge{u} \in H^s(\Omega) \text{ and } (\gamma,u,\eta) \mapsto u_n\vert_{\Sigma} + \gamma \partial_1 \eta \in H^{s+1/2}(\Sigma) 
\end{equation}
are smooth since the former is linear and the latter is a sum of the linear trace map and a  quadratic map.  Finally, Proposition \ref{AK_smooth} and the second item of Corollary \ref{Feta_comp}, applied with $r = \sigma +1 = s+2$, show that the map $(\varphi_0,\varphi_1,p,\eta) \mapsto  p -\varphi_0 - \varphi_1 \circ \ff_\eta \vert_{\Sigma}$ is $C^1$.  These combine to show initially that $\Psi$ is well-defined and $C^1$ as a map into $H^{s+1}(\Omega;\R^n) \times H^{s}(\Omega) \times H^{s+1/2}(\Sigma)  \times H^{s+3/2}(\Sigma)$.

On the other hand, the map 
\begin{equation}
(\gamma,u,\eta) \mapsto  \int_{-b}^0 (\diverge{u})(\cdot,x_n) dx_n - (u_n\vert_{\Sigma} + \gamma \partial_1 \eta)
= \diverge' \int_{-b}^0 u'(\cdot,x_n) dx_n - \gamma \partial_1 \eta
\in \dot{H}^{-1}(\Sigma) 
\end{equation}
is well-defined (thanks to Theorem \ref{specialized_properties}) and quadratic,  and thus smooth.  Combining the above observations with Remark \ref{Vs_remark}, we conclude that $\Psi$ is actually a $C^1$ mapping into the space $V^s$.
\end{proof}

\subsection{Invoking the implicit function theorem: proof of the main existence theorem} \label{sec_tw_exist}

Finally, we are ready to invoke the implicit function theorem to prove the existence of solutions to the system \eqref{muskat_flattened}.

\begin{proof}[Proof of Theorem \ref{well_posedness_flat:intro}]

For the sake of brevity, write $X^s = \R \times H^{s+3/2}(\Sigma) \times H^{s+3}(\Gamma \times \R) \times H^{s+1}(\Omega;\R^n) \times H^{s+2}(\Gamma \times \R;\R^n)$ and $W^s ={_n}H^{s+1}(\Omega;\R^n) \times H^{s+2}(\Omega) \times \H^{s+3/2}(\Sigma)$ and note that $U^s_\delta \subseteq W^s$ is an open subset.  Let  $\Psi: X^s  \times U^s_\delta  \to V^s$ be the $C^1$ map given in Theorem \ref{nlin_well_def} and note that  a given tuple $(\gamma, \varphi_0, \varphi_1, \mathfrak{f}_0, \mathfrak{f}_1, u,p,\eta) \in  X^s \times U^s_\delta$
satisfies \eqref{well_posedness_flat_02:intro} if and only if $\Psi (\gamma, \varphi_0, \varphi_1, \mathfrak{f}_0, \mathfrak{f}_1, u,p,\eta) = (0,0,0,0) \in V^s$.

Given the product structure $\Psi : X^s \times U^s_\delta \to V^s$, we can construct the partial derivatives of $\Psi$ with respect to each factor:
\begin{equation}
 D_1 \Psi : X^s \times U^s_\delta \to \L(X^s; V^s) \text{ and } D_2 \Psi: X^s \times U^s_\delta \to \L(W^s; V^s).
\end{equation}
It is then a simple matter to check that for $\gamma \in \mathfrak{C}$
\begin{equation}
 \Psi(\gamma,0,0,0,0,0,0,0) = (0,0,0,0) 
 \text{ and that }
 D_2 \Psi(\gamma,0,0,0,0,0,0,0) = T_3,
\end{equation}
where $T_3 : W^s \to V^s$ is the isomorphism constructed in Theorem \ref{T3_iso}.  For any $\gamma_\ast \in  \mathfrak{C}$ we may then employ the implicit function theorem to find open sets $\mathcal{D}^s(\gamma_\ast) \subseteq X^s$ and $\mathcal{S}^s(\gamma_\ast) \subseteq U^s_\delta$ and a $C^1$ and Lipschitz function $\Xi_{\gamma_\ast} : \mathcal{D}^s(\gamma_\ast) \to \mathcal{S}^s(\gamma_\ast)$ such that $(\gamma_\ast,0,0,0,0) \in \mathcal{D}^s(\gamma_\ast)$, $(0,0,0) \in \mathcal{S}^s(\gamma_\ast)$, and 
\begin{equation}
\Psi(\gamma, \varphi_0,\varphi_1, \mathfrak{f}_0, \mathfrak{f}_1, \Xi_{\gamma_\ast}(\gamma, \varphi_0,\varphi_1, \mathfrak{f}_0, \mathfrak{f}_1)) = (0,0,0,0)
\end{equation}
for all $(\gamma, \varphi_0,\varphi_1, \mathfrak{f}_0, \mathfrak{f}_1) \in \mathcal{D}^s(\gamma_\ast)$.  The implicit function theorem also guarantees that $\Xi_{\gamma_\ast}$ parameterizes the unique such solution within $\mathcal{S}^s(\gamma_\ast)$.

We now define the open sets 
\begin{equation}
 \mathcal{D}^s = \bigcup_{\gamma_\ast \in \mathfrak{C}}  \mathcal{D}^s(\gamma_\ast) \subseteq X^s \text{ and }  \mathcal{S}^s = \bigcup_{\gamma_\ast \in \mathfrak{C}}  \mathcal{S}^s(\gamma_\ast) \subseteq U^s_\delta.
\end{equation}
By construction, we have the inclusions listed in the first item.  Using the above analysis, we can define the map $\Xi : \mathcal{D}^s \to \mathcal{S}^s$ via $\Xi(\gamma, \varphi_0,\varphi_1, \mathfrak{f}_0, \mathfrak{f}_1) = \Xi_{\gamma_\ast}(\gamma, \varphi_0,\varphi_1, \mathfrak{f}_0, \mathfrak{f}_1)$ whenever $(\gamma, \varphi_0,\varphi_1, \mathfrak{f}_0, \mathfrak{f}_1) \in \mathcal{D}^s(\gamma_\ast)$.  This is well-defined, $C^1$, and locally Lipschitz by the above consequences of the implicit function theorem.  The second and third items then follow by setting $(u,p,\eta) = \Xi(\gamma, \varphi_0,\varphi_1, \mathfrak{f}_0, \mathfrak{f}_1)$.

\end{proof}

\section{Analysis of the Dirichlet-Neumann operator}\label{Section:DN}
This section is devoted to the analysis of the Dirichlet-Neumann operator $G(\eta)$ when $\eta$ is small in $\cH^s$. Since we will  primarily work with the horizontal coordinates, it is more convenient  to denote a point in $\Omega_\eta$ by $(x, y)$, where $x\in \R^d$, $d=n-1\ge 1$, and $y\in \R$. Then, we recall that 
\begin{equation}
[G(\eta)f](x)=N(x)\cdot(\nab_{x, y}\psi)(x, \eta(x)),\quad N(x)=(-\nab f(x), 1),
\end{equation}
where $\psi$ solves the problem 
 \begin{equation}
\begin{cases}
\Delta_{x, y} \psi=0 &\text{in } \Omega_\eta,\\
\psi=f  &\text{on } \Sigma_\eta,\\
\p_y\psi=0 &\text{on }  \Sigma_{-b}.
\end{cases}
\end{equation}
 
We straighten the domain $\Omega_\eta=\{(x, y)\in \R^d\times \R: -b<y<\eta(x)\}$ using the mapping 
\begin{equation}
\mathfrak{F}_\eta:\quad (x, z)\ni \Omega_\eta=M^d\times (-b, 0)\mapsto (x, \varrho(x, z))\in \Omega_\eta,\quad \varrho(x, z)=\frac{z+b}{b}e^{z|D|}\eta(x)+z.
\end{equation}
Note that $\mathfrak{F}_\eta$ is the mapping \eqref{flattening_def} written our new notation. Since 
\begin{equation}
\p_z\varrho(x, z)=\frac{1}{b}e^{z|D|}\eta(x)+\frac{z+b}{b}e^{z|D|}|D|\eta(x)+1, 
\end{equation}
if 
 \begin{equation}
\| e^{z|D|}\eta\|_{L^\infty}+ b\| e^{z|D|}|D|\eta\|_{L^\infty}<b
 \end{equation}
 then $\mathfrak{F}_\eta$ is a Lipschitz diffeomorphism. A direct calculation shows that if $g:\Omega_\eta\to \R$ then $\wt g(x, z):=(g\circ \mathfrak{F}_\eta)(x, z)=g(x, \varrho(x, z))$ satisfies 
 \begin{equation}
 \di_{x, z}(\mathcal{A}\nab_{x, z}\wt g)(x, z)=\p_z\varrho(\Delta_{x, y}g)(x, \varrho(x)),
 \end{equation}
 where
 \begin{equation}
\mathcal{A}= \begin{bmatrix} 
 \p_z\varrho I_d& -\nab_x\varrho \\ -(\nab_x\varrho)^T & \frac{1+|\nab_x\varrho|^2}{\p_z\varrho}
 \end{bmatrix},
 \end{equation}
 $I_d$ being the $d\times d$ identity matrix.   $\mathcal{A}$ is the matrix  \eqref{A_def} written in our new notation. Since $\psi$ is harmonic in $\Omega_\eta$, $v=\psi\circ \mathfrak{F}_\eta$ satisfies  $ \di_{x, z}(\mathcal{A}\nab_{x, z}v)=0$ in $\Omega_\eta$. We write $\mathcal{A}$ as a perturbation of the identity  matrix
  \begin{equation}
 \mathcal{A}=I_{d+1}+\begin{bmatrix}\frac{1}{b}e^{z|D|}\eta(x)+\frac{z+b}{b}e^{z|D|}|D|\eta(x)& -\frac{z+b}{b}e^{z|D|}\nab\eta \\  -\frac{z+b}{b}e^{z|D|}\nab\eta^T & \frac{(\frac{z+b}{b})^2|e^{z|D|}\nab\eta|^2-\frac{1}{b}e^{z|D|}\eta(x)-\frac{z+b}{b}e^{z|D|}|D|\eta(x)}{\frac{1}{b}e^{z|D|}\eta(x)+\frac{z+b}{b}e^{z|D|}|D|\eta(x)+1},\end{bmatrix}
  \end{equation}
Consequently, $v$ satisfies 
  \begin{equation}\label{Deltav}
  \Delta_{x, z}v=\p_zQ_a[v]+\di_x Q_b[v],
  \end{equation}
  where
\begin{equation}
\begin{aligned}
&Q_a[v]=\frac{z+b}{b}e^{z|D|}\nab\eta\cdot \nab_x v-\frac{(\frac{z+b}{b})^2|e^{z|D|}\nab\eta|^2-\frac{1}{b}e^{z|D|}\eta(x)-\frac{z+b}{b}e^{z|D|}|D|\eta(x)}{\frac{1}{b}e^{z|D|}\eta(x)+\frac{z+b}{b}e^{z|D|}|D|\eta(x)+1}\p_zv,\\
&Q_b[v]=-\left(\frac{1}{b}e^{z|D|}\eta(x)+\frac{z+b}{b}e^{z|D|}|D|\eta(x)\right)\nab_xv+\frac{z+b}{b}e^{z|D|}\nab\eta\p_zv.
\end{aligned}
\end{equation}
Setting 
\begin{equation}
\D(z)=|D|\tanh((z+b)|D|),
\end{equation}
we decompose $\Delta_{x, z}=(\p_z+\D(z))(\p_z-\D(z))$. Then, \eqref{Deltav} is equivalent to the following system of forward and backward parabolic equations
\begin{align}\label{eq:w}
&(\p_z+\D(z))w=\di_x Q_b[v]-\D(z)Q_a[v],\\ \label{eq:v}
&(\p_z-\D(z))v=w+Q_a[v].
\end{align}
Since $\D(-b)=0$ and $\p_zv(x, -b)=\p_z\varrho(x, -b)\p_y\phi(x, -b)=0$, we have $Q_a[v](x, -b)$ and $w(x, -b)=0$.  

By the chain rule and \eqref{eq:v}, the Dirichlet-Neumann operator can be written in terms of $f$ and $w$ as 
\begin{equation}\label{DN:v}
\begin{aligned}
G(\eta)f&=\Big(-\nab_x\varrho\cdot \nab_xv+\frac{1+|\nab_x\varrho|^2}{\p_z\varrho}\p_zv\Big)\vert_{z=0}
=\p_zv\vert_{z=0}+\Big(-\nab_x\varrho\cdot \nab_xv+\big(\frac{1+|\nab_x\varrho|^2}{\p_z\varrho}-1\big)\p_zv\Big)\vert_{z=0}\\
&=\Big(\D(z)v+w+Q_a[v]\Big)\vert_{z=0}-Q_a[v]\vert_{z=0}
=m(D)f+w\vert_{z=0},
\end{aligned}
\end{equation}
where we have denoted
\begin{equation}
m(D)=|D|\tanh(b|D|).
\end{equation}
Using the identities 
\begin{equation}\begin{aligned}
&(\p_z+\D(z))w=[\cosh((z+b)|D|)]^{-1}\p_z\Big\{\cosh((z+b)|D|) w\Big\}\\
&(\p_z-\D(z))v=\cosh((z+b)|D|)\p_z\Big\{[\cosh((z+b)|D|)]^{-1}v\Big\},
\end{aligned}\end{equation}
we can integrate \eqref{eq:w} and \eqref{eq:v} to obtain 
\begin{align}\label{form:w}
&w(z)=\int_{-b}^z\frac{\cosh((z'+b)|D|)}{\cosh((z+b)|D|)}\left\{\di_x Q_b[v](z')-\D(z)Q_a[v](z')\right\}dz',\\\label{form:v}
&v(z)=\frac{\cosh((z+b)|D|)}{\cosh(b|D|)}f-\int_z^0 \frac{\cosh((z+b)|D|)}{\cosh((z'+b)|D|)}\left\{w(z')+Q_a[v](z')\right\}dz'
\end{align}
for $z\in [-b, 0]$. It then follows from \eqref{DN:v} that 
\begin{align}
&G(\eta)f=m(D)f+R(\eta)f,\\\label{def:RDN}
&R(\eta)f=\int_{-b}^0\frac{\cosh((z'+b)|D|)}{\cosh(b|D|)}\left\{\di_x Q_b[v](z')-|D|\tanh(b|D|)Q_a[v](z')\right\}dz'.
\end{align}
Clearly $R(\eta)f$ is a derivative. 

We denote $I=[-b, 0]$ and 
\begin{equation}
U^r=\wL^\infty(I; H^r(M^d))\cap \wL^1(I; H^{r+1}(M^d)),
\end{equation}
where the definition of the Chemin-Lerner spaces $\wL^p H^s$ is recalled in Definition \ref{defi:Chermin-Lerner}. In Appendix \ref{Appendix:LlittlewoodPaley}, we establish estimates in Chemin-Lerner norms for the operators appearing in \eqref{form:w} and \eqref{form:v}. The next proposition uncovers the low-frequency structure and provides the boundedness of the remainder $R(\eta)$. 
\begin{prop}\label{prop:estvw}
 Let $\sigma\ge \sigma_0> 1+\frac{d}{2}$. There exist a small positive constant $c_1=c_1(\sigma, \sigma_0, b, d)$ such that if $\|  \eta\|_{\cH^{\sigma_0}}<c_1$  then the following assertions hold.\\
1) We have
\begin{equation}\label{est:nav} \| \nab_{x,z}v\|_{U^{\sigma-1}}\le   C\| \nab_xf\|_{H^{\sigma-1}}+C\| \eta\|_{\cH^{\sigma}}\| \nab_xf\|_{H^{\sigma_0-1}},
\end{equation}
where and $C=C(\sigma, \sigma_0, b, d)$. \\
2) For any continuous symbol $\ell:\R\to \R$ satisfying 
\begin{equation}\label{symbol:ell}
\ell(\xi)\asymp 
\begin{cases} 
|\xi| &\text{if }  |\xi|\asymp \infty,\\
|\xi|^2 &\text{if }  |\xi|\asymp 0\end{cases},
\end{equation}
we have
\begin{equation}\label{est:RDN}
\| \ell^{-\hal}(D) R(\eta)f\|_{H^{\sigma-\hal}(\R^d)}\le C\| \eta\|_{\cH^{\sigma_0}(\R^d)}\| \nab_xf\|_{H^{\sigma-1}(\R^d)}+ C\| \eta\|_{\cH^\sigma(\R^d)}\| \nab_xf\|_{H^{\sigma_0-1}(\R^d)},
\end{equation}
where $C=C(\ell, \sigma, \sigma_0, b, d)$. Moreover, for $\eta,~f: \T^d\to \R$ we have $\wh{R(\eta)f}(0)=0$ and
\begin{equation}\label{est:RDN1}
\| |D|^{-\hal}R(\eta)f\|_{H^{\sigma-\hal}(\T^d)}\le C\| \eta\|_{H^{\sigma_0}(\T^d)}\| \nab_xf\|_{H^{\sigma-1}(\T^d)}+ C\| \eta\|_{H^\sigma(\T^d)}\| \nab_xf\|_{H^{\sigma_0-1}(\T^d)},
\end{equation}
where $C=C(\sigma, \sigma_0, b, d)$. 
\end{prop}
\begin{proof}
 Applying the estimates \eqref{estop1} and \eqref{estop2} to \eqref{form:v} and \eqref{form:w}, we obtain
\begin{equation}\label{estv:1}
\begin{aligned}
\|\nab_xv\|_{U^{\sigma-1}}&\le C\| \nab_xf\|_{H^{\sigma-1}}+C\| \nab_xw\|_{\wL^1(I; H^{\sigma-1})}+C\| \nab_xQ_a[v]\|_{\wL^1(I; H^{\sigma-1})}\\
&\le C\| \nab_xf\|_{H^{\sigma-1}}+C\| w\|_{\wL^1(I; H^{\sigma})}+C\| Q_a[v]\|_{\wL^1(I; H^{\sigma})}
\end{aligned}
\end{equation}
and 
\begin{equation}\label{estw:1}
\begin{aligned}
\| w\|_{U^{\sigma-1}}&\le C\| \di_x Q_b[v]\|_{\wL^1(I; H^{\sigma-1})}+ C\| \D(z)Q_a[v]\|_{\wL^1(I; H^{\sigma-1})}\\
&\le  C\|  Q_b[v]\|_{\wL^1(I; H^{\sigma})}+ C\|Q_a[v]\|_{\wL^1(I; H^{\sigma})},
\end{aligned}
\end{equation}
where $C=C(b)$.  It follows that 
\begin{equation}\label{estv:2}
\begin{aligned}
\|\nab_xv\|_{U^{\sigma-1}}&\le C\| \nab_xf\|_{H^{\sigma-1}}+C\|  Q_a[v]\|_{\wL^1(I; H^{\sigma})}+ C\|Q_b[v]\|_{\wL^1(I; H^{\sigma})},\quad C=C(b).
\end{aligned}
\end{equation}
On the other hand,  equation \eqref{eq:v} gives $\p_zv= \D(z)v+w+Q_a[v]$. Since $\D(z)=|D|\tanh((z+b)|D|)$ and $0\le \tanh((z+b)|\xi|)\le 1$ for $z\in [-b, 0]$, we obtain $\| \D(z)v\|_{U^{\sigma-1}}\le \| \nab_xv\|_{U^{\sigma-1}}$. Combining this with \eqref{estv:1} and \eqref{estw:1}, we deduce 
\begin{equation}\label{estv:3}
\|\p_zv\|_{U^{\sigma-1}}\le  C\| \nab_xf\|_{H^{\sigma-1}}+C\|  Q_a[v]\|_{U^{\sigma-1}}+ C\|Q_b[v]\|_{U^{\sigma-1}},\quad C=C(b).
\end{equation}
For $\sigma\ge \sigma_0>1+\frac{d}{2}$ and $\|  \eta\|_{\cH^{\sigma_0}}<c_1$ small enough, we can apply the product estimate \eqref{ppestcH:2} and the nonlinear estimate \eqref{nonl:cH} to obtain
\begin{equation}\label{est:Qab}
\| (Q_a[v], Q_b[v])\|_{U^{\sigma-1}}\le \cF(\| \eta\|_{\cH^{\sigma_0}})\left\{\| \eta\|_{\cH^{\sigma_0}}\| \nab_{x,z}v\|_{U^{\sigma-1}}+\| \eta\|_{\cH^{\sigma}}\| \nab_{x,z}v\|_{U^{\sigma_0-1}}\right\},
\end{equation}
where $\cF:\R^+\to\R^+$ is  nondecreasing and depends only on $(\sigma, \sigma_0, b, d)$.

A combination of \eqref{estv:2}, \eqref{estv:3} and \eqref{est:Qab} yields 
\begin{equation}\label{est:nav:4}
 \| \nab_{x,z}v\|_{U^{\sigma-1}}\le C\| \nab_xf\|_{H^{\sigma-1}}+ \cF(\| \eta\|_{\cH^\sigma_0})\left\{\| \eta\|_{\cH^{\sigma_0}}\| \nab_{x,z}v\|_{U^{\sigma-1}}+\| \eta\|_{\cH^{\sigma}}\| \nab_{x,z}v\|_{U^{\sigma_0-1}}\right\},
\end{equation}
where $C=C(\sigma, \sigma_0, b, d)$.  Applying \eqref{est:nav:4} with $\sigma=\sigma_0$ we deduce that there exists $c_0=c_0(\sigma_0, b, d)>0$ small enough such that if $\| \eta\|_{\cH^{\sigma_0}}<c_0$, then $\| \nab_{x,z}v\|_{U^{\sigma_0-1}}\le C(\sigma_0, b, d)\| \nab_xf\|_{H^{\sigma_0-1}}$. Inserting this into \eqref{est:nav:4} yields
\begin{equation}
\| \nab_{x,z}v\|_{U^{\sigma-1}}\le C\| \nab_xf\|_{H^{\sigma-1}}+ C\| \eta\|_{\cH^{\sigma_0}}\| \nab_{x,z}v\|_{U^{\sigma-1}}+C\| \eta\|_{\cH^{\sigma}}\| \nab_xf\|_{H^{\sigma_0-1}},\quad C=C(\sigma, \sigma_0, b, d).
\end{equation}
Therefore, for some $c_1=c_1(\sigma, \sigma_0, b, d)\le c_0$ small enough, we have
\begin{equation}\label{estv:6}
\| \nab_{x,z}v\|_{U^{\sigma-1}}\le C\| \nab_xf\|_{H^{\sigma-1}}+C\| \eta\|_{\cH^{\sigma}}\| \nab_xf\|_{H^{\sigma_0-1}}
\end{equation}
provided that $\| \eta\|_{\cH^{\sigma_0}}<c_1$. This concludes the proof of \eqref{est:nav}. 

We turn to prove  \eqref{est:RDN}.  Using the formula \eqref{def:RDN} and the  estimate \eqref{estop3} we obtain 
\begin{equation}
\| \ell^{-\hal}(D)R(\eta)f\|_{H^{\sigma-\hal}}\ls \| \ell^{-\hal}(D)\di_x Q_b[v]\|_{\wL^1(I;  H^{\sigma-\hal})}+ \| \ell^{-\hal}(D)|D|\tanh(b|D|)Q_a[v]\|_{\wL^1(I; H^{\sigma-\hal})}
\end{equation}
Noticing that 
\begin{equation}
\ell^{-\hal}(\xi)|\xi|\asymp 
\begin{cases} 
1 & \text{if }  |\xi|\asymp 0,\\
|\xi|^\hal &\text{if }  |\xi|\asymp \infty\end{cases},
\end{equation}
we deduce
\begin{equation}\label{est:RDN:10}
\begin{aligned}
\| \ell^{-\hal}(D)R(\eta)f\|_{H^{\sigma-\hal}} &\ls \|Q_b[v]\|_{\wL^1(I;  H^{\sigma})}+ \| \tanh(b|D|)Q_a[v]\|_{\wL^1(I; H^{\sigma})}\\
&\ls \|Q_b[v]\|_{\wL^1(I;  H^{\sigma})}+ \| Q_a[v]\|_{\wL^1(I; H^{\sigma})}. 
\end{aligned}
\end{equation}
Therefore, \eqref{est:RDN} follows from \eqref{est:RDN:10}, \eqref{est:Qab} and \eqref{estv:6}. Finally, \eqref{est:RDN1} can be proved analogously. 
\end{proof}
Next we establish the contraction estimate for $R(\eta)$. 
\begin{prop}\label{prop:DNd}
Let $\sigma\ge \sigma_0>1+\frac{d}{2}$ and consider $f\in H^\sigma$ and $\eta_j\in H^\sigma$, $j=1, 2$. Set $\eta_\delta=\eta_1-\eta_2$. There exists  a positive constant $c_2=c_2(\sigma, \sigma_0, b, d)\le c_1$  such that  if $\| \eta_j\|_{\cH^{\sigma_0}}<c_2$, $j=1, 2$,  then 
the following estimates hold.

1) For any symbol $\ell$ satisfying \eqref{symbol:ell}, we have
\begin{equation}\label{est:DNd}
\begin{aligned}
&\| \ell^{-\hal}(D)\left\{R(\eta_1)f-R(\eta_2)f\right\}\|_{H^{\sigma-\hal}(\R^d)}\\
&\quad\le C\| \eta_\delta\|_{\cH^{\sigma_0}(\R^d)}\Big(\| \nab_xf\|_{H^{\sigma-1}(\R^d)}+\| \eta_1\|_{\cH^{\sigma}(\R^d)}\| \nab_xf\|_{H^{\sigma_0-1}(\R^d)}\Big)+C\|\eta_\delta\|_{\cH^\sigma(\R^d)} \| \nab_x f\|_{H^{\sigma_0-1}(\R^d)},
\end{aligned}
\end{equation}
where $C=C(\ell, \sigma, \sigma_0, b, d)$.

2) We have 
\begin{equation}\label{est:DNd1}
\begin{aligned}
&\| |D|^{-\hal}\left\{R(\eta_1)f-R(\eta_2)f\right\}\|_{H^{\sigma-\hal}(\T^d)}\\
&\quad\le C\| \eta_\delta\|_{H^{\sigma_0}(\T^d)}\Big(\| \nab_xf\|_{H^{\sigma-1}(\T^d)}+\| \eta_1\|_{H^{\sigma}(\T^d)}\| \nab_xf\|_{H^{\sigma_0-1}(\T^d)}\Big)+C\|\eta_\delta\|_{H^\sigma(\T^d)} \| \nab_x f\|_{H^{\sigma_0-1}(\T^d)},
\end{aligned}
\end{equation}
where $C=C(\sigma, \sigma_0, b, d)$.
\end{prop}
\begin{proof}
We shall only prove \eqref{est:DNd} since the proof of \eqref{est:DNd1} is similar. Consider $\eta_j,~f: \R^d\to \R$ such that $\| \eta_j\|_{\cH^{s_0}}<c_1$.  We recall from \eqref{def:RDN} that 
\begin{equation}
R(\eta_j)f_j(x)=\int_{-b}^0\frac{\cosh((z'+b)|D|)}{\cosh(b|D|)}\left\{\di_x Q_b[v](z')-|D|\tanh(b|D|)Q_a[v](z')\right\}dz',
\end{equation}
where $\| \nab_{x,z}v_j\|_{U^{\sigma-1}}\ls \| \nab_xf\|_{H^{\sigma-1}}$ by virtue of \eqref{est:nav}. 

We shall adopt the notation $g_\delta=g_1-g_2$. Arguing as in \eqref{est:RDN:10}, we obtain 
\begin{equation}
\| \ell^{-\hal}(D)\left\{R(\eta_1)f-R(\eta_2)f\right\}\|_{H^{\sigma-\hal}} 
\ls  \| \big((Q_b[v])_\delta, (Q_a[v])_\delta\big)\|_{\wL^1(I; H^\sigma)}:=A_\sigma.
\end{equation}
Using the  product estimate \eqref{ppestcH:2}, the nonlinear estimate \eqref{nonl:cH} and the bound
\begin{equation}
 \| \big(e^{z|D|}\eta_1, e^{z|D|}\eta_2, e^{z|D|}\nab\eta_1, e^{z|D|}\nab\eta_2\big)\|_{L^\infty(I; \cH^{\sigma_0-1})}\Big)\ls \|\eta_1\|_{_{\cH^{\sigma_0}}} +\|\eta_2\|_{\cH^{\sigma_0}}\ls 1,
\end{equation}
one can prove that 
\begin{multline}
A_\sigma  \ls  \| \big((e^{z|D|}\eta)_\delta, (e^{z|D|}\nab\eta)_\delta\big)\|_{L^\infty(I; \cH^{\sigma_0-1})}\| \nab_{x,z}v_1\|_{\wL^1(I; H^\sigma)}  \\
 +\| \big((e^{z|D|}\eta)_\delta, (e^{z|D|}\nab\eta)_\delta\big)\|_{\wL^1(I; H^\sigma_\sharp)}\| \nab_{x,z}v_1\|_{L^\infty(I; L^\infty)} 
+\| \big(e^{z|D|}\eta_1, e^{z|D|}\nab\eta_1\big)\|_{L^\infty(I; \cH^{\sigma_0-1})}\| \nab_{x,z}v_\delta\|_{\wL^1(I; H^\sigma)}\\
 + \| \big(e^{z|D|}\eta_1, e^{z|D|}\nab\eta_1\big)\|_{\wL^1(I; H^\sigma_\sharp)}\| \nab_{x,z}v_\delta\|_{L^\infty(I; L^\infty)}.
\end{multline}
Combining this with the  easy inequalities
\begin{equation}
\| \big(e^{z|D|}g, e^{z|D|}\nab g\big)\|_{\wL^1(I; H_\sharp^\sigma)}\ls \| g\|_{H^\sigma_\sharp}\ls \| g\|_{\cH^\sigma} \text{ and }
\| \big(e^{z|D|}g, e^{z|D|}\nab g\big)\|_{L^\infty(I; \cH^{\sigma_0-1})}\ls \| g\|_{\cH^{\sigma_0}},
\end{equation}
we obtain
\begin{equation}
\begin{aligned}
A_\sigma&\ls  \| \eta_\delta\|_{\cH^{\sigma_0}}\| \nab_{x,z}v_1\|_{\wL^1(I; H^\sigma)}+\|\eta_\delta\|_{\cH^\sigma}\| \nab_{x,z}v_1\|_{L^\infty(I; L^\infty)}\\
&\qquad
+\| \eta_1\|_{\cH^{\sigma_0}}\| \nab_{x,z}v_\delta\|_{\wL^1(I; H^\sigma)}+\| \eta_1\|_{\cH^\sigma}\| \nab_{x,z}v_\delta\|_{L^\infty(I; L^\infty)}.
\end{aligned}
\end{equation}
Assuming that $\| \eta_j\|_{\cH^{\sigma_0}}<c_1$, we can invoke the estimate   \eqref{est:nav} to have
\begin{equation}\begin{aligned}
&\| \nab_{x, z}v_j\|_{L^\infty(I; L^\infty)}\ls \| \nab_{x, z}v_j\|_{\wL^\infty(I; H^{\sigma_0})}\ls \| \nab_x f\|_{H^{\sigma_0-1}},\\
&\| \nab_{x, z}v_j\|_{\wL^1(I; H^\sigma)}\ls \| \nab_xf\|_{H^{\sigma-1}}+\| \eta_j\|_{\cH^{\sigma}}\| \nab_xf\|_{H^{\sigma_0-1}}.
\end{aligned}\end{equation}
Consequently,
\begin{equation}\label{est:Asigma}
\begin{aligned}
A_\sigma&\ls \| \eta_\delta\|_{\cH^{\sigma_0}}(\| \nab_xf\|_{H^{\sigma-1}}+\| \eta_1\|_{\cH^{\sigma}}\| \nab_xf\|_{H^{\sigma_0-1}})+\|\eta_\delta\|_{\cH^\sigma} \| \nab_x f\|_{H^{\sigma_0-1}}\\
&\qquad +\| \eta_1\|_{\cH^{\sigma_0}}\| \nab_{x,z}v_\delta\|_{U^{\sigma-1}}+\| \eta_1\|_{\cH^{\sigma}}\| \nab_{x,z}v_\delta\|_{U^{\sigma_0-1}}.
\end{aligned}
\end{equation}
On the other hand, the proof of \eqref{estv:2} and \eqref{estv:3} yields
\begin{equation}\label{est:nabvd}
\|\nab_{x,z}v_\delta\|_{U^{\sigma-1}}\ls  \| ((Q_a[v])_\delta, (Q_b[v])_\delta)\|_{U^{\sigma-1}}.
\end{equation}
Arguing as in the proof of \eqref{est:Asigma} one can show that 
\begin{multline}\label{est:Qabd}
 \| ((Q_a[v])_\delta, (Q_b[v])_\delta)\|_{U^{\sigma-1}} 
 \ls \| \eta_\delta\|_{\cH^{\sigma_0}}(\| \nab_xf\|_{H^{\sigma-1}}+\| \eta_1\|_{\cH^{\sigma}}\| \nab_xf\|_{H^{\sigma_0-1}})\\
  +\|\eta_\delta\|_{\cH^\sigma} \| \nab_x f\|_{H^{\sigma_0-1}}+\| \eta_1\|_{\cH^{\sigma_0}}\| \nab_{x,z}v_\delta\|_{U^{\sigma-1}} 
  +\| \eta_1\|_{\cH^{\sigma}}\| \nab_{x,z}v_\delta\|_{U^{\sigma_0-1}}.
\end{multline}
It follows from \eqref{est:nabvd} and \eqref{est:Qabd} that 
\begin{equation}\label{est:vd:100}
\begin{aligned}
\|\nab_{x,z}v_\delta\|_{U^{\sigma-1}}&\ls   \| \eta_\delta\|_{\cH^{\sigma_0}}(\| \nab_xf\|_{H^{\sigma-1}}+\| \eta_1\|_{\cH^{\sigma}}\| \nab_xf\|_{H^{\sigma_0-1}})+\|\eta_\delta\|_{\cH^\sigma} \| \nab_x f\|_{H^{\sigma_0-1}}\\
&\qquad +\| \eta_1\|_{\cH^{\sigma_0}}\| \nab_{x,z}v_\delta\|_{U^{\sigma-1}}+\| \eta_1\|_{\cH^{\sigma}}\| \nab_{x,z}v_\delta\|_{U^{\sigma_0-1}}.
\end{aligned}
\end{equation}
With $\sigma=\sigma_0$, \eqref{est:vd:100} implies that if $\| \eta_j\|_{\cH^{\sigma_0}}<\wt{c_1}\le c_1$ then  
\begin{equation}
\|\nab_{x,z}v_\delta\|_{U^{\sigma_0-1}}\ls \|\eta_\delta\|_{\cH^{\sigma_0}} \| \nab_x f\|_{H^{\sigma_0-1}}.
\end{equation}
We then insert the preceding estimate into \eqref{est:vd:100} to obtain that if $\| \eta_1\|_{\cH^{\sigma_0}}<c_2\le \wt{c_1}$ then  
\begin{equation}\label{est:nabvd:f}
\|\nab_{x,z}v_\delta\|_{U^{\sigma-1}}\ls  \| \eta_\delta\|_{\cH^{\sigma_0}}(\| \nab_xf\|_{H^{\sigma-1}}+\| \eta_1\|_{\cH^{\sigma}}\| \nab_xf\|_{H^{\sigma_0-1}})+\|\eta_\delta\|_{\cH^\sigma} \| \nab_x f\|_{H^{\sigma_0-1}}.
\end{equation}
Finally, inserting \eqref{est:nabvd:f} into \eqref{est:Asigma} we arrive at \eqref{est:DNd}. 
\end{proof}

 \section{Stability of traveling wave solutions} 

In this section, we consider the Muskat problem without  external bulk force ($\mathfrak{f}=0$). In order to simplify the presentation, we shall assume that the external pressure $\varphi$ is independent of the vertical variable, i.e. $\varphi(x, y)=\varphi_0(x)$, where we adopt the notation $(x, y)$ for the horizontal and vertical components of a point in the fluid domain. More precisely, we study equation \eqref{Muskat} for the free boundary $\eta$:
 \begin{equation}\label{Muskat:1}
\p_t\eta=\gamma \p_1\eta-G(\eta)(\eta+{\varphi_0}),\quad (x, t)\in \R^d\times \R_+,\quad d=n-1\ge 1.
\end{equation}
The proofs of all the results in this section can be generalized to the more general case $\varphi(x, y)=\varphi_0(x)+\varphi_1(x, y)$ with extra regularity assumption of $\varphi_1$ as in Theorem \ref{well_posedness_flat:intro}.

\subsection{Existence of traveling wave solutions}  

The existence and uniqueness of steady solutions to \eqref{Muskat:1} have been obtained in Theorem \ref{well_posedness_flat:intro} by means of the  implicit function theorem.  In this subsection, we  shall apply the results in Section \ref{Section:DN} for the Dirichlet-Neumann operator to provide an alternative proof in this special case of the data, which serves to motivate and inform the strategy we will employ in studying the time-dependent problem \eqref{Muskat:1}.

Solutions to the steady  equation
  \begin{equation}\label{eq:steady}
\gamma \p_1\eta-G(\eta)(\eta+{\varphi_0})=0
\end{equation}
shall be constructed by a fixed point argument. To this end, we first use the expansion $G(\eta)=m(D)+R(\eta)$ to equivalently rewrite \eqref{eq:steady} as
\begin{equation}\label{eq:steady:1}
(\gamma\p_1-m(D))\eta=R(\eta)(\eta+{\varphi_0})+m(D){\varphi_0}.
\end{equation}
We note that the symbol $\gamma i\xi_1-m(\xi)$ vanishes only at $\xi=0$, and it follows from the definition \eqref{def:RDN} of $R(\eta)$ that the right-hand side of \eqref{eq:steady:1} vanishes at zero frequency. Therefore, we may seek solutions that vanish at zero frequency by solving the fixed point problem 
\begin{equation}\label{eq:steady2}
\eta=\mathcal{T}_{\varphi_0} (\eta):=(\gamma\p_1-m(D))^{-1}\left\{R(\eta)(\eta+{\varphi_0})+m(D){\varphi_0}\right\},
\end{equation}
where we adopt the convention that $\widehat{\cT_{\varphi_0}(\eta)}(0)=0$.

\begin{thm}\label{theo:existence:finite}
Let $d\ge 1$, $s>1+\frac{d}{2}$ and $\gamma\in \R\setminus\{0\}$. There exist small positive constants $r_0$ and $r_1$, both depending only on $(\gamma, s, b, d)$, such that for $\| \nab{\varphi_0}\|_{H^{s-1}(\R^d)}<r_1$, $\cT_{\varphi_0}$ is a contraction mapping on $B_{\cH^s(\R^d)}(0, r_0)$. Moreover, the mapping that maps ${\varphi_0}$ to the unique fixed point of $\cT_{\varphi_0}$ in $B_{\cH^s(\R^d)}(0, r_0)$ is  Lipschitz continuous.
\end{thm}
\begin{proof}
In the finite depth case, we have that $m(D)=|D|\tanh(b|D)$ satisfies 
\begin{equation}\label{finitesymbol}
m(\xi)\asymp
\begin{cases} 
|\xi|^2 & \text{for } |\xi|\asymp 0,\\
|\xi| & \text{for } |\xi|\asymp \infty. 
\end{cases}
\end{equation}
Consequently, for low frequencies $|\xi|<1$, we have 
\begin{equation}\label{lowfr:T}
\begin{aligned}
\int_{|\xi|<1}\omega(\xi)\left|\mathscr{F}\left\{(\gamma\p_1-m(D))^{-1}g\right\}(\xi)\right|^2 d\xi&=\int_{|\xi|<1}\frac{\xi_1^2+|\xi|^4}{|\xi|^2}\frac{1}{|\gamma|^2\xi_1^2+m^2(\xi)}|\wh{g}(\xi)|^2 d\xi\\
&\le C(\gamma, b)\int_{|\xi|<1}\frac{1}{|\xi|^2}|\wh{g}(\xi)|^2 d\xi.
\end{aligned}
\end{equation}
where $\mathscr{F}$ denotes the Fourier transform.  On the other hand, for high frequencies $|\xi|\ge 1$, we have 
\begin{equation}\label{highfr:T}
\begin{aligned}
\int_{|\xi|\ge 1}\langle \xi\rangle^{2s}\left|\mathscr{F}\left\{(\gamma\p_1-m(D))^{-1}g\right\}(\xi)\right|^2 d\xi&\le C(b) \int_{|\xi|\ge 1}\langle \xi\rangle^{2s}\frac{1}{|\xi|^2}|\wh{g}(\xi)|^2 d\xi\\
&\le  C(b) \int_{|\xi|\ge 1}\langle \xi\rangle^{2(s-\hal)}\frac{1}{|\xi|}|\wh{g}(\xi)|^2 d\xi.
\end{aligned}
\end{equation}
We note that the condition $\gamma \ne 0$ was used in the low-frequency estimate \eqref{lowfr:T} but not in the high-frequency estimate \eqref{highfr:T}.

It follows from \eqref{finitesymbol}, \eqref{lowfr:T} and \eqref{highfr:T} that
\begin{equation}\label{finitesymbolest}
\| (\gamma\p_1-m(D))^{-1}g\|_{\cH^s}\le C(\gamma, b) \| m^{-\hal}(D)g\|_{H^{s-\hal}}
\end{equation}
From \eqref{finitesymbolest}  and the definition of $\cT_{\varphi_0}(\eta)$, we deduce
\begin{equation}\label{estT:1}
\begin{aligned}
\| \cT_{\varphi_0}(\eta)\|_{\cH^s}&\le  C(\gamma, b) \| m^{-\hal}(D)\left\{R(\eta)(\eta+{\varphi_0})+m(D){\varphi_0}\right\}\|_{H^{s-\hal}}\\
&\le  C(\gamma, b)\left\{\| m^{-\hal}(D)R(\eta)(\eta+{\varphi_0})\|_{H^{s-\hal}}+\| m^\hal(D){\varphi_0}\|_{H^{s-\hal}}\right\}\\
&\le    C(\gamma, b)\left\{\| m^{-\hal}(D)R(\eta)(\eta+{\varphi_0})\|_{H^{s-\hal}}+\| \nab {\varphi_0}\|_{H^{s-1}}\right\},
\end{aligned}
\end{equation}
where  we have used the fact that the norms $\| m^\hal(D){\varphi_0}\|_{H^{s-\hal}}$ and $\| \nab {\varphi_0}\|_{H^{s-1}}$ are equivalent. 

 Suppose that $\| \eta\|_{\cH^s}<c_1$, where $c_1$ is given in Proposition \ref{prop:estvw}. Then the estimate \eqref{est:RDN} with $\sigma=\sigma_0=s$ yields 
\begin{equation}\label{estT:2}
\| m^{-\hal}(D) R(\eta)(\eta+{\varphi_0})\|_{H^{s-\hal}} \le C(s, b, d)\| \eta\|_{\cH^s}\| \nab (\eta+{\varphi_0})\|_{H^{s-1}}
\le C(s, b, d) \| \eta\|_{\cH^s}\left(\| \eta\|_{\cH^s}+\| \nab{\varphi_0} \|_{H^{s-1}}\right),
\end{equation}
where we have used the inequality $\| \nab \eta\|_{H^{s-1}}\le C\| \eta\|_{\cH^s}$. 

It follows from \eqref{estT:1} and \eqref{estT:2} that for $\| \eta\|_{\cH^s}<c_1$ we have
\begin{equation}
\| \cT_{\varphi_0}(\eta)\|_{\cH^s}\le C_1\left\{\| \eta\|_{\cH^s}\left(\| \eta\|_{\cH^s}+\| \nab{\varphi_0} \|_{H^{s-1}}\right)+\| \nab {\varphi_0}\|_{H^{s-1}}\right\},
\end{equation}
where $C_1=C_1(\gamma, s, b, d)$. If $\| \nab {\varphi_0}\|_{H^{s-1}}<\frac{1}{2C_1(2C_1+1)}$, then $\cT_{\varphi_0}$ maps the ball $B_{\cH^s}(0, r_0)\subset \cH^s$ to itself, where $r_0=\min\{c_1, \frac{1}{2C_1+1}\}$. By virtue of Proposition \ref{prop:DNd}, one can reduce the size of $\| \nab{\varphi_0}\|_{H^{s-1}}$ and the radius $r_0$ so that $\cT_{\varphi_0}$ is a contraction on $B_{\cH^s}(0, r_0)$. By the Banach contraction mapping principle, $\cT_{\varphi_0}$ has a unique fixed point  $\eta$ in $ B_{\cH^s}(0, r_0)$. The Lipschitz continuous dependence of $\eta$ on ${\varphi_0}$ again follows from Proposition \ref{prop:DNd}.
\end{proof}

We now define the space of Sobolev functions with average zero on the torus.

\begin{dfn}
For $0 \le s \in \R$ we define the space $\mathring{H}^s(\T^d)=\left\{f\in H^s(\T^d) \st  \int_{\T^d}f=0\right\}$.
\end{dfn}

We now record a variant of Theorem \ref{theo:existence:finite} for the torus case.

\begin{thm}\label{theo:existence:finite:p}
Let $d\ge 1$, $s>1+\frac{d}{2}$ and $\gamma\in \R$. There exist small positive constants $r_0$ and $r_1$, both depending only on $(s, b, d)$, such that for $\| \nab{\varphi_0}\|_{H^{s-1}(\T^d)}<r_1$, $\cT_{\varphi_0}$ is a contraction mapping on $B_{\mathring{H}^s(\T^d)}(0, r_0)$. Moreover, the mapping that maps ${\varphi_0}$ to the unique fixed point of $\cT_{\varphi_0}$ in $B_{\mathring{H}^s(\T^d)}(0, r_0)$ is  Lipschitz continuous.
\end{thm}
\begin{proof}
The proof mostly follows from  obvious modifications to the proof of Theorem \ref{theo:existence:finite} with the following caveat.   Since the zero mode of $\cT_{\varphi_0}$ vanishes in the periodic setting, there is no need for a low frequency estimate such as \eqref{lowfr:T} and only the high frequency estimate \eqref{highfr:T} is needed.  Consequently, $\gamma \ne 0$ is not needed, and the constants then do not depend on $\gamma$.

\end{proof}

\subsection{Stability of periodic traveling wave solutions}\label{sec_stability}

In this subsection, we prove that small periodic traveling wave solutions obtained in Theorem \ref{theo:existence:finite:p} are nonlinearly asymptotically stable.  The remainder of this section is devoted to the proof of Theorem \ref{theo:stability:intro}.


Suppose that for  fixed $(\gamma, {\varphi_0})$, $\eta_*$ is a steady solution of \eqref{Muskat}. We perturb $\eta_*$ by $f_0$ and  set $f(x, t)=\eta(x, t)-\eta_*(x)$, where $\eta$ is the solution of \eqref{Muskat} with initial data $\eta_0=\eta_*+f_0$.  We have
\begin{equation}\label{eq:f:0}
\p_t f=\gamma\p_1 f-\big\{G(\eta_*+f)(\eta_*+f+{\varphi_0})-G(\eta_*)(\eta_*+{\varphi_0})\big\}.
\end{equation}
Using the expansion $G(\eta)g=m(D)g+R(\eta)f$ we rewrite \eqref{eq:f:0} as
\begin{equation}\label{eq:f:1}
\begin{aligned}
\p_t f&=\gamma\p_1 f-G(\eta_*+f)f-\big\{G(\eta_*+f)(\eta_*+{\varphi_0})-G(\eta_*)(\eta_*+{\varphi_0})\big\}\\
&=\gamma\p_1 f-m(D)f-R(\eta_*+f)f\\
&\quad-\left\{ m(D)(\eta_*+{\varphi_0})+R(\eta_*+f)(\eta_*+{\varphi_0})-m(D)(\eta_*+{\varphi_0})-R(\eta_*)(\eta_*+{\varphi_0})\right\}\\
&=\gamma\p_1 f-m(D)f+\left[R(\eta_*)(\eta_*+{\varphi_0})-R(\eta_*+f)(\eta_*+{\varphi_0})\right]-R(\eta_*+f)f.
\end{aligned}
\end{equation}
The solution of \eqref{eq:f:1} with initial data $f_0$ will be sought as the fixed point of the mapping
\begin{equation}\label{def:cN}
\mathcal{N}(f):=e^{(\gamma\p_1 -m(D))t}f_0+\cL(f)(t)-\cB(f, f)(t),
\end{equation}
where
\begin{align}
&\cL(f)(t)=\int_0^t e^{(\gamma\p_1 -m(D))(t-\tau)} \left[R(\eta_*)(\eta_*+{\varphi_0})-R(\eta_*+f)(\eta_*+{\varphi_0})\right](\tau)d\tau,\\
&\cB(g, f)(t)=\int_0^t e^{(\gamma\p_1 -m(D))(t-\tau)}[R(\eta_*+g)f](\tau) d\tau.
\end{align}
To that end, we shall appeal to the following fixed point lemma.
\begin{lem}\label{lemm:fixedpoint}
Let $(E, \| \cdot\|)$ be a Banach space and let $\nu>0$. Denote by $B_\nu$ the open ball of radius $\nu$ centered at $0$ in $E$.  Assume that $\cL:B_\nu\to E$ and $\cB:B_\nu\times E\to E$ satisfy the following conditions.
\begin{itemize}
\item For all $x\in B_\nu$, $\cB(x, \cdot)$ is linear. 
\item There exists a constant $C_\cL\in (0, 1)$ such that $\| \cL(x)\|\le C_\cL \|x\|$ for all $x\in B_\nu$.
\item There exists an increasing function $\mathcal{G}_\cB$ such that $\| \cB(x, y)\|\le \mathcal{G}_\cB(\|x\|)\|y\|$ for all $x\in B_\nu$ and $y\in E$. There exists $r_*>0$ such that 
\begin{equation}\label{cd:cG}
C_\cL+\cG_\cB(r_*)<1.
\end{equation}
\item There exists an increasing function $\cF_\cL:\R_+\to \R_+$ such that 
\begin{equation}
\forall x_1, x_2\in B_\nu,~\| \cL(x_1)-\cL(x_2)\|\le  \|x_1-x_2\|\cF_\cL(\| x_1\|+\|x_2\|).
\end{equation}
\item There exists an increasing function $\cF_\cB:\R_+\to \R_+$ such that 
\begin{equation}
\forall x_1, x_2\in B_\nu,~\forall y\in E,~\| \cB(x_1, y)-\cB(x_2, y)\|\le  \|x_1-x_2\|\|y\|\cF_\cB(\| x_1\|+\|x_2\|).
\end{equation}
\end{itemize}
Assume moreover that
\begin{equation}\label{cd:cF}
\cF_\cL(2\nu)+\cG_\cB(\nu)+\nu\cF_\cB(2\nu)<\hal. 
\end{equation}
Then there exists $\delta=\delta(\nu, r_*)>0$ small enough such for all $x_0\in E$ with $\|x_0\|<\delta$, the mapping $B_\nu \ni x\mapsto \cN(x):=x_0+\cL(x)+\mathcal{B}(x, x) \in B_\nu$  has a unique fixed point $x_*$ in  $B_\nu$ with $\|x_*\|\le 2\|x_0\|$.
\end{lem}
\begin{proof}
Let $\delta<\min\{\frac{\nu}{2}, \frac{r_*}{2}\}$ and let $x_0\in E$, $\| x_0\|<\delta$.  The fixed point of $\cN$ will be obtained by the Picard iteration $x_{n+1}=\cN(x_n)$, $n\ge 1$. It can be shown using induction with the aid of \eqref{cd:cG} that $\| x_n\|<2\|x_0\|$ for all $n\ge 0$, hence $(x_n)\subset B_\nu$. 

Using the assumptions on $\cL$ and $\cB$, we obtain
\begin{equation}\label{cN:diff}
\| \cN(x)-\cN(y)\|\le \|x-y\|\left\{\cF_\cL(\|x\|+\|y\|)+\cG_\cB(\|x\|)+\|y\|\cF_\cB(\|x\|+\|y\|)\right\}\quad\forall x, y\in B_\nu.
\end{equation}
Combining \eqref{cN:diff} with \eqref{cd:cF} yields
\begin{equation}
\begin{aligned}
\| x_{n+1}-x_n\|
&\le \|x_n-x_{n-1}\|\left\{\cF_\cL(2\nu)+\cG_\cB(\nu)+\nu\cF_\cB(2\nu)\right\}\le \hal \|x_n-x_{n-1}\|.
\end{aligned}
\end{equation}
It follows that $(x_n)$ is a Cauchy sequence, hence $x_n\to x_*\in E$. In particular, we have $\|x_*\|\le 2\|x_0\|<\nu$, and thus \eqref{cN:diff} implies that $\cN(x_n)\to \cN(x_*)$. Passing to the limit in the scheme $x_{n+1}=\cN(x_n)$ yields $x_*=\cN(x_*)$. The uniqueness of $x_*$ in $B_\nu$ again follows from \eqref{cN:diff}.
\end{proof}

We now have all the tools needed to prove Theorem \ref{theo:stability:intro}.

\begin{proof}[Proof of Theorem \ref{theo:stability:intro}]

We consider $\eta_*$ and ${\varphi_0}$ in $H^s(\T^d)$ with $s>1+\frac{d}{2}$ and $\| {\varphi_0}\|_{H^s}<\frac{c_2}{3}$ and  $\| \eta_*\|_{H^s}<\frac{c_2}{3},$ where $c_2$ is the constant given in Proposition \ref{prop:DNd}. 

We note that if $\int_{\T^d}u=0$ then  $\Delta_0 u=0$ (see \eqref{Delta0=mean}). Consequently,
for $1\le q_2\le q_1\le \infty$, we have 
\begin{equation}\label{semig1:finite}
\| e^{(\gamma\p_1 -m(D))t}u\|_{\wL^{q_1}([\alpha, \beta]; H^{\mu+\frac{1}{q_1}}(\T^d))}\le C(b, d)\| u\|_{H^\mu},
\end{equation}
\begin{equation}\label{semig2:finite}
\norm{\int_a^t e^{(\gamma\p_1 -m(D))t}g }_{\wL^{q_1}([\alpha, \beta]; H^{\mu+\frac{1}{q_1}}(\T^d))}\le C(b, d)\| g\|_{\wL^{q_2}([\alpha, \beta]; H^{\mu-1+\frac{1}{q_2}}(\T^d))},
\end{equation}
provided that $u$ and $g(\cdot, t)$ have zero mean. These estimates can be proved as in Proposition \ref{prop:fourierop} with the aid of the dyadic estimate 
\begin{equation}
\| \Delta_j e^{(\gamma\p_1 -m(D))t}u\|_{L^2_x}\le C(d)e^{-c(b, d)2^jt}\| \Delta_j u\|_{L^2_x},\quad j\ge 1.
\end{equation}
Set $Y^\mu([\alpha, \beta])=\wL^\infty([\alpha, \beta]; H^{\mu}(\T^d))\cap \wL^2([\alpha, \beta]; H^{\mu+\hal}(\T^d))$ and
\begin{equation}
\| \cdot\|_{Y^\mu([\alpha, \beta])}=\| \cdot\|_{\wL^\infty([\alpha, \beta]; H^{\mu}(\T^d))}+\| \cdot\|_{\wL^2([\alpha, \beta]; H^{\mu+\hal}(\T^d))}.
\end{equation}
Consider $f_0\in \mathring{H}^s(\T^d)$ and let $T>0$ be arbitrary.  By Proposition \ref{prop:estvw} 2),  $R$ always has zero mean, and so  do $\cL$ and $\cB$. Therefore, using \eqref{semig1:finite}, \eqref{semig2:finite} and the fact that 
\begin{equation}
\| \cdot\|_{\wL^2([\alpha, \beta]; H^{\mu}(\T^d))}=\| \cdot\|_{L^2([\alpha, \beta]; H^{\mu}(\T^d))},
\end{equation}
 we obtain
 \begin{align}\label{est:linsemig}
& \|e^{(\gamma\p_1 -m(D))t}f_0\|_{Y^s([0, T])}\le C(b, d)\| f_0\|_{H^s},\\ \label{est:Lf:1}
&\| \cL(f)\|_{Y^s([0, T])}\le C(b, d)\| R(\eta_*)(\eta_*+{\varphi_0})-R(\eta_*+f)(\eta_*+{\varphi_0})\|_{L^2([0, T]; H^{s-\hal})},\\ \label{est:Bgf:1}
&\| \cB(g, f)\|_{Y^s([0, T])}\le C(b, d)\| R(g)f\|_{L^2([0, T]; H^{s-\hal})}.
\end{align}
 We want to apply Lemma \ref{lemm:fixedpoint} with 
 \begin{equation}
 E_T=\left\{f\in Y^s([0, T]) \st \int_{\T^d}f(x, t)dx=0~\text{a.e.~}t\in [0, T]\right\}.
 \end{equation}
 Let $B_\nu$ be the open ball in $E_T$ with center $0$ and radius $\nu<\frac{2c_2}{3}$. 
Let $f\in B_\nu$. We have $\| f\|_{L^\infty([0, T]; H^s)}\le \| f\|_{\wL^\infty([0, T]; H^s)}$, hence $\|f(t)\|_{H^s}<\nu<c_2/3$ a.e. $t\in [0, T]$.  Consequently, $\| \eta_*+f(t)\|_{H^s}<c_2$ a.e. $t\in [0, T]$ and thus we can apply the contraction estimate \eqref{est:DNd} with $\eta_1=\eta_*$, $\eta_2=\eta_*+f$, $\sigma=s+\hal$ and $\sigma_0=s$,
\begin{multline}\label{est:nonl2}
\| R(\eta_*)(\eta_*+{\varphi_0})-R(\eta_*+f)(\eta_*+{\varphi_0})\|_{H^{s-\hal}} \\
\le C\| f\|_{\cH^s}\left\{\| \nab_x(\eta_*+{\varphi_0})\|_{H^{s-\hal}}+\| \eta_*\|_{\cH^{s+\hal}}\| \nab_x(\eta_*+{\varphi_0})\|_{H^{s-1}}\right\}
+C\|f\|_{\cH^{s+\hal}} \| \nab_x (\eta_*+{\varphi_0})\|_{H^{s-1}}\\
\le C\| f\|_{H^{s+\hal}}\left\{\| \nab_x(\eta_*+{\varphi_0})\|_{H^{s-\hal}}+\| \eta_*\|_{H^{s+\hal}}\| \nab_x(\eta_*+{\varphi_0})\|_{H^{s-1}}\right\}
\end{multline}
a.e. $t\in [0, T]$,  where $C=C(s, b, d)$ and we have used the embedding $H^\mu\subset \cH^\mu$ for $\mu\ge 0$.  Combining \eqref{est:Lf:1} and \eqref{est:nonl2} yields  
\begin{equation}\label{est:cL}
\| \cL(f)\|_{Y^s([0, T])}\le C\| f\|_{L^2([0, T]; H^{s+\hal})}\left\{\| \nab_x(\eta_*+{\varphi_0})\|_{H^{s-\hal}}+\| \eta_*\|_{\cH^{s+\hal}}\| \nab_x(\eta_*+{\varphi_0})\|_{H^{s-1}}\right\},
\end{equation}
where $C=C(s, b, d)$. 

Next, for $g\in B_\nu$, we can apply the remainder estimate \eqref{est:RDN} with $\sigma=s+\hal$ and $\sigma_0=s$ to have
\begin{equation}\label{est:nonl1}
\begin{aligned}
\| R(\eta_*+g)f\|_{H^{s-\hal}}&\le C\| \eta_*+g\|_{\cH^s}\| \nab_xf\|_{H^{s-\hal}}+ C\| \eta_*+g\|_{\cH^{s+\hal}}\| \nab_xf\|_{H^{s-1}}\\
&\le C\|  \eta_*\|_{H^{s+\hal}}\| \nab_xf\|_{H^{s-\hal}}+ C\|  g\|_{H^s}\| \nab_xf\|_{H^{s-\hal}}+C \| g\|_{H^{s+\hal}}\| \nab_xf\|_{H^{s-1}}
\end{aligned}
\end{equation}
a.e. $t\in [0, T]$, where $C=C(s, b, d)$.  It follows from \eqref{est:Bgf:1} and \eqref{est:nonl1} that 
\begin{multline}\label{est:cB}
\| \cB(g, f)\|_{Y^s([0, T])} \le C\|  \eta_*\|_{H^{s+\hal}}\| \nab_xf\|_{L^2([0, T]; H^{s-\hal})}+C\| g\|_{L^\infty([0, T]; H^s)}\| \nab_xf\|_{L^2([0, T]; H^{s-\hal})}\\
+ C\|g\|_{L^2([0, T]; H^{s+\hal})}\| \nab_xf\|_{L^\infty([0, T]; H^{s-1})} 
\le C\big(\|  \eta_*\|_{H^{s+\hal}}+\| g\|_{Y^s([0, T])}\big)\|f\|_{Y^s([0, T])},
\end{multline}
where $C=C(s, b, d)$.  By a completely analogous argument, we obtain that 
\begin{equation}\label{est:cLd}
\begin{aligned}
&\| \cL(f_1)-\cL(f_2)\|_{Y^s([0, T])}\\
&\le C\| R(\eta_*+f_1)(\eta_*+{\varphi_0})-R(\eta_*+f_2)(\eta_*+{\varphi_0})\|_{L^2([0, T]; H^{s-\hal})}\\
&\le  C\| f_1-f_2\|_{L^2([0, T]; H^s)}\left\{\| \nab_x(\eta_*+{\varphi_0})\|_{H^{s-\hal}}+\| \eta_*\|_{H^{s+\hal}}\| \nab_x(\eta_*+{\varphi_0})\|_{H^{s-1}}\right\}\\
&\quad+C\| f_1-f_2\|_{L^\infty([0, T]; H^s)}\| f_1\|_{L^2([0, T]; H^{s+\hal})}\| \nab_x(\eta_*+{\varphi_0})\|_{H^{s-1}}\\
&\quad+C\| f_1-f_2\|_{L^2([0, T]; H^{s+\hal})}\| \nab_x(\eta_*+{\varphi_0})\|_{H^{s-1}}\\
&\le C\|f_1-f_2\|_{Y^s([0, T])}\left\{\| \nab_x(\eta_*+{\varphi_0})\|_{H^{s-\hal}}+\| \eta_*\|_{H^{s+\hal}}\| \nab_x(\eta_*+{\varphi_0})\|_{H^{s-1}}\right.\\
&\qquad\qquad\left.+\| f_1\|_{Y^s([0, T])}\| \nab_x(\eta_*+{\varphi_0})\|_{H^{s-1}}\right\}
\end{aligned}
\end{equation}
for $f_1$,  $f_2\in B_\nu$, and 
\begin{equation}\label{est:cBd}
\begin{aligned}
&\| \cB(g_1, f)-\cB(g_2, f)\|_{Y^s([0, T])}\\
&\le C\| R(\eta_*+g_1)f-R(\eta_*+g_2)f\|_{L^2([0, T]; H^{s-\hal})}\\
&\le  C\| g_1-g_2\|_{L^\infty([0, T]; H^s)}\left\{\| \nab_xf\|_{L^2([0, T]; H^{s-\hal})}+\|g_1\|_{L^2([0, T]; H^{s+\hal})}\| \nab_xf\|_{L^\infty([0, T]; H^{s-1})}\right\}\\
&\quad+C\| g_1-g_2\|_{L^2([0, T]; H^s)}\|\eta_*\|_{H^{s+\hal}}\| \nab_xf\|_{L^\infty([0, T]; H^{s-1})}\\
&\quad+C\| g_1-g_2\|_{L^2([0, T]; H^{s+\hal})}\| \nab_xf\|_{L^\infty([0, T]; H^{s-1})}\\
&\le C\| g_1-g_2\|_{Y^s([0, T])}\|f\|_{Y^s([0, T])}(1+\|\eta_*\|_{H^{s+\hal}}+\| g_1\|_{Y^s([0, T])})
\end{aligned}
\end{equation}
for $g_1$, $g_2\in B_\nu$. In both  \eqref{est:cLd} and \eqref{est:cBd}, $C=C(s, b, d)$.

In view of \eqref{est:cL}, \eqref{est:cB}, \eqref{est:cLd} and \eqref{est:cBd}, we find that the conditions in Lemma \ref{lemm:fixedpoint} are satisfied with 
\begin{equation}
\begin{aligned}
C_\cL &= C\left\{\| \nab_x(\eta_*+{\varphi_0})\|_{H^{s-\hal}}+\| \eta_*\|_{H^{s+\hal}}\| \nab_x(\eta_*+{\varphi_0})\|_{H^{s-1}}\right\},\\
\cG_\cB(z) &= C(\|  \eta_*\|_{H^{s+\hal}}+z),\\
\cF_\cL(z) &= C\left\{\| \nab_x(\eta_*+{\varphi_0})\|_{H^{s-\hal}}+\| \eta_*\|_{H^{s+\hal}}\| \nab_x(\eta_*+{\varphi_0})\|_{H^{s-1}}+z\| \nab_x(\eta_*+{\varphi_0})\|_{H^{s-1}}\right\},\\
\cF_\cB(z) &=C(1+\|\eta_*\|_{H^{s+\hal}}+z),
\end{aligned}
\end{equation}
where $C=C(s, b, d)$. According to Theorem \ref{theo:existence:finite:p}, if $\|\nab{\varphi_0}\|_{H^{s-\hal}}$ is small enough then $\| \eta_*\|_{H^{s+\hal}}\le C(s, b, d)\| \nab {\varphi_0}\|_{H^{s-\hal}}$. Therefore, for sufficiently  small  $\|\nab{\varphi_0}\|_{H^{s-\hal}}$, we have $C_\cL<1$ and the conditions \eqref{cd:cG} and \eqref{cd:cF} hold for sufficiently small $r_*$ and $\nu$. Therefore, by virtue of Lemma \ref{lemm:fixedpoint} and with the aid of \eqref{est:linsemig}, there exist $\delta>0$ and $C(b, d)>0$ such that if $\|f_0\|_{H^s}<\delta$ then $\cN$ has a unique fixed point $f$ in $B_\nu$, and 
\[
\|f\|_{Y^s([0, T])}\le C(b, d)\| f_0\|_{H^s}\quad\forall T>0.
\] 
Since the smallness of ${\varphi_0}$, $\nu$  and $\delta$ is independent of $T$, $f$ is a global solution.  

Next, we prove that $f$ decay exponentially in $H^s$. Since $f\in Y^s([0, T])$ for all $T>0$, using \eqref{est:nonl2} and  \eqref{est:nonl1} we deduce that $\p_t f\in  L^2([0, T]; H^{s-\hal})$ for all $T>0$. Then applying Theorem 3.1 in \cite{LionMage} yields $f\in C([0, T]; H^s)$ for all $T>0$.  Appealing to \eqref{est:nonl2} and  \eqref{est:nonl1} again, we deduce
\begin{equation}
\begin{aligned}
\hal\frac{d}{dt}\| f(t)\|_{H^s}^2&=\langle \p_t f, f\rangle_{H^{s-\hal}, H^{s+\hal}}\\
&=\langle \gamma \p_1f, f\rangle_{H^{s-\hal}, H^{s+\hal}}-\langle m(D)f, f\rangle_{H^{s-\hal}, H^{s+\hal}}\\
&\quad -\langle R(\eta_*+f)f, f\rangle_{H^{s-\hal}, H^{s+\hal}}+\langle R(\eta_*)(\eta_*+{\varphi_0})-R(\eta_*+f)(\eta_*+{\varphi_0}), f\rangle_{H^{s-\hal}, H^{s+\hal}}\\
&\le -c_0\| f\|_{H^{s+\hal}}^2+C\| f\|^2_{H^{s+\hal}}\left\{\| \nab_x(\eta_*+{\varphi_0})\|_{H^{s-\hal}}+\| \eta_*\|_{H^{s+\hal}}\| \nab_x(\eta_*+{\varphi_0})\|_{H^{s-1}}\right\}\\
&\quad +C\left\{\|\eta_*\|_{H^{s+\hal}}\|f\|_{H^{s+\hal}}+\| f\|_{H^s}\| f\|_{H^{s+\hal}}\right\}\| f\|_{H^{s+\hal}},
\end{aligned}
\end{equation}
where we have used the facts that  $f(\cdot, t)$ has zero mean to have
\begin{equation}
\langle m(D)f, f\rangle_{H^{s-\hal}, H^{s+\hal}}\ge c_0\| f\|_{H^{s+\hal}}^2,\quad c_0=c_0(s, b, d).
\end{equation}
It follows that 
\begin{equation}\label{dtf:Hs}
\hal\frac{d}{dt}\| f(t)\|_{H^s}^2\le -c_0\| f\|_{H^{s+\hal}}^2+C\| f\|^2_{H^{s+\hal}}\left\{\mu+\| f\|_{H^s}\right\},
\end{equation}
where
\begin{equation}
\mu=\| \nab_x(\eta_*+{\varphi_0})\|_{H^{s-\hal}}+\| \eta_*\|_{H^{s+\hal}}\| \nab_x(\eta_*+{\varphi_0})\|_{H^{s-1}}+\|\eta_*\|_{H^{s+\hal}}.
\end{equation}
We choose $\|\nab_x {\varphi_0}\|_{H^{s-\hal}}$ small enough so that $\mu<\frac{c_0}{3C}$, and assume that $\| f_0\|_{H^s}<\min\{\delta, \frac{c_0}{3C}\}$. Using the continuity of $t\mapsto \| f(t)\|_{H^s}$, we deduce from \eqref{dtf:Hs} that  $\|f(t)\|_{H^s}<\frac{c_0}{3C}$ for all $t>0$, and thus 
\begin{equation}
\frac{d}{dt}\| f(t)\|_{H^s}^2\le -\frac{c_0}{3}\| f(t)\|_{H^{s+\hal}}^2\le -\frac{c_0}{3}\| f(t)\|_{H^s}^2.
\end{equation}
Therefore, we obtain the exponential decay
\begin{equation}
\| f(t)\|_{H^s}\le \|f_0\|_{H^s}e^{-\frac{c_0}{6}t}\quad\forall t>0
\end{equation}
and the global dissipation bound 
\begin{equation}
\int_0^\infty \| f(t)\|_{H^{s+\hal}}^2dt\le \frac{3}{c_0}\| f_0\|^2_{H^s}. 
\end{equation}
This completes the proof of Theorem \ref{theo:stability:intro}.
 
\end{proof}

\appendix

\section{Analytic tools} 

This appendix collects a number of analysis tools used throughout the paper.
 
\subsection{Specialized scales of anisotropic Sobolev spaces}\label{appendixA1}

In this subsection we recall the properties of a scale of anisotropic Sobolev spaces introduced in \cite{LeoniTice}.

\begin{dfn}\label{specialized_dfn}
Let $0 \le s \in \R$.
\begin{enumerate}
 \item We define the anisotropic Sobolev-type space 
\begin{equation}
 \H^s(\R^d) = \{ f \in \mathscr{S}'(\R^d) \st f = \bar{f}, \hat{f} \in L^1_{loc}(\R^d), \text{ and } \norm{f}_{\H^s} < \infty \},
\end{equation}
where the square-norm is defined by
\begin{equation}
 \norm{f}_{\H^s}^2 = \int_{B(0,1)} \frac{\xi_1^2 + \abs{\xi}^4 }{\abs{\xi}^2} \abs{\hat{f}(\xi)}^2 d\xi + \int_{B(0,1)^c} \br{\xi}^{2s} \abs{\hat{f}(\xi)}^2 d\xi.
\end{equation}
We endow the space $\H^s(\R^d)$ with the obvious associated inner-product.  We write $\H^s(\Sigma) = \H^s(\R^{n-1})$ with the usual identification of $\Sigma \simeq \R^{n-1}$. 

 \item We define 
\begin{equation}
 \H^s(\T^d) = \mathring{H}^s(\T^d) = \{ f \in H^s(\T^d) \st \int_{\T^d} f = 0 \}
\end{equation}
with the usual norm.

 \item We write $\H^s(\Sigma) = \H^s(\Gamma)$ via the natural identification of $\Sigma = \Gamma \times \{0\}$ with $\Gamma \in \{\R^{n-1},\T^{n-1}\}$.
\end{enumerate}
\end{dfn}

The following result summarizes the fundamental properties of this space.

\begin{thm}\label{specialized_properties}
Let $0 \le s \in \R$.  Then the following hold.
\begin{enumerate}
 \item $\H^s(\R^d)$ is a Hilbert space, and the set of real-valued Schwartz functions is dense in $\H^s(\R^d)$.

 \item $H^s(\R^d) \hookrightarrow \H^s(\R^d)$, and this inclusion is strict for $d\ge 2$.  If $d =1$, then $H^s(\R) = \H^s(\R)$.
 
 \item If $t \in \R$ and $s < t$, then we have the continuous inclusion $\H^t(\R^d) \hookrightarrow \H^s(\R^d)$.
 
 \item  For $R \in \R_+$ and $f \in \H^s(\R^d)$ define the low-frequency localization $f_{l,R} = (\hat{f} \vchi_{B(0,R)})^{\vee}$ and the high-frequency localization $f_{h,R} = (\hat{f} \vchi_{B(0,R)^c})^\vee$.  Then $f_{l,R} \in  \bigcap_{t \ge 0} \H^t(\R^d) \cap \bigcap_{k\in \N} C^k_0(\R^d)$, and $f_{h,R} \in \H^s(\R^d) \cap H^s(\R^d)$.  Moreover, we have the estimates 
\begin{equation}
 \norm{f_{l,R}}_{\H^s} \le \norm{f}_{\H^s} \text{ and } \norm{f_{h,R}}_{\H^s} \le \norm{f}_{\H^s}
\end{equation}
as well as 
\begin{equation}
 \norm{f_{l,R}}_{C^k_0} \ls_k \norm{f}_{\H^s},  \norm{f_{l,R}}_{\H^t} \ls \norm{f}_{\H^s}, \text{ and } \norm{f_{h,R}}_{H^s} \ls \norm{f}_{\H^s}.
\end{equation}

 \item For each $k \in \N$ we have the continuous inclusion $\H^s(\R^d) \hookrightarrow C^k_0(\R^d) + H^s(\R^d)$.

 \item If $s > d/2$ then $\hat{f}\in L^1(\R^d;\C)$, and 
\begin{equation}\label{L1cH}
 \norm{\hat{f}}_{L^1} \ls \norm{f}_{\H^s}.
\end{equation}

 \item If $k \in \N$ and $s >k+ d/2$, then we have the continuous inclusion  $\H^s(\R^d) \hookrightarrow C^k_0(\R^d)$.


 \item If $s \ge 1$, then  $\norm{\nab f}_{H^{s-1}} \ls \norm{f}_{\H^s}$ for each $f \in \H^s(\R^d)$. In particular, we have that $\nab : \H^s(\R^d) \to H^{s-1}(\R^d;\R^d)$ is a bounded linear map, and this map is injective.

 \item If $s \ge 1$, then $\snorm{\p_1 f}_{\dot{H}^{-1}} \ls \norm{f}_{\H^s}$ for $f \in \H^s(\R^d)$. In particular, we have that $\p_1 : \H^s(\R^d) \to H^{s-1}(\R^d)\cap \dot{H}^{-1}(\R^d)$ is a bounded linear map, and this map is injective.
\end{enumerate}
\end{thm}
\begin{proof}
 These are proved in  Proposition 5.2 and Theorems 5.5 and 5.6 of \cite{LeoniTice}.

\end{proof}

Next we recall another space introduced in \cite{LeoniTice} that will be useful in working with the Poisson extension of functions in $\H^s(\Sigma)$.

\begin{dfn}\label{P_aniso_def}
Let $0 \le s \in \R$ and $n \ge 2$. 
\begin{enumerate}
 \item When $\Gamma = \R^{n-1}$ we define the space 
\begin{multline}
 \P^s(\Omega) = H^s(\Omega) + \H^s(\Sigma) = \{f \in L^1_{\text{loc}}(\Omega) \st \text{there exist } g \in H^s(\Omega) \text{ and } h \in \H^s(\Sigma) \\
 \text{ such that } f(x) = g(x) + h(x') \text{ for almost every }x \in \Omega \}. 
\end{multline}
We endow $\P^s(\Omega)$ with the norm 
\begin{equation}
 \norm{f}_{\P^s} = \inf\{ \norm{g}_{H^s} + \norm{h}_{\H^s} \st f = g +h \text{ for } g \in H^s(\Omega), h\in \H^s(\R^{n-1})  \}.
\end{equation}
 \item When $\Gamma = \T^{n-1}$ we define the space $\P^s(\Omega) = H^s(\Omega)$ and endow it with the usual $H^s(\Omega)$ norm.
\end{enumerate}

\end{dfn}

The key properties of this space are recorded in the following.

\begin{thm}\label{Ps_properties}
 Let $0 \le s \in \R$ and $n \ge 2$.  Then the following hold.
\begin{enumerate}
\item If $t \in \R$ and $s < t$, then we have the continuous inclusion $\P^t(\Omega) \subset \P^s(\Omega)$.

\item For each $f \in \H^s(\Sigma)$ we have that $\norm{f}_{\P^s} \le \norm{f}_{\H^s}$, and hence we have the continuous inclusion $\H^s(\Sigma) \subset \P^s(\Omega)$.

\item If $k \in \N$ and $s >k+ n/2$, then $\norm{f}_{C^k_b} \ls  \norm{f}_{\P^s}$ for all $f \in \P^s(\Omega)$. Moreover, we have the continuous inclusion  
\begin{equation}
\P^s(\Omega) \subseteq \{f \in C^k_b(\Omega) \st \lim_{\abs{x'} \to \infty} \p^\alpha f(x) = 0 \text{ for } \abs{\alpha} \le k\} \subset     C^k_b(\Omega).
\end{equation}

 \item If $s \ge 1$, then $\norm{\nab f}_{H^{s-1}} \ls \norm{f}_{\P^s}$ for each $f \in \P^s(\Omega)$.
In particular, we have that $\nab : \P^s(\Omega) \to H^{s-1}(\Omega;\R^n)$ is a bounded linear map.

 \item If $s > 1/2$, then the trace map $\tr: H^s(\Omega) \to H^{s-1/2}(\Sigma)$ extends to a bounded linear map $\tr: \P^s(\Omega) \to \H^{s-1/2}(\Sigma)$.  More precisely, the following hold.
\begin{enumerate}
 \item If $f \in C^0(\bar{\Omega}) \cap \P^s(\Omega)$, then $\tr f = f \vert_{\Sigma}$.
 \item If $\varphi \in C_c^1(\R^{n-1} \times (-b,0])$, then 
\begin{equation}
 \int_{\Sigma} \tr f \varphi = \int_{\Omega} \p_n f \varphi + f \p_n \varphi \text{ for all } f \in \P^s(\Omega).
\end{equation}
 \item We have the bound $\norm{\tr f}_{\H^{s-1/2}} \ls \norm{f}_{\P^s}$  for all $f\in \P^s(\Omega)$.
\end{enumerate}

\end{enumerate} 
 
\end{thm}
\begin{proof}
In the case $\Gamma = \R^{n-1}$ these are proved in Theorems 5.7, 5.9, and 5.11 of \cite{LeoniTice}.  In the case $\Gamma= \T^{n-1}$ they follow from standard properties of Sobolev spaces.
\end{proof}

Next we record a crucial fact about the $\P^s$ spaces: they give rise to standard $H^s$ multipliers.

\begin{thm}\label{Ps_product_supercrit}
Let $n \ge 2$ and $s > n/2$.  Then for $0 \le r \le s$ 
\begin{equation}\label{omega_product_supercrit_00}
 \norm{fg}_{H^r} \ls \norm{f}_{\P^s} \norm{g}_{H^r} \text{ for all }f\in \P^s(\Omega) \text{ and }g \in H^r(\Omega).
\end{equation}
In particular, for $1 \le k \in \N$ the mapping 
\begin{equation}\label{omega_product_supercrit_01}
H^r(\Omega) \times \prod_{j=1}^k \P^s(\Omega) \ni (g,f_1,\dotsc,f_k) \mapsto g \prod_{j=1}^k f_j \in H^r(\Omega) 
\end{equation}
is a bounded $(k+1)-$linear map.    
\end{thm}
\begin{proof}
When $\Gamma = \T^{n-1}$ this follows from the standard properties of Sobolev spaces.  Assume then that $\Gamma = \R^{n-1}$.  The bound \eqref{omega_product_supercrit_00} is proved for $r=s$ in Theorem 5.14 in \cite{LeoniTice}.  When $r=0$, the bound \eqref{omega_product_supercrit_00} follows immediately from the third item of Theorem \ref{Ps_properties}.  The general case $0 < r < s$ then follows from these endpoint bounds and interpolation (see, for instance, \cite{bergh_lof,triebel}).  
\end{proof}

\subsection{Poisson extension} \label{appendix:poisson_ext}

We now study the Poisson extension operator, first on standard Sobolev spaces.

\begin{prop}\label{poisson_sobolev_map}
Let $-1/2 \le s \in \R$.  For $\eta \in H^s(\Sigma)$ define $\pf \eta : \Omega \to \R$ via 
\begin{equation}
 \pf \eta(x) = \frac{1}{(2\pi)^{n-1}} \int_{\hat{\Gamma}} e^{i x'\cdot \xi} e^{\abs{\xi} x_n} \hat{\eta}(\xi) d\xi.
\end{equation}
Then $\pf \eta \in H^{s+1/2}(\Omega)$ and $\norm{\pf \eta}_{H^{s+1/2}} \ls \norm{\eta}_{H^{s}}$.  In particular, $\pf : H^{s}(\Sigma) \to H^{s+1/2}(\Omega)$ defines a bounded linear map.
\end{prop}
\begin{proof}
We'll only present the proof in the case $\Gamma = \R^{n-1}$, and the case $\Gamma = \T^{n-1}$ is similar and simpler.  First note that 
\begin{equation}\label{poisson_sobolev_map_1}
\int_{-b}^0 \abs{e^{\abs{\xi} x_n}}^2 dx_n = \frac{1}{2 \abs{\xi}}(1 - e^{-2 \abs{\xi}b})  \asymp 
\begin{cases}
b & \text{for } \abs{\xi} \asymp 0 \\
\frac{1}{2 \abs{\xi}} &\text{for } \abs{\xi} \asymp \infty.
\end{cases}
\end{equation}

Suppose initially that $m \in \N$ and that $\eta \in H^{m-1/2}(\Sigma)$. Using \eqref{poisson_sobolev_map_1}, we may readily bound 
\begin{equation}
 \norm{\pf \eta}_{H^m(\Omega)}^2 \ls \int_{\R^{n-1}}  \br{\xi}^{2m} \abs{\hat{\eta}(\xi)}^2  \int_{-b}^0\abs{e^{\abs{\xi} x_n}}^2 dx_n d\xi \ls \int_{\R^{n-1}} \br{\xi}^{2m-1} \abs{\hat{\eta}(\xi)}^2 d\xi = \norm{\eta}_{H^{m-1/2}}^2.
\end{equation}
Thus, $\pf : H^{m-1/2}(\Sigma) \to H^m(\Omega)$ defines a bounded linear operator for every $m \in \N$.  Standard interpolation theory (see, for instance, \cite{bergh_lof,triebel}) then shows that $\pf : H^{t-1/2}(\Sigma) \to H^t(\Omega)$ defines a bounded linear operator for every $0 \le t \in \R$, which is the desired result upon setting $t-1/2 = s$.
\end{proof}

Next we consider the Poisson extension on the anisotropic spaces $\H^s(\Sigma)$, which requires the use of the $\P$ spaces from Definition \ref{P_aniso_def}.

\begin{thm}\label{poisson_aniso_map}
Let $0 \le s \in \R$ and $\Gamma = \R^{n-1}$.  For $\eta \in \H^s(\Sigma)$ define $\pf \eta : \Omega \to \R$ via 
\begin{equation}
 \pf \eta(x) = \frac{1}{(2\pi)^{n-1}}  \int_{\R^{n-1}} e^{i x'\cdot \xi} e^{\abs{\xi} x_n} \hat{\eta}(\xi) d\xi.
\end{equation}
Then the following hold.
\begin{enumerate}
 \item $\pf \eta - \eta_l \in H^{s+1/2}(\Omega)$ and $\norm{\pf \eta - \eta_l}_{H^{s+1/2}} \ls \norm{\eta}_{\H^s}$, where $\eta_l = \eta_{l,1} \in \H^{s+1/2}(\Sigma) \cap \bigcap_{k \in \N} C^k_0(\Sigma)$ in the notation of Theorem \ref{specialized_properties}.
 \item $\pf \eta \in \P^{s+1/2}(\Omega)$ and $\norm{\pf \eta}_{\P^{s+1/2}} \ls \norm{\eta}_{\H^{s}}$.
 \item The induced map $\pf : \H^{s}(\Sigma) \to \P^{s+1/2}(\Omega)$ is bounded and linear.
\end{enumerate}

\end{thm}
\begin{proof}
We split $\eta$ into its high and low frequency parts: $\eta = \eta_h + \eta_l$, where $\hat{\eta}_h = \vchi_{B(0,1)^c} \hat{\eta}$ and $\hat{\eta}_{l} = \vchi_{B(0,1)} \hat{\eta}$.  Then we know from Theorem \ref{specialized_properties} that $\eta_l,\eta_h \in \H^s(\Sigma)$ with $\norm{\eta_{j}}_{\H^s} \le \norm{\eta}_{\H^s}$ for $j \in \{l,h\}$.  We also know that $\eta_h \in H^s(\Sigma)$ with $\norm{\eta_h}_{H^s} \ls \norm{\eta}_{\H^s}$. Consequently, Proposition \ref{poisson_sobolev_map} shows that $\pf \eta_h \in H^{s+1/2}(\Omega)$ with 
\begin{equation}\label{poisson_aniso_map_1}
 \norm{\pf \eta_h}_{H^{s+1/2}} \ls \norm{\eta_h}_{H^s} \ls \norm{\eta}_{\H^s}.
\end{equation}

Now consider the function $\pf \eta_l - \eta_l : \Omega \to \R$, which satisfies
\begin{equation}
 \pf \eta_l(x) - \eta_l(x') = \frac{1}{(2\pi)^{n-1}}  \int_{B(0,1)} e^{i x'\cdot \xi} (e^{ \abs{\xi} b}-1) \hat{\eta}(\xi) d\xi.
\end{equation}
We calculate
\begin{equation}
 \int_{-b}^0 \abs{e^{\abs{\xi} x_n}-1}^2 dx_n =  \frac{1}{2 \abs{\xi}}  (-3+ 2 \abs{\xi} b +4 e^{- \abs{\xi} b}  - e^{-2 \abs{\xi} b}) 
 \asymp 
\begin{cases}
\frac{  \abs{\xi}^2 b^3}{3} & \text{for } \abs{\xi} \asymp 0 \\
b &\text{for } \abs{\xi} \asymp \infty.
\end{cases} 
\end{equation}
For any $m \in \N$ this allows us to bound 
\begin{multline}
 \norm{\pf \eta_l - \eta_l}_{H^m}^2 \ls   \int_{B(0,1)}  \br{\xi}^{2m} \abs{\hat{\eta}(\xi)}^2  \int_{-b}^0\abs{e^{ \abs{\xi} x_n}-1}^2 dx_n d\xi \ls \int_{B(0,1)} \br{\xi}^{2m-1} \abs{\xi}^{2}\abs{\hat{\eta}(\xi)}^2 d\xi  \\
 \ls \int_{B(0,1)}   \abs{\xi}^{2}\abs{\hat{\eta}(\xi)}^2 d\xi  \ls \norm{\eta}_{\H^s}^2. 
\end{multline}
Thus, $\pf \eta_l - \eta_l \in \bigcap_{m \in \N} H^m(\Omega)$, but in particular we can choose a fixed $s+1/2 \le m \in \N$ to see that $\pf \eta_l - \eta_l \in H^{s+1/2}(\Omega)$ with
\begin{equation}\label{poisson_aniso_map_2}
 \norm{\pf \eta_l - \eta_l}_{H^{s+1/2}} \ls  \norm{\pf \eta_l - \eta_l}_{H^{m}} \ls \norm{\eta}_{\H^s}.
\end{equation}

Finally, note that $\eta_l \in \bigcap_{t > 0} \H^t(\Sigma)$ and that 
\begin{equation}\label{poisson_aniso_map_3}
 \norm{\eta_l}_{\h^t} = \norm{\eta_l}_{\h^s} \text{ for all } 0 \le t \in \R.
\end{equation}
In particular, $\eta_l \in \H^{s+1/2}(\Sigma)$ with $\norm{\eta_l}_{\H^{s+1/2}} \le \norm{\eta}_{\H^s}$.  We may thus combine \eqref{poisson_aniso_map_1}, \eqref{poisson_aniso_map_2}, and \eqref{poisson_aniso_map_3} to see that $\pf \eta = [\pf \eta_h + (\pf \eta_l - \eta_l)] + \eta_l$ with 
\begin{equation}
 \norm{\pf \eta_h + (\pf \eta_l - \eta_l)}_{H^{s+1/2}} + \norm{\eta_l}_{\H^{s+1/2}} \ls \norm{\eta}_{\H^s}.
\end{equation}
Hence, $\pf \eta \in \P^{s+1/2}(\Omega)$ with $\norm{\pf \eta}_{\P^{s+1/2}} \ls \norm{\eta}_{\H^s},$ which is the desired bound.
\end{proof}

Finally, we record some results about the normal derivative of the Poisson extension.

\begin{prop}\label{normal_poisson_neg_est}
Let $0 \le s \in \R$ and $\eta \in \H^{s+3/2}(\Sigma)$.  Then the following hold.
\begin{enumerate}
 \item If $\Gamma = \R^{n-1}$, then  $\p_n \pf\eta(\cdot,0) - \p_n \pf \eta(\cdot,-b) \in \dot{H}^{-1}(\R^{n-1})$ and 
\begin{equation}
 \snorm{\p_n \pf\eta(\cdot,0) - \p_n \pf \eta(\cdot,-b)}_{\dot{H}^{-1}} \le  b \norm{\nab \pf \eta}_{L^2} \ls  \norm{\eta}_{\H^{s+3/2}}.
\end{equation}
 \item If $\Gamma = \T^{n-1}$, then  $\widehat{\p_n \pf \eta \vert_{\Sigma}}(0)  = \widehat{\p_n \pf \eta \vert_{\Sigma_{-b}}}(0)=0$.
\end{enumerate}
\end{prop}
\begin{proof}
We'll only prove the first item, as the second is simpler and similar.  Theorem \ref{poisson_aniso_map} tells us that $\pf \eta \in \P^{s+2}(\Omega)$, and so Theorem \ref{Ps_properties} then implies that $\p_n \pf \eta \in H^{s+1}(\Omega)$.  Note that $\Delta \pf \eta =0$ in $\Omega$.  Using this and the absolute continuity of Sobolev functions on lines (see, for instance, Theorem 11.45 in \cite{Leoni_2017}), we may then compute 
\begin{equation}
\p_n \pf\eta(x',0) - \p_n \pf \eta(x',-b) = \int_{-b}^0 \p_n^2 \pf \eta(x',t) dt = -\int_{-b}^0 \Delta' \pf \eta(x',t) dt = -\diverge' \int_{-b}^0 \nab' \pf \eta(x',t) dt.
\end{equation}
Thus, Cauchy-Schwarz, Fubini-Tonelli, and Parseval imply that
\begin{multline}
\snorm{\p_n \pf\eta(x',0) - \p_n \pf \eta(x',-b)}_{\dot{H}^{-1}}^2 = \int_{\R^{n-1}} \frac{1}{\abs{\xi}^2} \abs{ \xi \cdot \int_{-b}^0  \xi   \widehat{\pf \eta}(\xi,t) dt }^2 d\xi \\
\le  b^2 \int_{\R^{n-1}} \int_{-b}^0 \abs{\xi \widehat{\pf \eta}(\xi,t)}^2 dt \xi \le  b^2 \int_{\Omega} \abs{\nab \pf\eta}^2.
\end{multline}
The stated inequality follows from this and Theorems \ref{Ps_properties} and \ref{poisson_aniso_map}.
\end{proof}

\subsection{Composition} 

In this subsection we aim to study some composition operators.  We begin by introducing some notation that allows us to extend the flattening maps to full space.

\begin{dfn}\label{ef_diff_def}
Let $\chi \in C^\infty_c(\R)$ be such that $0 \le \chi \le 1$, $\chi =1$ on $[-2b,2b]$, and $\supp(\chi) \subset (-3b,3b)$.  Given $\eta \in \H^{\sigma+1/2}(\Sigma)$ define $\ef_\eta : \Gamma \times \R \to \Gamma \times \R$ via 
\begin{equation}
 \ef_\eta(x) = x + \chi(x_n)[\eta_l(x') + E(\pf \eta - \eta_l)(x)]\left(1+\frac{x_n}{b}\right) e_n,
\end{equation}
where $E : L^2(\Omega) \to L^2(\Gamma \times \R)$ is a Stein extension operator, $\pf$ is the Poisson extension as defined in Proposition \ref{poisson_aniso_map} and Theorem \ref{poisson_aniso_map}, and when $\Gamma = \R^{n-1}$ we take $\eta_l = \eta_{l,1} \in \bigcap_{t \ge 0} \H^{t}(\Sigma) \cap \bigcap_{k \in \N} C^k_0(\Sigma)$ in the notation of Theorem \ref{specialized_properties}, while when $\Gamma = \T^{n-1}$ we take $\eta_l =0$.  Note that Proposition \ref{poisson_sobolev_map} and Theorem \ref{poisson_aniso_map} show that $\pf \eta - \eta_l \in H^{\sigma+1}(\Omega)$,  and since the Stein extension restricts to a bounded map $E : H^{\sigma+1}(\Omega) \to H^{\sigma+1}(\Gamma \times \R)$ we have that $E(\pf \eta -\eta_l) \in H^{\sigma+1}(\Gamma \times \R)$.

\end{dfn}

Next we record some properties of these maps.

\begin{prop}\label{E_eta_properties}
Let $\sigma > n/2$, $\eta \in \H^{\sigma+1/2}(\Sigma)$, and define $\ef_\eta : \Gamma \times \R \to \Gamma \times \R$ as in Definition \ref{ef_diff_def}.   Then the following hold.
\begin{enumerate}
 \item  The map $\mathfrak{E}_\eta$ is Lipschitz and $C^1$, and $\norm{\nab \mathfrak{E}_\eta - I}_{C^0_b} \ls \norm{\eta}_{\H^{s+1/2}}$.

 \item If $V$ is a real finite dimensional inner-product space and  $0 \le r \le \sigma$, then 
\begin{equation}
\sup_{1\le j,k \le n} \norm{\p_j \mathfrak{E}_\eta \cdot e_k f  }_{H^r} \ls(1 + \norm{\eta}_{\H^{\sigma+1/2}}) \norm{f}_{H^r}
\end{equation}
and 
\begin{equation}
\sup_{1\le j,k \le n} \norm{(\p_j \mathfrak{E}_\eta \cdot e_k - \p_j \mathfrak{E}_\zeta \cdot e_k ) f  }_{H^r} \ls \norm{\eta-\zeta}_{\H^{\sigma+1/2}} \norm{f}_{H^r} 
\end{equation}
for every $\eta,\zeta \in \H^{\sigma+1/2}(\Sigma)$ and $f \in H^r(\Gamma \times \R;V)$. 
 
\item There exists $0 < \delta_\ast <1$ such that if $\norm{\eta}_{\H^{\sigma+1/2}} < \delta_\ast$, then $\mathfrak{E}_\eta$ is a bi-Lipschitz homeomorphism and a $C^1$ diffeomorphism, and we have the estimate $\norm{\nab \mathfrak{E}_\eta - I}_{C^0_b} < 1/2$.
\end{enumerate}

\end{prop}
\begin{proof}
First note that $\sigma+1 >n/2 +1$, so Proposition \ref{poisson_sobolev_map},  Theorem \ref{poisson_aniso_map} and standard Sobolev embeddings show that $E(\pf \eta - \eta_l) \in C^1_b(\Gamma \times \R)$.  On the other hand, $\eta_l \in \bigcap_{t \ge 0} \H^t(\Sigma)$, so Theorem \ref{specialized_properties} shows that $\eta_l \in C^1_b(\Sigma)$.  These observations and their associated bounds then imply the first item.  Next we write $\ef_\eta = I + \omega e_n$ so that $\nab \ef_\eta = I + e_n \otimes \nab \omega$.  To prove the second item it suffices to show that $\norm{\p_j \omega f}_{H^r} \ls \norm{\eta}_{\H^{\sigma+1/2}} \norm{f}_{H^r}$ for $0 \le r \le \sigma$ and $1 \le j \le n$.  To establish this we observe that on the one hand, thanks to Theorem \ref{specialized_properties}, $\chi \eta_l \in \bigcap_{k \in \N} C^k_0(\Gamma \times \R)$, and on the other $E(\pf \eta - \eta_l) \in H^{\sigma+1}(\Gamma \times \R)$.  Thus, $\p_j \omega$ consists of linear combinations of terms in $\bigcap_{k \in \N} C^k_0(\Gamma \times \R)$ and in $H^{\sigma}(\R^n)$, and so the sufficient bound follows from standard  Sobolev multiplier results (see, for instance, Lemma A.8 in \cite{LeoniTice}).   

To prove the third item we note that if $\omega$ has Lipschitz constant less than unity, then $\omega e_n$ is contractive on $\R^n$, and so the Banach fixed point theorem implies that $\mathfrak{E}_\eta$ is a bi-Lipschitz homeomorphism.  To control the Lipschitz constant of $\omega$ we use the supercritical Sobolev embeddings as above to verify that this constant is less than unity  provided that $\norm{\eta}_{\H^{\sigma+1/2}} < \delta_\ast$ for some sufficiently small universal constant $\delta_\ast \in (0,1)$.
\end{proof}

The next result studies the smoothness properties of composition with the maps from Definition \ref{ef_diff_def}.

\begin{thm}\label{Lambda_diff}
Let $n/2 < \sigma \in \N$,  $0 < \delta_\ast <1$ be as in the third item of Proposition \ref{E_eta_properties}, and $V$ be a real finite dimensional inner-product space.  Let $r \in \N$ satisfy $0 \le r \le \sigma +1$ and let $k \in \{0,1\}$.  Consider the map $\Lambda : H^{r+k}(\Gamma \times \R;V) \times B_{\H^{\sigma+1/2}(\Sigma)}(0,\delta_\ast) \to H^{r}(\Gamma \times \R;V)$ given by $\Lambda(f,\eta) = f\circ \mathfrak{E}_\eta,$ where $\mathfrak{E}_\eta : \Gamma \times \R \to \Gamma \times \R$ is as defined in Definition \ref{ef_diff_def}.  Then 
$\Lambda$ is well-defined and $C^k$, and if $k =1$ then $D\Lambda(f,\eta)(g,\zeta) = \chi \tilde{b} (\eta_l + E(\pf \eta - \eta_l) (\p_n f \circ \mathfrak{E}_\eta)\zeta + g \circ \mathfrak{E}_\eta,$ where $\tilde{b}(x) = (1+x_n/b)$.
\end{thm}
\begin{proof}
With Proposition \ref{E_eta_properties} established, the result follows from minor and evident modifications of the argument used to prove Theorem 1.1 in \cite{Inci_etal}  (see also Theorem 5.20 in \cite{LeoniTice}).
\end{proof}

Finally, as a byproduct of this theorem we obtain smoothness properties associated to composition with the flattening maps $\ff_\eta$.

\begin{cor}\label{Feta_comp}
Let $n/2 < \sigma \in \N$,  $0 < \delta_\ast <1$ be as in the third item of Proposition \ref{E_eta_properties}, and $V$ be a real finite dimensional inner-product space.   Let $r \in \N$ satisfy $0 \le r \le \sigma +1$.  For $\eta \in \H^{\sigma+1/2}(\Sigma)$ define $\ff_\eta : \Omega \to \Omega_\eta$ via \eqref{flattening_def}.  Then the following hold.
\begin{enumerate}
 \item The map  $\Lambda_\Omega : H^{r+1}(\Gamma \times \R;V) \times B_{\H^{\sigma+1/2}(\Sigma)}(0,\delta_\ast) \to H^{r}(\Omega;V)$ given by $\Lambda(f,\eta) = f\circ \ff_\eta$ is well-defined and $C^1$ with
$D\Lambda_\Omega (f,\eta)(g,\zeta) = \tilde{b} \pf \eta  (\p_n f \circ \ff_\eta)\zeta + g \circ \ff_\eta,$ 
where $\tilde{b}(x) = (1+x_n/b)$.

 \item Assume $r \ge 1$.  Then the map $\sf_\Sigma : H^{r+1}(\Gamma \times \R;V) \times B_{\H^{\sigma+1/2}(\Sigma)}(0,\delta_\ast) \to H^{r-1/2}(\Sigma;V)$ given by $\sf_\Sigma(f,\eta) = f \circ \ff_\eta \vert_{\Sigma}$ is well-defined and $C^1$ with $D\sf_\Sigma (f,\eta)(g,\zeta) = \eta  (\p_n f \circ \ff_\eta)\zeta \vert_{\Sigma} + g \circ \ff_\eta \vert_\Sigma$.
\end{enumerate}
\end{cor}
\begin{proof}
The first item follows from Theorem \ref{Lambda_diff} and the observation that $\Lambda_\Omega(f,\eta) = R_\Omega \Lambda(f,\eta)$, where $R_\Omega : H^{r}(\Gamma \times \R;V) \to H^{r}(\Omega;V)$ is the bounded linear map given by restriction to $\Omega$.  This identity follows directly from the fact that, by construction, $\mathfrak{E}_\eta = \ff_\eta$ in $\Omega$.  The second item follows by composing the first item with the bounded linear trace map.
\end{proof}

\subsection{Littlewood-Paley analysis for the anisotropic Sobolev space $\cH^s$}\label{Appendix:LlittlewoodPaley}

In this subsection we develop some Littlewood-Paley theory for the anisotropic spaces.

\begin{dfn}
Let $\chi\in C^\infty(\R^d)$ be a radial function such that $\chi(\xi)=1$ for $|\xi|\le \hal$, $\chi(\xi)=0$ for $|\xi|\ge 1$.  Set 
\begin{equation}
\varphi(\xi)=\chi(\xi)-\chi(2\xi), \quad
\chi_j(\xi)=\chi(2^{-j}\xi) \text{ for }  j\in \Zz, \quad
\varphi_0=\chi, \text{ and } \varphi_j(\xi)=\varphi(2^{-j}\xi) \text{ for } j\ge 1.
\end{equation}
The Littlewood-Paley dyadic block $ \Delta_j$ is defined by the Fourier multiplier  
\begin{equation}
 \Delta_j=\varphi_j(D_x)\quad \text{for } j\ge 0,\quad \Delta_j=0\quad\text{for } j\le -1.
 \end{equation}
  The low-frequency cut-off operator $ S_j$ is defined by
\begin{equation}
S_j=\chi_j(D)=\sum_{k=0}^j\Delta_k \text{ for } j\ge 0.
\end{equation}
\end{dfn}
The above Fourier multipliers can act on functions (distributions) defined on $\R^d$ or $\T^d$, and the Fourier transform is defined accordingly. In particular, for $u:\T^d\to \R$ we have
\begin{equation}\label{Delta0=mean}
\Delta_0(D)u=\frac{1}{(2\pi)^d}\wh{u}(0)=\frac{1}{(2\pi)^d}\int_{\T^d} u.
\end{equation}
Since $\sum_{j=0}^\infty \varphi_j(\xi)=1$ for all $\xi\in \R^d$,  we have that $\sum_{j=0}^\infty \Delta_j=\text{Id}$. Moreover, we have $\supp \varphi_j\subset\{ 2^{j-2}<|\xi|<2^j\}$ for $j\ge 1$ and $\chi\varphi_j=0$ for  $j\ge 2$.

Bony's decomposition for product of functions  is
\begin{equation}\label{Bony}
fg=T_fg+T_gf+R(f, g),
\end{equation}
where
\begin{equation}
T_fg=\sum_{j\ge 3} S_{j-3}f\Delta_j g \text{ and } R(f, g)=\sum_{j, k\ge 0, |j-k|\le 2}\Delta_j f\Delta_kg.
\end{equation}
We note that $\supp\wh{S_{j-3}f\Delta_j g}\subset\{2^{j-3}<|\xi|<2^{j+1} \}$ for $j\ge 1$.

 We recall the following result from \cite{BCD}.

 \begin{lem}[\protect{\cite[Lemma~2.2]{BCD}}]\label{lemm:Bernstein}
 Let $\mathcal{C}$ be an annulus in $\R^d$,  $m\in \R$, and $k=2[1+\frac{d}{2}]$, where $[r]$ denotes the integer part of $r$. Let $\sigma$ be a k-times differentiable function on $\R^d\setminus\{0\}$ such that for all $\alpha\in \R^d$ with $|\alpha|\le k$, there exists a constant $C_\alpha$ such that 
 \begin{equation}\label{homosymbol}
 |\p^\alpha \sigma(\xi)|\le C_\alpha|\xi|^{m-|\alpha|} \text{ for all }   \xi\in \R^d\setminus\{ 0\}.
 \end{equation}
 There exists a constant $C$, depending only on the constants $C_\alpha$, such that for any $p\in [1, \infty]$ and any constant $\ld>0$, we have, for any  function $u\in L^p(M^d)$, $M\in \{\R, \T\}$, with Fourier transform supported in $\lambda \mathcal{C}$, 
 \begin{equation}
 \| \sigma(D)u\|_{L^p}\le C\lambda^m\|  u\|_{L^p}.
 \end{equation}
 \end{lem}

Next we recall the definition of the Chemin-Lerner norm.
\begin{dfn}\label{defi:Chermin-Lerner}
Let $M$ be either $\R$ or $\T$. For $I\subset \R$ and $s\in \R$, the Chemin-Lerner norm is defined by
\begin{equation}
\| u\|^2_{\wt L^q(I; H^s(M^d))}=\sum_{j=0}^\infty 2^{2sj}\|  \Delta_j u\|^2_{L^q(I; L^2(M^d))}.
\end{equation}
When the low-frequency part is removed, we denote 
\begin{equation}
\| u\|^2_{ H^s_\sharp(M^d)}=\sum_{j= 1}^\infty 2^{2sj}\|  \Delta_j u\|^2_{L^2(M^d)}
\text{ and }
\| u\|^2_{\wt L^q(I; H^s_\sharp(M^d))}=\sum_{j= 1}^\infty 2^{2sj}\|  \Delta_j u\|^2_{L^q(I; L^2(M^d))}.
\end{equation}
\end{dfn} 

It what follows, unless otherwise specified, when the set  $M$ is omitted in function space notation, it can be either $\R$ or $\T$.  We recall another result from \cite{BCD}, this time about products.

\begin{prop}[\protect{\cite[Corollary 2.54]{BCD}}]
For $I\subset \R$, $q\in [1, \infty]$ and $s>0$, there exists $C=C(d, s)$ such that
\begin{align}\label{pr}
&\Vert f g \Vert_{H^s}\le C \Vert f\Vert_{L^\infty}\Vert g\Vert_{H^s}+C\Vert g\Vert_{L^\infty}\Vert f\Vert_{H^s},\\
\label{pr:CL}
&\Vert fg \Vert_{\wt L^q(I; H^s)}\le C \Vert f\Vert_{L^\infty(I; L^\infty)}\Vert g\Vert_{\wt L^q(I; H^s)}+C\Vert g\Vert_{L^\infty(I; L^\infty)}\Vert f\Vert_{\wt L^q(I; H^s)}
\end{align}
provided that the right-hand sides are finite. 
\end{prop}

Next we study the boundedness of some key operators in the Chemin-Lerner norm.

\begin{prop}\label{prop:fourierop}
The following hold.
\begin{enumerate}
 \item There exists an absolute constant $C$ such that for all $1\le p\le \infty$, $\sigma\in \R$ and $u\in H^\sigma(\R^d)$, we have
\begin{equation}\label{estop1}
\left\| \frac{\cosh((z+b)|D|)}{\cosh(b|D|)}u\right\|_{\wL^p_z([-b, 0]; H^{\sigma+\frac{1}{p}})}\le \max\{2b^{\frac{1}{p}}, C\}\| u\|_{H^\sigma}. 
\end{equation}
\item There exists an absolute constant $C$ such that for all $1\le q_2\le q_1\le \infty$, $\sigma\in \R$ and $f\in \wL^{p_2}_z([-b, 0];  H_x^{\sigma-1+\frac{1}{p_2}}) $, we have
\begin{equation}\label{estop2}
\left\| \int_{-b}^z\frac{\cosh((z'+b)|D|)}{\cosh((z+b)|D|)}f(x, z')dz'\right\|_{\wL^{q_1}_z([-b, 0]; H^{\sigma+\frac{1}{q_1}})}\le \max\{b^\frac{q_1+q_2'}{q_1q_2'}, C\}\| f\|_{\wL^{q_2}_z([-b, 0];  H^{\sigma-1+\frac{1}{q_2}})},
\end{equation}
where $\frac{1}{q_2}+\frac{1}{q_2'}=1$. In addition, for any $z\in [-b, 0]$, we have
\begin{equation}\label{estop3}
\left\| \int_{-b}^z\frac{\cosh((z'+b)|D|)}{\cosh((z+b)|D|)}f(x, z')dz'\right\|_{H^\sigma}\le \max\{b^\frac{q_1+q_2'}{q_1q_2'}, C\}\| f\|_{\wL^{q_2}_z([-b, 0];  H^{\sigma-1+\frac{1}{p_2}})},
\end{equation}
\end{enumerate} 
\end{prop} 
\begin{proof}
For all $-b\le z_1\le z_2$, we have $0\le z_2-z_1\le z_2+b$ and hence  
\begin{equation}\label{est:symbol}
1\le \frac{\cosh((z_1+b)c)}{\cosh((z_2+b)c)}e^{(z_2-z_1)c}=\frac{e^{2(z_2+b)c}+e^{2(z_2-z_1)c}}{e^{2(z_2+b)c}+1}\le 2.
\end{equation}
for all $c\ge 0$. 

To prove the first item we note that \eqref{est:symbol}  implies 
\begin{equation}
e^{zc} \le \frac{\cosh((z+b)c)}{\cosh(bc)}\le 2e^{zc}
\end{equation}
for all $z\in [-b, 0]$. Consequently, for $j\ge 1$ and $u\in \dot H^\sigma$, we have 
\begin{multline}
 \left\| \Delta_j\frac{\cosh((z+b)|D|)}{\cosh(b|D|)}u\right\|_{L^2_x} =\left\| \frac{\cosh((z+b)|D|)}{\cosh(b|D|)}\Delta_ju\right\|_{L^2_x}\\
 \le \left(\int_{\R^d}4e^{2z|\xi|}|\wh{\Delta_j u}(\xi)|^2d\xi\right)^\hal
 \le 2e^{z2^{j-2}}\|\Delta_j u\|_{L^2}
\end{multline}
since $|\xi|\ge 2^{j-2}$ on the support of $\wh{\Delta_j}u(\xi)$. It follows that
\begin{equation}\label{estop1:high}
 \left\| \Delta_j\frac{\cosh((z+b)|D|)}{\cosh(b|D|)}u\right\|_{L^p_z([-b, 0]; L^2)}\le C2^{-\frac{j}{p}}\|\Delta_j u\|_{L^2},
 \end{equation}
 where $C$ is an absolute constant. On the other hand, the low frequency part can be bounded as
 \begin{equation}\label{estop1:low}
  \left\| \Delta_0\frac{\cosh((z+b)|D|)}{\cosh(b|D|)}u\right\|_{L^p([-b, 0]; L^2)}\le   2\|\| \Delta_0u\|_{L^2_x}\|_{L^p_z([-b, 0])}\le 2b^{\frac{1}{p}}\| \Delta_0u\|_{L^2}.
  \end{equation}
  Combining \eqref{estop1:high} and \eqref{estop1:low} yields 
  \begin{equation}\begin{aligned}
  \left\| \frac{\cosh((z+b)|D|)}{\cosh(b|D|)}u\right\|_{\wL^p_z([-b, 0]; H^{\sigma+\frac{1}{p}})}^2&\le (2b^{\frac{1}{p}})^2\| \Delta_0u\|_{L^2}^2+\sum_{j=1}^\infty C^22^{2j(\sigma+\frac{1}{p}-\frac{1}{p})}\|\Delta_j u\|_{L^2}^2\\
  &\le \max\{2b^{\frac{1}{p}}, C\}^2\| u\|^2_{H^\sigma}.
  \end{aligned}\end{equation}
 This completes the proof of the first item.

 We now turn to the proof of the second item.  To prove \eqref{estop2}, we set
 \begin{equation}
 g(x, z)=\int_{-b}^z\frac{\cosh((z'+b)|D|)}{\cosh((z+b)|D|)}f(x, z')dz'.
 \end{equation}
  For $z\in [-b, 0]$ and $j\ge 1$,  using  \eqref{est:symbol} we estimate
 \begin{multline}
 \left\| \Delta_jg(\cdot, z)\right\|_{L^2} = \left\|\int_{-b}^z\frac{\cosh((z'+b)|D|)}{\cosh((z+b)|D|)} \Delta_jf(x, z')dz'\right\|_{L^2_x} 
\le\int_{-b}^z \left\|\frac{\cosh((z'+b)|D|)}{\cosh((z+b)|D|)} \Delta_jf(\cdot, z')\right\|_{L^2_x}dz' \\
 \le \int_{-b}^z \left(\int_{\R^d}e^{-2(z-z')|\xi|} |\varphi(2^{-j}\xi)\wh{f}(\xi, z')|^2d\xi\right)^\hal dz' 
  \le \int_{-b}^z \left(\int_{\R^d}e^{-(z-z')2^{j-1}} |\wh{\Delta_j f}(\xi, z')|^2d\xi\right)^\hal dz' \\
    \le \int_{-b}^z  e^{-(z-z')2^{j-2}}\|\Delta_j f(\cdot, z')\|_{L^2} dz'.
\end{multline}
 Applying Young's inequality in $z$ we deduce 
 \begin{equation}\label{estop2:high}
  \begin{aligned}
\| \Delta_jg\|_{L^{q_1}_z([-b, 0]; L^2)}& \le \| e^{-z 2^{j-2}}\|_{L^q_z(\R_+)}\| \Delta_j f\|_{L^{q_2}_z([-b, 0]; L^2)}\\
&=\frac{1}{q^\frac{1}{q}}2^{-\frac{j-2}{q}}\| \Delta_j f\|_{L^{q_2}_z([-b, 0]; L^2)}\le C2^{-\frac{j}{q}}\| \Delta_j f\|_{L^{q_2}_z([-b, 0]; L^2)},
 \end{aligned}
 \end{equation}
 where $\frac{1}{q}=1+\frac{1}{q_1}-\frac{1}{q_2}$ and $C$ is an absolute constant. On the other hand, it is readily seen that 
 \begin{equation}
  \left\| \Delta_0g(\cdot, z)\right\|_{L^2} \le 2\int_{-b}^z  \|\Delta_0 f(\cdot, z')\|_{L^2} dz' \le 2b^\frac{1}{q_2'} \|\Delta_0 f\|_{L^{q_2}_z([-b, 0]; L^2)},
  \end{equation}
  and hence
   \begin{equation}\label{estop2:low}
  \left\| \Delta_0g\right\|_{L^{q_1}_z([-b, 0]; L^2)}  \le 2b^\frac{q_1+q_2'}{q_1q_2'} \|\Delta_0 f\|_{L^{q_2}_z([-b, 0]; L^2)}.
  \end{equation}
A combination of \eqref{estop2:high} and \eqref{estop2:low} leads to \eqref{estop2}. 

Finally, the proof of \eqref{estop3} is similar to the case $q_1=\infty$ of \eqref{estop2}. 
\end{proof}

Next we consider some more product estimates.

\begin{prop}\label{prop:psCL}
Let $s>0$, $ p\in [1,  \infty]$, and $I\subset \R$. Then, there exists $C=C(d, s)$ such that the estimate 
 \begin{equation}\label{ppestcH}
 \begin{aligned}
\| fg\|_{\wL^p(I; H^s(\R^d))}&\le C \left(\| f\|_{L^\infty(I; L^\infty(\R^d))} +\| \chi \wh{f}\|_{L^\infty(I; L^1(\R^d))}\right)\| g\|_{\wL^p(I; H^s(\R^d))}\\
&\quad+C\| g\|_{L^\infty(I; L^\infty(\R^d))} \| f\|_{\wL^p(I; H^s_\sharp(\R^d))}
\end{aligned}
\end{equation}
holds provided that the right-hand side is finite. Consequently,  for $s>0$ and $s_0>\frac{d}{2}$, there exists $C=C(d, s, s_0)$ such that
\begin{equation}\label{ppestcH:2}
\| fg\|_{\wL^p(I; H^s)}\le C\| f\|_{L^\infty(I; \cH^{s_0}(\R^d))}\| g\|_{\wL^p(I; H^s)}+C\| g\|_{L^\infty(I; L^\infty)} \| f\|_{\wL^p(I; H^s_\sharp)}.
\end{equation}
\end{prop}
\begin{proof}
We first note that for $M=\T$, \eqref{ppestcH:2} is a consequence of \eqref{pr:CL} and the continuous embedding $H^{s_0}(\T^d)\subset L^\infty(\T^d)$ for $s_0>\frac{d}{2}$. 

To prove \eqref{ppestcH} and \eqref{ppestcH:2} for $M=\R$, we shall consider functions $f(x, z)$ and $g(x, z)$ defined on $\R^d\times I$. For fixed $z\in I$, we use Bony's decomposition \eqref{Bony}: $fg=T_fg+T_gf+R(f, g)$, where $T_fg=\sum_{j\ge 3}S_{j-3}f\Delta_j g$. For $j\ge 3$ we have $\supp\wh{S_{j-3}f\Delta_j g}\subset\{2^{j-3}<|\xi|<2^{j+1}\}$ and hence $\Delta_k (S_{j-3}f\Delta_j g)=0$  for all $k\ge 0$ satisfying $|j-k|\ge 3$. Thus, for $k\ge 0$ using Bernstein's inequality we obtain
\begin{multline}
2^{sk}\| \Delta_k T_fg\|_{L^2_x} 
=2^{sk}\| \sum_{j\ge 3, |j-k|\le 2} \Delta_k(S_{j-3}f\Delta_j g)\|_{L^2_x}
\le C\sum_{j\ge 3, |j-k|\le 2} 2^{sj}\|S_{j-3}f\Delta_j g\|_{L^2_x}\\
\le C\| f\|_{L^\infty_x} \sum_{j\ge 3, |j-k|\le 2}2^{sj}\| \Delta_j g\|_{L^2_x},
\end{multline}
where $C=C(d, s)$. Since $f\in \wL^p(I; H^s_\sharp)$, we have $\Delta_j f\in L^2_x$ a.e. $z\in I$ for $j\ge 3$. Consequently, the preceding estimate for $T_fg$ also holds for $T_gf$; that is,
\begin{equation}
2^{sk}\| \Delta_k T_gf\|_{L^2_x}\le C\| g\|_{L^\infty_x} \sum_{j\ge 3, |j-k|\le 2}2^{sj}\| \Delta_j f\|_{L^2_x}.
\end{equation} 
It follows that
\begin{multline}\label{est:Tfg}
\| T_fg\|_{\wt L^p(I; H^s)}^2 
\le C\| f\|_{L^\infty(I; L^\infty)} ^2 \sum_{k=0}^\infty \Big(\sum_{j\ge 3, |j-k|\le 2}2^{sk}\| \Delta_j g\|_{L^p(I; L^2)}\Big)^2\\
\le C\| f\|_{L^\infty(I; L^\infty)} ^2 \sum_{j=3}^\infty 2^{2sj}\| \Delta_j g\|_{L^p(I; L^2)}^2 
\le C\| f\|_{L^\infty(I; L^\infty)}^2\| g\|_{\wt L^p(I; H_\sharp^s)}^2,
\end{multline}
and similarly we have
\begin{equation}\label{est:Tgf}
\| T_gf\|_{\wt L^p(I; H^s)}\le C\| g\|_{L^\infty(I; L^\infty)}\| f\|_{\wt L^p(I; H_\sharp^s)}.
\end{equation}
As for the remainder $R(f, g)=\sum_{j\ge 0} \sum_{|\nu|\le 2} \Delta_j f\Delta_{j+\nu}g$, we note that $\supp\wh{ \Delta_j f\Delta_{j+\nu}g}\subset\{|\xi|<2^{j+3}\}$. Thus $\Delta_k(\Delta_j f\Delta_{j+\nu}g)=0$ for $k\ge j+5$ and 
\begin{equation}
\| \Delta_k \sum_{|\nu|\le 2} \Delta_j f\Delta_{j+\nu}g\|_{L^2_x}\le \sum_{j\ge k-4}\sum_{|\nu|\le 2} \| \Delta_k (\Delta_j f\Delta_{j+\nu}g)\|_{L^2_x}, 
\end{equation}
where
\begin{equation}
\| \Delta_k (\Delta_j f\Delta_{j+\nu}g)\|_{L^2_x} \ls 
\| \Delta_j f\|_{L^2_x}\| g\|_{L^\infty_x }\quad\text{if } j\ge 1
\end{equation}
and 
\begin{equation}
\| \Delta_k (\Delta_j f\Delta_{j+\nu}g)\|_{L^2_x} \ls  \| \Delta_0 f\|_{L^\infty_x}\| \Delta_{\nu} g\|_{L^2_x}\ls \| \chi \wh f\|_{L^1}\| \Delta_{\nu} g\|_{L^2_x}
\quad\text{if } j=0.
\end{equation}
It follows that
\begin{equation}\begin{aligned}
2^{ks}\| \Delta_k \sum_{|\nu|\le 2} \Delta_j f\Delta_{j+\nu}g\|_{L^p(I;  L^2)}\ls
\begin{cases} 2^{ks}\sum_{j\ge 1, j\ge k-4}\| \Delta_j f\|_{L^p(I; L^2)}\| g\|_{L^\infty(I; L^\infty)}\quad\text{if } k\ge 5,\\
2^{ks}\| \chi \wh f\|_{L^\infty(I; L^1)}\| \Delta_{\nu} g\|_{L^p(I; L^2)}\quad\text{if } k\le 4
\end{cases}\\
\ls \begin{cases} \sum_{j\ge 1, j\ge k-4}2^{js}\| \Delta_j f\|_{L^p(I; L^2)}\| g\|_{L^\infty(I; L^\infty)}2^{(k-j)s}\quad\text{if } k\ge 5,\\
\| \chi \wh f\|_{L^\infty(I; L^1)}\|  g\|_{\wt L^p(I; L^2)}\quad\text{if } k\le 4.
\end{cases}
\end{aligned}\end{equation}
By Young's inequality for series, we deduce
\begin{equation}
\Big\| 2^{ks}\| \Delta_k \sum_{|\nu|\le 2} \Delta_j f\Delta_{j+\nu}g\|_{L^p(I;  L^2)}\Big\|_{\ell^2(\{k\ge 5\})}\ls \| g\|_{L^\infty(I; L^\infty)}\| f\|_{\wt L^p(I; H_\sharp^s)}.
\end{equation}
We thus obtain
\begin{equation}\label{estRfg}
\| R(f, g)\|_{\wt L^p(I; H^s)}\ls \| \chi \wh f\|_{L^\infty(I; L^1)}\|  g\|_{\wt L^p(I; H^s)}+ \| g\|_{L^\infty(I; L^\infty)}\| f\|_{\wt L^p(I; H_\sharp^s)}.
\end{equation}
Combining \eqref{est:Tfg}, \eqref{est:Tgf} and \eqref{estRfg} we obtain \eqref{ppestcH}.
 Finally, \eqref{ppestcH:2} follows from \eqref{ppestcH} and \eqref{L1cH}. 
\end{proof}

Our next result records some estimates for nonlinear maps of the form $(f,g) \mapsto g(1+f)^{-1}$.

\begin{prop}\label{prop:nonlcH} Let $I\subset \R$, $p\in [1, \infty]$, $s>0$, and $s_0>\frac{d}{2}$. There exists a positive constant  $C=C(d, s, s_0)$ such that if $\| f\|_{L^\infty(I; \cH^{s_0})}<\frac{1}{2C}$ then 
\begin{equation}\label{nonl:cH}
\begin{aligned}
\norm{ \frac{g}{1+f}}_{\wL^p(I; H^s)}&\le \|g\|_{\wL^p(I; H^s)}+C\| f\|_{L^\infty(I; \cH^{s_0})}\| g\|_{\wL^p(I; H^s)}+\| g\|_{L^\infty(I; L^\infty)} \| f\|_{\wL^p(I; H^s_\sharp)}.
\end{aligned}
\end{equation}
\end{prop}
\begin{proof}
By virtue of \eqref{L1cH} we have $\| f\|_{L^\infty(I; L^\infty)}\le C_1\| f\|_{L^\infty(I; \cH^{s_0})}$, $C_1=C_1(d, s, s_0)$,  and hence $|f|\le \hal$ $a.e.$ if $\| f\|_{L^\infty(I; \cH^{s_0})}\le \frac{1}{2C_1}$. Then  the expansion 
\begin{equation}
\frac{g}{1+f}=\sum_{j\ge 0}(-1)^jgf^j
\end{equation}
holds $a.e$ on $\R^d$.  We claim that with $C_2=\max\{C_1, C\}$, where $C$ is given in \eqref{ppestcH:2}, we have 
\begin{equation}\label{product:j}
\| gf^j\|_{\wL^p(I; H^s)}\le C_2(C_2\| f\|_{L^\infty(I; \cH^{s_0})})^{j-1}\left\{ \| f\|_{L^\infty(I; \cH^{s_0})}\| g\|_{\wL^p(I; H^s)}+j\| g\|_{L^\infty(I; L^\infty)} \| f\|_{\wL^p(I; H^s_\sharp)}\right\}
\end{equation}
for all $j\ge 1$.  Indeed, the case $j=1$ follows at once from \eqref{ppestcH:2}. Assume that \eqref{product:j} holds for some $j\ge 1$. Applying \eqref{ppestcH:2} once again, we deduce
\begin{equation}\begin{aligned}
\| gf^{j+1}\|_{\wL^p(I; H^s)}&=\| f(gf^j)\|_{\wL^p(I; H^s)} 
\le C_2\| f\|_{L^\infty(I; \cH^{s_0})}\| gf^j\|_{\wL^p(I; H^s)}+C_2\| gf^j\|_{L^\infty(I; L^\infty)} \| f\|_{\wL^p(I; H^s_\sharp)}\\
&\le  C_2(C_2\| f\|_{L^\infty(I; \cH^{s_0})})^j\left\{\| f\|_{L^\infty(I; \cH^{s_0})}\| g\|_{\wL^p(I; H^s)}+j\| g\|_{L^\infty(I; L^\infty)} \| f\|_{\wL^p(I; H^s_\sharp)}\right\}\\
&\quad+C_2\| gf^j\|_{L^\infty(I; L^\infty)} \| f\|_{\wL^p(I; H^s_\sharp)}.
\end{aligned}\end{equation}
Combining this with the estimate
\begin{equation}
\| gf^j\|_{L^\infty(I; L^\infty)} \le \| g\|_{L^\infty(I; L^\infty)} \| f\|^j_{L^\infty(I; L^\infty)}\le  \| g\|_{L^\infty(I; L^\infty)}(C_1 \| f\|_{L^\infty(I; \cH^{s_0})})^j
\end{equation}
we obtain \eqref{product:j} for $j+1$. 

Finally, for $\| f\|_{L^\infty(I; \cH^{s_0})}\le \frac{1}{2C_2}$ we can sum \eqref{product:j} over $j\ge 1$ to  obtain 
\begin{multline}\label{nonl:cH:0}
\norm{ \frac{g}{1+f} }_{\wL^p(I; H^s)} \le \|g\|_{\wL^p(I; H^s)}
+\frac{C}{1-C\| f\|_{L^\infty(I; \cH^{s_0})}}\| f\|_{L^\infty(I; \cH^{s_0}(\R^d))}\| g\|_{\wL^p(I; H^s)}
\\
+\frac{C}{(1-C\| f\|_{L^\infty(I; \cH^{s_0})})^2}\| g\|_{L^\infty(I; L^\infty)} \| f\|_{\wL^p(I; H^s_\sharp)}.
\end{multline}
This implies \eqref{nonl:cH}.
\end{proof}

\section*{Acknowledgements}

H. Nguyen was supported by an NSF Grant (DMS \#2205710).  I. Tice was supported by an NSF Grant (DMS \#2204912). We thank B. Pausader for discussions on the finite-depth Dirichlet-Neumann operator. 

\section*{Data Availability Statement}

 Data sharing is not applicable to this article as no
datasets were generated or analyzed during the current study.


\bibliographystyle{abbrv}

\begin{thebibliography}{10}
\small


\bibitem{BCD}
H. Bahouri, J.-Y. Chemin,  R. Danchin.  {\em Fourier analysis and nonlinear partial differential equations},
  volume 343 of {\em Grundlehren der Mathematischen Wissenschaften [Fundamental
  Principles of Mathematical Sciences]}.  Springer, Heidelberg, 2011.
 
\bibitem{bergh_lof} J. Bergh, J. L\"{o}fstr\"{o}m. \emph{Interpolation spaces. An introduction.} Grundlehren der Mathematischen Wissenschaften, No. 223. Springer-Verlag, Berlin-New York, 1976. 

\bibitem{CCG} A. C\'ordoba, D. A. C\'ordoba, F. Gancedo. Interface evolution: the Hele-Shaw and Muskat problems.
\emph{Ann. of Math.} \textbf{173} (2011), no. 1, 477--542.

\bibitem{CGBS} C.H. Cheng, R. Granero-Belinchón and S. Shkoller. Well-posedness of the Muskat problem with $H^2$ initial data. \emph{Adv. Math.} \textbf{286} (2016), 32--104.

\bibitem{GGJPS} F. Gancedo, E. Garcia-Juarez, N. Patel, and R. M. Strain. On the Muskat problem with viscosity jump: global in time results. {\it Adv. Math.} \textbf{345} (2019), 552--597.

\bibitem{Inci_etal}  H. Inci, T. Kappeler, P. Topalov. On the regularity of the composition of diffeomorphisms. \emph{Mem. Amer. Math. Soc.} \textbf{226} (2013), no. 1062.
 

\bibitem{koganemaru_tice} J. Koganemaru. I. Tice.  Traveling wave solutions to the inclined or periodic free boundary incompressible Navier-Stokes equations.  Preprint (2022), arXiv:2207.07702, 45 pp.
 



\bibitem{Leoni_2017} G. Leoni.  \emph{A first course in Sobolev spaces.}  Second edition.  American Mathematical Society, Providence, RI, 2017.

\bibitem{LeoniTice} G. Leoni, I. Tice. Traveling wave solutions to the free boundary incompressible Navier-Stokes equations. Comm. Pure Appl. Math. (2022). https://doi.org/10.1002/cpa.22084 
 
 
 \bibitem{LionMage} J. L Lions and E. Magenes. \emph{Non-homogeneous boundary value problems and applications. Vol. I.} Translated from the French by P. Kenneth. Die Grundlehren der mathematischen Wissenschaften, Band 181. Springer-Verlag, New York-Heidelberg, 1972.

\bibitem{Nguyen2022}
H. Q. Nguyen.  Global solutions for the Muskat problem in the scaling invariant Besov space $\dot B^1_{\infty, 1}$. {\em Adv. Math.} \textbf{394} (2022), Paper No. 108122. 

\bibitem{NguyenPausader} H. Q. Nguyen, and B. Pausader. A paradifferential approach for well-posedness of the Muskat problem. {\it Arch. Ration. Mech. Anal} \textbf{237} (2020), no. 1, 35--100.


\bibitem{SieCafHow} M. Siegel, R. Caflisch and S. Howison, Global existence, singular solutions, and ill-posedness for the Muskat problem. {\it Commun. Pure Appl. Math.} \textbf{57} (2004), 1374--1411.

\bibitem{stevenson_tice} N. Stevenson, I. Tice. Traveling wave solutions to the multilayer free boundary incompressible Navier-Stokes equations. \emph{SIAM J. Math. Anal.} \textbf{53} (2021), no. 6, 6370--6423. 


\bibitem{triebel} H. Triebel. \emph{Interpolation theory, function spaces, differential operators.  Second edition.}  Johann Ambrosius Barth, Heidelberg, 1995.


\end{thebibliography}

\end{document}